\documentclass{article}
\usepackage{graphicx} 
\usepackage[english]{babel}
\usepackage[letterpaper,top=2cm,bottom=2cm,left=3cm,right=3cm,marginparwidth=1.75cm]{geometry}
\usepackage{caption}
\usepackage{amsmath} 
\usepackage{amsthm} 
\usepackage{mathtools} 
\usepackage[thinc]{esdiff} 
\usepackage{amssymb} 
\usepackage{physics} 
\usepackage{mathrsfs} 
\usepackage{subcaption}
\usepackage{wrapfig}
\mathtoolsset{showonlyrefs} 
\usepackage[section]{placeins} 
\usepackage[numbers, sort&compress]{natbib} 
\usepackage{enumitem}
\usepackage{hyperref}

\usepackage{tikz}
\usetikzlibrary{arrows.meta,positioning,backgrounds,calc,fit,shapes.misc,shapes.geometric}

\definecolor{Xcol}{HTML}{693CAA}      
\definecolor{Ycol}{HTML}{0F7FDB }      
\definecolor{robust}{HTML}{693CAA}    
\definecolor{distrib}{HTML}{0F7FDB}   
\definecolor{struct}{HTML}{DB6B0F}    
\definecolor{panelfill}{HTML}{F6F6F8}
\definecolor{centfill}{HTML}{FCF3E9} 

\definecolor{jointcol}{RGB}{120,120,130}
\definecolor{condYX}{HTML}{00C3EA}    
\definecolor{condXY}{HTML}{D752DA}    

\newcommand{\R}{\mathbb{R}} 
\newcommand{\C}{\mathcal{C}} 
\newcommand{\Okc}{\Omega_{K, c, \varepsilon}}

\newcommand{\Ocd}{\Omega_{c, \varepsilon, \delta}}
\newcommand{\D}{\Delta_{G,\ell, c}}
\newcommand{\phifix}{\phi_{\lambda, G}}

\newtheorem{theorem}{Theorem}[section]
\newtheorem{corollary}{Corollary}[theorem]
\newtheorem{lemma}[theorem]{Lemma}

\newtheorem{proposition}{Proposition}[section]
\theoremstyle{definition}
\newtheorem{definition}{Definition}[section]
\newtheoremstyle{remarksty}
  {}{}{\normalfont}{}{\bfseries}{.}{ }{}
\theoremstyle{remarksty}
\newtheorem{remark}{Remark}[section]
\newtheorem{example}{Example}[section]
\newenvironment{remark*}
  {\pushQED{\hfill$\lozenge$}\begin{remark}}
  {\popQED\end{remark}}
\newenvironment{example*}
  {\pushQED{\hfill$\lozenge$}\begin{example}}
  {\popQED\end{example}}
\newcommand{\email}[1]{\texttt{#1}}

\title{A Distributionally Robust Framework for Learned Reconstructions in Inverse Problems}
\author{Floor van Maarschalkerwaart\thanks{University of Twente, Enschede, the Netherlands 
(\email{\{f.vanmaarschalkerwaart, c.brune, m.c.carioni\}@utwente.nl}.}
\and Subhadip Mukherjee\thanks{IIT Kharagpur, Kharagpur, India 
  (\email{smukherjee@ece.iitkgp.ac.in}).} 
\and Christoph Brune\footnotemark[1] 
\and Marcello Carioni\footnotemark[1]}
\date{\today}

\begin{document}

\maketitle
\begin{abstract}
Learned reconstruction operators for inverse problems are typically trained under a fixed noise model, and generalize poorly when the distribution during testing differs from the one assumed during training. Distributionally robust optimization (DRO) addresses this by optimizing against the worst-case distribution within a prescribed ambiguity set, but standard Wasserstein DRO perturbs the full joint distribution uniformly, which can be overly conservative and ignores the physics of the measurement process. We develop a structured DRO framework in which the ambiguity set is restricted to structured perturbations aligned with the data-acquisition process. This allows us to learn data-driven reconstruction operators that remain robust to distributional shifts. By constraining perturbations to subsets such as $P(Y|X)$, our framework models uncertainty in the forward operator and noise model more faithfully, accommodating any noise model expressible as a stochastic forward operator. We establish strong duality for this general formulation and derive explicit finite-dimensional dual representations for perturbations in the joint, marginal, and conditional distributions. A central result is an explicit worst-case risk bound that induces Tikhonov regularization on the Lipschitz constant of the reconstruction operator, and is less conservative relative to standard DRO for well-posed problems. Numerical experiments on deblurring and sinogram-to-CT reconstruction demonstrate improved robustness, stability, and interpretability over standard DRO and MSE baselines. In the linear setting, the learned operator becomes effectively low-rank, truncating at the intrinsic dimension of the data and recovering a data-driven analogue of truncated-SVD regularization.
\end{abstract}

\section{Introduction}
Inverse problems aim to recover an unknown signal from indirect, noisy measurements. They are central to imaging, from deblurring to computed tomography, and are typically ill-posed. Learned reconstructions have become a popular tool, but they are usually trained under a fixed assumption on the noise model. More broadly, in many estimation and reconstruction problems, the data-generating distribution is often only partially known, and solutions optimized for an assumed model may be unstable when the true distribution differs. Distributionally robust optimization (DRO) addresses this by seeking solutions that perform well against all probability measures in a prescribed ambiguity set around the empirical distribution. Given an estimated in-sample distribution $\mu^*$, one learns a decision rule or reconstruction operator $G\in\mathcal{G}$ by solving the min–max problem:
\begin{align}\label{eq:DRO}
    \min_{G \in \mathcal{G}} \max_{\mu \in B_{D,\varepsilon} (\mu^*)}
    \int_S \ell(s,G)\,\mathrm{d}\mu(s),
\end{align}
where $\ell: S \times \mathcal{G} \to [0, \infty)$ is a measurable loss function and $B_{D,\varepsilon}(\mu^*)$ is the $\varepsilon$-ball of perturbations around $\mu^* \in P(S)$ with respect to a discrepancy $D: P(S) \times P(S) \to \R$. By optimizing against the worst case within the ambiguity set rather than the average (the observed data distribution), DRO provides robustness against noise and model misspecification and offers generalization to out-of-distribution data. 

A common choice for the ambiguity set is a Wasserstein ball. The Wasserstein distance $W: P(S)\times P(S)\to[0,\infty)$ \cite{villani_optimal_2009}, defined through optimal transport, captures discrepancies in both the support and mass of probability measures and carries an intuitive geometric interpretation as the minimal transport cost between two distributions. Unlike $\phi$-divergences such as the KL divergence, it is well-defined for any pair of measures and does not require overlapping support, making it suitable when data lie on low-dimensional manifolds. Moreover, it captures the geometry and relative positions of the supports of the distributions, which is not the case for other measures such as integral probability metrics (IPMs). Formally, given a cost $c: S \times S \to [0, \infty]$ the Wasserstein distance is defined as
\begin{align}\label{eq:w-dist}
    W(\mu_1, \mu_2) := \inf_{\pi \in \Pi(\mu_1, \mu_2)}
    \int_{S \times S} c(s, r)\,\mathrm{d}\pi(s,r),
\end{align}
where $\Pi(\mu_1, \mu_2)$ denotes the set of transport plans with marginals $\mu_1$ and $\mu_2$. The entropically regularized version known as the Sinkhorn distance \cite{peyre_computational_2019, cuturi_sinkhorn_2013} enables efficient computation and is used in our numerical experiments. 

We adopt a Bayesian interpretation of the data-measurement process: given an input $x \in X$, the observation $y$ is a $Y$-valued random variable distributed according to a conditional law $\mu_{Y|X=x}: X \to P(Y)$ whose mean is given by a forward operator $H: X \to Y$,
\begin{align}\label{eq:forward-operator}
    y \sim \mu_{Y|X=x}, \qquad \mathbb{E}[\,y | x\,] = Hx.
\end{align}
The conditional law $\mu_{Y|X=x}$ encodes the noise model: the additive formulation $y = Hx + \eta$ with zero-mean noise $\eta$ corresponds to $\mu_{Y|X=x} = \mathrm{Law}(Hx + \eta)$, while Poisson and other noise models arise from other choices of $\mu_{Y|X=x}$ (see Example~\ref{ex:noise}). Together with a prior $\mu_X \in P(X)$, the conditional law induces a joint distribution $\mu \in P(X\times Y)$.

\begin{figure}[tp]
  \centering
  \begin{tikzpicture}[font=\small, >=Stealth]
    \def\W{9} \def\H{6}
    \begin{scope}
      \clip (0,0) rectangle (\W,\H);
      \fill[jointcol!18, rotate around={26.57:(4.5,3)}] (4.5,3) ellipse [x radius=3.0, y radius=0.80];
      \draw[jointcol!55, rotate around={26.57:(4.5,3)}] (4.5,3) ellipse [x radius=3.0, y radius=0.80];
      \draw[jointcol!70, rotate around={26.57:(4.5,3)}] (4.5,3) ellipse [x radius=1.9, y radius=0.50];
    \end{scope}
    \draw[black, thick] (1.82,1.66) -- (7.18,4.34);
    \node[font=\footnotesize, rotate=26.57, anchor=south] at (2.55,2.07) {$y=Hx$};
    \draw[thick] (0,0) rectangle (\W,\H);
    \draw[gray!45, thin] (0.7,-0.30) -- (8.3,-0.30);
    \draw[gray!45, thin] (-0.30,0.4) -- (-0.30,5.6);
    \draw[Xcol, thick] plot[domain=0.7:8.3, samples=90, smooth, variable=\x]
          (\x, {-0.30 - 0.85*exp(-(\x-4.5)^2/2.0)});
    \node[Xcol] at (8.35,-1.0) {$\mu_X$};
    \draw[Ycol, thick] plot[domain=0.4:5.6, samples=80, smooth, variable=\y]
          ({-0.30 - 0.85*exp(-(\y-3)^2/1.3)}, \y);
    \node[Ycol] at (-1.0,5.55) {$\mu_Y$};
    \def\xz{3.0}
    \draw[condYX!55, dashed] (\xz,0) -- (\xz,\H);
    \draw[condYX, thick] plot[domain=0.8:3.7, samples=60, smooth, variable=\y]
          ({\xz + 0.55*exp(-(\y-2.25)^2/0.42)}, \y);
    \draw[<->, condYX, line width=1.4pt] (\xz,1.45) -- (\xz,3.05);
    \node[condYX, anchor=north, align=center, font=\footnotesize] at (\xz,0.85)
          {perturb $\mu_{Y| X}$ ($\mu_X$ fixed)};
    \def\yz{3.7}
    \draw[condXY!45, dashed] (0,\yz) -- (\W,\yz);
    \draw[condXY, thick] plot[domain=4.4:7.4, samples=60, smooth, variable=\x]
          (\x, {\yz + 0.45*exp(-(\x-5.9)^2/0.42)});
    \draw[<->, condXY, line width=1.1pt] (5.0,\yz) -- (6.8,\yz);
    \node[condXY, anchor=south, align=center, font=\footnotesize] at (5.9,\yz+0.45)
          {perturb $\mu_{X|Y}$ ($\mu_Y$ fixed)};
    \node[anchor=north, font=\footnotesize] at (4.5,-1.55) {$X$ \;(input / parameter)};
    \node[anchor=south, rotate=90, font=\footnotesize] at (-1.55,3) {$Y$ \;(measurement)};
  \end{tikzpicture}
  \caption{Structured perturbations in the product space $X\times Y$. Fixing the $X$-marginal confines the $Y| X$ adversary to vertical transport; fixing the $Y$-marginal gives horizontal $X| Y$ transport. In contrast, standard DRO perturbs the full joint distribution (gray oval).}
  \label{fig:product-space}
\end{figure}
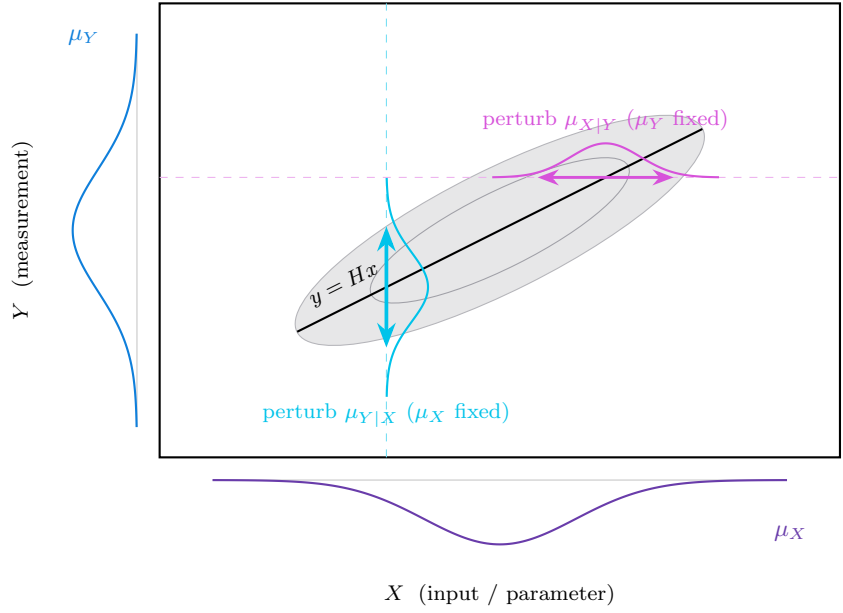

The task of recovering $x$ from noisy observations is often ill-posed: the operator $H$ may be ill-conditioned or non-injective, and small data perturbations can produce large reconstruction errors. Standard approaches address this through variational regularization or Bayesian formulations. The former minimizes $L(x,y) + R(x)$ where $L: X\times Y\to\R$ is a data fidelity term and $R: X\to[0,\infty]$ encodes prior knowledge, with Tikhonov regularization $R(x)=\|Tx\|_2^2$ as a canonical example. Bayesian formulation targets the posterior of $x$ given $y$ \cite{dashti_bayesian_2017}. While classical regularization mitigates noise amplification, it does not distinguish between different sources of uncertainty. Likewise, standard Wasserstein DRO treats all perturbations uniformly by perturbing the full joint distribution $\mu \in P(X\times Y)$, which may be unnecessarily conservative and may not reflect the physics of the measurement process. From a learning perspective, the goal is not merely to reconstruct a single signal $x$, but to learn a reconstruction operator $G:Y\to X$ that generalizes across measurements and remains stable under uncertainty in the data-generating process. 

We build \emph{Structured DRO}, a novel distributionally robust optimization framework that allows for structured perturbations within a given subset $K \subseteq P(X\times Y)$, recovering perturbations of the marginal, conditional, and joint distributions as special cases. This is precisely where DRO and inverse problems meet, c.f. Figure~\ref{fig:overview}. The data-generating law factorizes as $\mu = \mu_X \otimes \mu_{Y|X=x}$ into a prior over inputs and a conditional measurement likelihood, representing physically distinct sources of uncertainty: the input distribution on one hand, and the forward operator together with measurement noise on the other. Choosing $K$ in accordance with this factorization lets the adversary perturb each component separately rather than treating the joint as a whole. In particular, fixing $\mu_X$ and perturbing only $\mu_{Y|X=x}$ isolates uncertainty in the measurement model $P(Y| X)$, naturally aligning the framework with the physics of the inverse problem, as illustrated in Figure~\ref{fig:product-space}. The central object is a learned reconstruction operator $G \in \mathcal{G}$ (linear or a neural network) trained to minimize worst-case reconstruction risk over all admissible perturbations. Our earlier work \cite{van_maarschalkerwaart_perturbation-aware_2025} is a special case, corresponding to a $K$ that fixes $\mu_X$ and restricts $\mu_{Y|X=x}$ to Gaussian models with mean $Hx$ and unknown variance. 

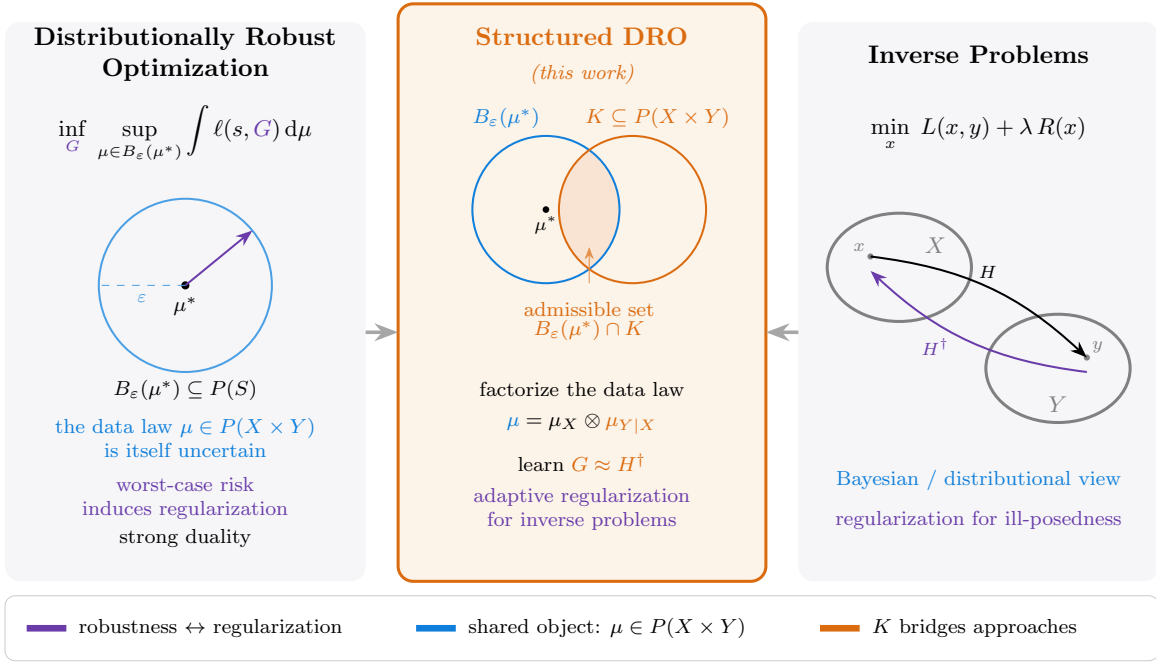
\begin{figure}[tp]
  \centering
  \resizebox{0.98\textwidth}{!}{%
    \begin{tikzpicture}[
        font=\small,
        >=Stealth,
        title/.style={font=\bfseries\normalsize, align=center},
        obj/.style={align=center},
        key/.style={align=center, inner sep=1pt},
        chip/.style={rounded corners=2pt, inner sep=2.5pt, font=\footnotesize, align=center},
      ]
      \begin{scope}[on background layer]
        \fill[panelfill, rounded corners=6pt] (0,1.15) rectangle (5.0,8.9);
        \fill[centfill, draw=struct, very thick, rounded corners=6pt] (5.45,1.15) rectangle (10.55,9.15);
        \fill[panelfill, rounded corners=6pt] (11.0,1.15) rectangle (16.0,8.9);
      \end{scope}
      \node[title] at (2.5,8.45) {Distributionally Robust\\Optimization};
      \node[obj] at (2.5,7.35) {$\displaystyle \inf_{\textcolor{robust}{G}}\;\sup_{\mu\in B_{\varepsilon}(\mu^*)}\int \ell(s,\textcolor{robust}{G})\,\mathrm{d}\mu$};
      \coordinate (bc) at (2.5,5.25);
      \draw[distrib!75, thick] (bc) circle (1.2);
      \fill[black] (bc) circle (1.6pt);
      \node[below=1pt, font=\footnotesize] at (bc) {$\mu^*$};
      \draw[->, robust, thick] (bc) -- ($(bc)+(0.93,0.76)$);
      \draw[distrib!75, dashed] (bc) -- ($(bc)+(-1.2,0)$) node[midway,below,font=\scriptsize]{$\varepsilon$};
      \node[font=\footnotesize] at (2.5,3.8) {$B_{\varepsilon}(\mu^*)\subseteq P(S)$};
      \node[align=center, font=\footnotesize, text=distrib] at (2.5,3.15)
         {the data law $\mu\in P(X\times Y)$\\ is itself uncertain};
      \node[chip, text=robust] at (2.5,2.3) {worst-case risk\\ induces regularization};
      \node[key, font=\footnotesize] at (2.5,1.75) {strong duality};
      \node[title] at (13.5,8.45) {Inverse Problems};
      \node[obj] at (13.5,7.35) {$\displaystyle \min_{x}\; L(x,y)+\lambda\, R(x)$};
      \node[draw=black!50, very thick, ellipse, minimum width=2.0cm, minimum height=1.5cm] (Xsp) at (12.4,5.5) {};
      \node[draw=black!50, very thick, ellipse, minimum width=2.0cm, minimum height=1.5cm] (Ysp) at (14.6,4.1) {};
      \node[black!50] at (12.9,5.8) {$X$};
      \node[black!50] at (14.6,3.6) {$Y$};
      \fill[black!50] (12.0,5.65) circle (1.2pt) node[above left=-2pt,font=\scriptsize,text=black!50]{$x$};
      \fill[black!50] (15.0,4.25) circle (1.2pt) node[above right=-2pt,font=\scriptsize,text=black!50]{$y$};
      \draw[->, thick] (12.0,5.65) to[bend left=18] node[above,font=\scriptsize]{$H$} (15.0,4.25);
      \draw[->, robust, thick] (15.0,4.05) to[bend left=18] (12.0,5.45);
      \node[font=\scriptsize, text=robust] at (12.9,4.4) {$H^{\dagger}$};
      \node[chip, text=robust]  at (13.5,2.0) {regularization for ill-posedness};
      \node[chip, text=distrib] at (13.5,2.55) {Bayesian / distributional view};
      \node[title, text=struct] at (8,8.7) {Structured DRO};
      \node[font=\footnotesize\itshape, text=struct] at (8,8.18) {(this work)};
      \coordinate (Bc) at (7.5,6.3);
      \coordinate (Kc) at (8.7,6.3);
      \begin{scope}
        \clip (Bc) circle (1.02);
        \fill[struct!18] (Kc) circle (1.02);
      \end{scope}
      \draw[distrib, thick] (Bc) circle (1.02);
      \draw[struct,  thick] (Kc) circle (1.02);
      \fill[black] (Bc) circle (1.3pt);
      \node[font=\scriptsize, below=-1pt] at (Bc) {$\mu^*$};
      \node[distrib, font=\footnotesize, anchor=south] at (6.95,7.3) {$B_\varepsilon(\mu^*)$};
      \node[struct,  font=\footnotesize, anchor=south] at (9.05,7.3) {$K\subseteq P(X\times Y)$};
      \node[align=center, font=\footnotesize, text=struct] at (8.1,4.75)
         {admissible set\\[-2pt] $B_\varepsilon(\mu^*)\cap K$};
      \draw[struct!75, ->] (8.1,5.28) -- (8.1,5.78);
      \node[align=center, font=\footnotesize] at (8,3.55)
        {factorize the data law\\[2pt]
         ${\color{distrib}\mu}=\mu_X\otimes{\color{struct}\mu_{Y| X}}$};
      \node[align=center, font=\footnotesize] at (8,2.8)
        {learn \color{struct}$G\approx H^\dagger$};
      \node[align=center, font=\footnotesize] at (8,2.15)
        {\color{robust} adaptive regularization\\ \color{robust}for inverse problems};
      \draw[->, line width=1.4pt, gray!70] (5.0,4.6) -- (5.45,4.6);
      \draw[->, line width=1.4pt, gray!70] (11.0,4.6) -- (10.55,4.6);
      \begin{scope}[on background layer]
        \fill[white, rounded corners=4pt, draw=gray!40] (0,0.05) rectangle (16,0.95);
      \end{scope}
      \draw[robust, line width=2.6pt]  (0.3,0.5) -- (0.85,0.5);
      \node[anchor=west, font=\footnotesize] at (0.9,0.5) {robustness $\leftrightarrow$ regularization};
      \draw[distrib, line width=2.6pt] (5.7,0.5) -- (6.25,0.5);
      \node[anchor=west, font=\footnotesize] at (6.3,0.5) {shared object: $\mu\in P(X\times Y)$};
      \draw[struct, line width=2.6pt]  (11.3,0.5) -- (11.85,0.5);
      \node[anchor=west, font=\footnotesize] at (11.9,0.5) {$K$ bridges approaches};
    \end{tikzpicture}%
  }
  \caption{Schematic positioning \emph{Structured DRO} (center) between distributionally robust optimization (left) and inverse problems (right). DRO optimizes against the worst case over an ambiguity ball $B_\varepsilon(\mu^*)$, inducing regularization made explicit through strong duality. Inverse problems stabilize a reconstruction $H^\dagger$ of an ill-posed forward operator $H$ by adding a regularizer $R$ to the data-fidelity term. Structured DRO bridges the two: the data law factorizes as $\mu=\mu_X\otimes\mu_{Y|X}$, and restricting the worst-case distribution to a structured set $K\subseteq P(X\times Y)$ yields the admissible set $B_\varepsilon(\mu^*)\cap K$, over which a reconstruction operator $G\approx H^\dagger$ is learned with regularization adapted to the inverse problem. The legend (bottom) illustrates the three threads connecting the viewpoints.} 
  \label{fig:overview}
\end{figure}

We establish strong duality for this general formulation on general parameter and data spaces, extending results of Blanchet and Murthy \cite{blanchet_quantifying_2019} and Blanchet et al. \cite{blanchet_robust_2019} to structured perturbation sets, and for each perturbation space considered (joint, marginal, and conditional), we derive finite‑dimensional dual representations that make the resulting DRO problems interpretable and tractable and offer a clear link to regularization. For each perturbation space, we further obtain explicit upper bounds on the worst‑case risk. For perturbations in $P(Y|X)$, the bound features a Tikhonov-type regularization term penalizing the Lipschitz constant of the reconstruction operator, manifesting concretely in the linear setting as truncated-SVD regularization, with the learned inverse truncating at the intrinsic dimension of the data. Restricting perturbations to $P(Y|X)$ is less conservative than standard Wasserstein DRO, with the gap widening as the problem becomes more well-posed. Finally, we use the conditional Sinkhorn distance as a tractable approximation of the structured framework and follow \cite{wang_sinkhorn_2026} to apply our framework to inverse problems, including linear differentiation, deblurring, and sinogram‑to‑CT reconstruction, as well as to a robust forward simulator, demonstrating that the framework extends beyond the inverse setting. Across all experiments, structured DRO achieves improved robustness, stability, and interpretability, and can outperform baselines across a range of practically relevant scenarios.

\paragraph{Main contributions} The main contributions of this paper are the following.
\begin{enumerate}
    \item We propose a novel DRO framework for learning robust reconstruction operators under structured uncertainty in inverse problems and prove strong duality for general structured perturbation sets, extending the results of Blanchet and Murthy \cite{blanchet_quantifying_2019} and Blanchet et al. \cite{blanchet_robust_2019} beyond linear regression to general parameter and data spaces and constrained ambiguity sets.
    \item We derive finite-dimensional dual formulations and worst-case risk bounds for perturbations in the joint, marginal, and conditional distributions. For perturbations in $P(Y|X)$, the bound induces Tikhonov-type regularization on the Lipschitz constant of the reconstruction operator. In the linear setting, this recovers truncated-SVD regularization at the intrinsic dimension of the data.
    \item We apply the framework to specific inverse problems and demonstrate improved robustness, stability, and interpretability over baselines on differentiation, deblurring, a robust forward simulator, and sinogram-to-CT reconstruction.
\end{enumerate}

\paragraph{Organization of the paper}
Section 2 introduces the framework and its application to inverse problems. Sections 3 and 4 contain the theoretical contributions: strong duality and explicit dual formulations for structured perturbation sets, followed by explicit worst-case risk bounds and their connections to known duality results and classical regularization. Section 5 is a theoretical validation study comparing different DRO formulations for fixed, hand-crafted inverses across inverse problems of increasing ill-posedness, illustrating the bounds of Section 4. Section 6 forms the main empirical part, solving the DRO problem for learned reconstruction maps via a tractable entropic approximation of the conditional Wasserstein distance and applying the framework to inverse problems, including image reconstruction tasks.

\subsection{Related Work}
The literature on DRO has expanded in recent years \cite{rahimian_distributionally_2019}. Much of this work focuses on optimal transport distances \cite{blanchet_quantifying_2019, mohajerin_esfahani_data-driven_2018, gao_distributionally_2023}, though different metrics remain relevant \cite{zhu_kernel_2021, hu_kullback-leibler_2012, blanchet_unifying_2025, wang_sinkhorn_2026}. Bayesian formulations of DRO have also been investigated \cite{zhang_distributionally_2025, nguyen_bridging_2023}, and DRO has been proven useful for data-driven approaches \cite{kuhn_wasserstein_2019}. In particular, Wasserstein DRO has been shown to induce implicit regularization, connecting worst-case risk minimization to regularized estimators \cite{shafieezadeh-abadeh_regularization_2019, blanchet_robust_2019, blanchet_distributionally_2025, li_tikhonov_2022}. Entropic regularization further enables efficient computation of Wasserstein distances \cite{cuturi_sinkhorn_2013, peyre_computational_2019}. 

While classical DRO perturbs the full joint distribution, recent work has explored structured ambiguity sets that exploit problem-specific geometry \cite{liu_need_2023} or conditional structures \cite{chenreddy_data-driven_2022}. These approaches can reduce conservatism by distinguishing between different sources of uncertainty. However, existing methods rarely account for the conditional structure inherent in inverse problems, where the forward operator and measurement noise define distinct components of uncertainty. The difference between our approach and that of \cite{chenreddy_data-driven_2022} lies in what is being perturbed: there, robustness is taken over an uncertainty set of realizations of the perturbation vector, learned from data and conditioned on observed covariates, so the formulation is ``deep data-driven'' robust rather than \emph{distributionally} robust. We instead perturb the data-generating measure itself and impose conditional structure by fixing the $X$-marginal while allowing perturbations only in the conditional law $P(Y|X)$.

In the context of inverse problems, regularization and Bayesian formulations remain the standard for stabilizing ill-posed reconstructions \cite{benning_modern_2018}, while data-driven methods increasingly integrate machine learning \cite{arridge_solving_2019, mukherjee_end--end_2021, hauptmann_convergent_2025, ongie_deep_2020}. These range from purely learned reconstruction maps \cite{adler_learned_2018, jin_deep_2017, zhu_image_2018} to learned regularizers \cite{li_nett_2020, kobler_total_2020-1, habring_neural-network-based_2024, goujon_neural-network-based_2023} and plug-and-play schemes that substitute a denoiser for the proximal step \cite{sun_online_2019, pesquet_learning_2021}, with Bayesian and generative formulations used to quantify reconstruction uncertainty \cite{barbano_quantifying_2021, oliviero-durmus_generative_2025}. Comparative and convergence-oriented studies of learned regularizers \cite{hertrich_learning_2025, mukherjee_learned_2023} and a survey of diffusion-based approaches \cite{daras_survey_2024} give broader overviews.

Optimal transport-based methods are gaining popularity in this setting \cite{carioni_unsupervised_2023}. The Wasserstein distance itself finds applications in various aspects of inverse problems, serving as a regularization term \cite{bredies_generalized_2023-1}, appearing in a projection-based method \cite{heaton_wasserstein-based_2022}, or as a loss function \cite{adler_learning_2017}, and it underlies adversarial regularizers \cite{lunz_adversarial_2019} as well as OT-based generative models \cite{arjovsky_wasserstein_2017, genevay_learning_2018, patrini_sinkhorn_2020}. More recently, flow matching has been used to build solvers for inverse problems \cite{pourya_flower_2026, martin_PnP-flow_2025}.

Our structured DRO framework explicitly models uncertainty in the different components of the measurement process, extending standard DRO to general data spaces and incorporating the possible ill-posedness and non-linearity of an inverse problem, while exploiting their properties through structured perturbation sets and providing dual formulations yielding interpretable regularization.
\section{Structured Distributionally Robust Optimization}
This section introduces our framework for structured distributionally robust optimization. We begin by specifying the admissible perturbations through an $S$-parametrized family of sets $K_s \subset P(S)$. The set $K_s$ contains all perturbations of the empirical distribution $\mu^* \in P(S)$ that can be chosen by the adversary in \eqref{eq:DRO}. Note that since $K_s$ depends explicitly on $s$, the adversary can choose perturbations depending on the input $s$. In the rest of the paper, a $*$ denotes that we assume the measure is fixed. 

More concretely, we characterize admissible perturbations of $\mu^*$ obtained through the action of $\mu^*$-measurable kernels $\pi_{S|S} : S \rightarrow P(S)$ such that $\pi_{S|S=s} \in K_s$ for $\mu^*$-a.e. $s \in S$ as 
\begin{align*}
    \pi[\mu^*](A) := \int \pi_{S|S=s}(A)\, d\mu^*(s)  \quad A\subset S \ \ \text{Borel set}.
\end{align*}
With a slight abuse of notation, we will denote such measurable kernels by $\pi_{S|S} : S \rightarrow K_s$; however, we warn the reader that they should be intended to map to $K_s$ for $\mu^*$-a.e. $s \in S$.
Note that $\pi[\mu^*] \in P(S)$ is simply the second marginal of the joint distribution $\mu^* \otimes \pi_{S|S=s}$. 

We now consider family of admissible kernels $\pi_{S|S}$ whose action lies within an $\varepsilon$-ball around $\mu^*$, that is the set 
\begin{align*}
    \Okc &:= \Big\{\pi_{S|S}: S \rightarrow K_s,\ \int_S \int_S c(s,r)\,d\pi_{S|S=s}(r)\,d\mu^*(s) \leq \varepsilon\Big\}.
\end{align*}
Note that if $K_s = P(S)$, then $\Okc$ is simply the $\varepsilon$-Wasserstein ball around $\mu^*$. However, in the case of more restrictive choices on $K_s$, this is no longer the case.
Our focus is on distributionally robust optimization problems defined by such type of perturbations. Given a parameter space $\mathcal{G}$ and a measurable loss function $\ell: S \times \mathcal{G} \to [0, \infty)$, the structured Wasserstein-DRO (W-DRO) problem is formulated as follows.

\begin{definition}\label{def:padro}
    The structured W-DRO problem is 
    \begin{align*}
        \inf_{G \in \mathcal{G}} \sup_{\pi_{S|S} \in \Okc} \int_S\int_S \ell(r,G)\,d\pi[\mu^*](r) 
      \end{align*}
      or equivalently
\begin{align}\label{eq:d-padro}
        \inf_{G \in \mathcal{G}} \sup_{\pi_{S|S} \in \Okc} \int_S\int_S \ell(r,G)\,d\pi_{S|S=s}(r)\,d\mu^*(s).
    \end{align}
\end{definition}

\subsection{Structured DRO Applied to Inverse Problems}\label{sec:str}
In the context of inverse problems, where the sample space is $S = X\times Y$ (we will often write $s = (x,y)$ to indicate an element of $S$), the choice of admissible perturbations $K_s \subset P(X \times Y)$ should reflect the type of robustness we aim to enforce in the solution. This makes the choice of $K_s$ inherently problem-dependent. We adopt a Bayesian interpretation of the inverse problem, where the data acquisition process induces a joint distribution $\mu \in P(X \times Y)$ over the parameter and measurement spaces. From this perspective, the goal is to identify worst-case posterior distributions with respect to different classes of perturbations $K_s$. 

Throughout, $\mu$ denotes a probability measure on $X\times Y$ while $\mu_X$, $\mu_Y$ are its marginals and conditional law $\mu_{Y|X}: X \to P(Y)$, while $\pi$ is reserved for couplings and perturbation kernels on the sample space $S\times S$ that transport such measures. 

We distinguish four main types of perturbations:
\begin{enumerate}
    \item \textbf{$X$-perturbations}: These perturb the distribution of the inputs. For example, via transformations, shifts, or frequency modifications. Here, 
    \begin{align*}
        K_{(x,y)} \subset \{\mu_X \otimes \mu^*_{Y|X} : \mu_X \in P(X)\}
    \end{align*}
    where $\mu^*_{Y|X}$ is an estimated posterior and can be chosen, for example, as $\mu^*_{Y|X=u}= \delta_{Hu}$. Note that in this case $K_{(x,y)}$ is independent of the input $(x,y)$.
    \item \textbf{$Y$-perturbations}: These model changes in the distribution of the measurements, effectively perturbing the observed data. In this case, 
    \begin{align*}
        K_{(x,y)} \subset \{ \mu_Y \otimes \mu^*_{X|Y}  : \mu_Y \in P(Y)\}
    \end{align*}
    where $\mu^*_{X|Y}$ is a given estimated likelihood such as for example $\mu^*_{X|Y=v} = \delta_{H^\dagger v}$ where $H^\dagger$ denotes some pseudo-inverse of $H$. Note that also in this case $K_{(x,y)}$ is independent on the input $(x,y)$.
    \item \textbf{$Y|X$-perturbations}: These describe perturbations in the conditional distribution $Y|X$, which are particularly relevant for inverse problems as it mimics the data acquisition process. It can include different types of noise present in the measurements as well as other types of model imperfections, such as an imperfect forward model. A key example is additive noise in measurements as in \eqref{eq:forward-operator}. 
    The admissible set is defined as all joint distributions whose marginal over $X$ matches a given reference. 
    In particular,
    \begin{align*}
        K_{(x,y)} \subset \{ \delta_x \otimes \mu_{Y|X} :  \mu_{Y|X} \in P(Y)\}.
    \end{align*}
    Note that if equality holds above, then the set of admissible perturbations reduces to the set of $\mu \in P(S)$ with marginal $\mu_X^*$. Moreover, in this case, the set $K_{(x,y)}$ depends on $x$. It fixes $x$ by allowing only perturbations in the conditional law $\mu_{Y|X=x}$.
    \item \textbf{$X|Y$-perturbations}: Similar to the previous class, but instead focused on the inverse conditional process $X|Y$. The admissible set here consists of
    \begin{align*}
        K_{(x,y)} \subset \{ \delta_y \otimes \mu_{X|Y} : \mu_{X|Y}\in P(X) \}.
    \end{align*}
\end{enumerate}

In this work, we focus on \textbf{$\mathbf{Y|X}$-perturbations}, due to their particular relevance in modeling measurement noise. As an example, consider an inverse problem with a general noisy forward operator $y=H^\delta x$ where $H^\delta: X \to Y$ incorporates the uncertainty of the measurements. Different noise models correspond to different choices of $H^\delta$, and in the structured DRO framework, that leads to a corresponding set of joint distributions
\begin{align*}
    K_s = \left\{ \delta_x \otimes \mu_{Y|X}^{H^\delta} : H^\delta \in \mathcal{H} \right\},
\end{align*}
where $\mathcal{H}$ denotes the class of admissible stochastic forward operators and $\mu_{Y|X}^{H^\delta} \in P(Y)$ is the conditional distribution of $v$ given $u$, defined by $\mu_{Y|X=x}^{H^\delta} = \text{Law}(H^\delta x)$ with $\text{Law}(Z)$ denoting the probability distribution of a random variable $Z$. Note that, for any $\phi \in C(S)$ and $\pi_{S|S} \in K_s$ it holds that
\begin{align*}
    \pi[\mu^*](\varphi) &  = \int \phi(u,v)\ d\mu_{Y|X=u}^{H^\delta}(v)\ d\delta_x(u)\ d\mu^*(x,y) 
    = \int \phi(x,v)\ d\mu_{Y|X=x}^{H^\delta}(v)\ d\mu_X^*(x)
\end{align*}
where $\mu_X^*$ is the first marginal of $\mu^*$, implying thus that $\pi[\mu^*] = \mu_X^*\otimes \mu_{Y|X}^{H^\delta}$. 
Therefore admissible perturbations are the ones that keep $\mu_X^*$ fixed and perturb each measurement $y \in Y$ given $x\in X$ as $\mu_{Y|X=x}^{H^\delta}$. As a consequence, as we will discuss in Section \ref{sec:reg-y|x}, the W-DRO problem \eqref{eq:d-padro} searches for ambiguities in a conditional Wasserstein ball \cite{chemseddine_conditional_2025}.

The set $K_s$ accommodates any noise model expressible via a stochastic forward operator $H^\delta$, allowing robust reconstructions under a wide range of uncertainties.
\begin{example}\label{ex:noise}
Noise models in inverse problems can be formulated in the structured DRO framework as follows.
\begin{enumerate}[label=\textnormal{(\alph*)}, itemsep=0pt, topsep=2pt]
    \item \textbf{Additive Gaussian noise.}\label{ex:noise:gaussian}
    The forward operator is $H^\delta x = Hx + \delta$ with $\delta \sim \mathcal{N}(0,\Sigma)$ and $\Sigma \succ 0$, so the conditional distribution is $\mu_{Y|X=x}^{H^\delta}(y) = \mathrm{Law}(Hx+\delta)(y) = \mathcal{N}(Hx,\Sigma)(y)$. We recover this model with $K_s = \{ \delta_x \otimes \mathcal{N}(Hx, \Sigma) : \Sigma \preceq M_1 \}$ for $0 \prec M_1$.

    \item \textbf{Poisson noise.}\label{ex:noise:poisson}
    Poisson noise models photon-counting errors, with applications in low-dose CT and other photon-limited imaging. The noisy operator is $H^\delta x \sim \tfrac{1}{I_0}\mathrm{Pois}(I_0 Hx)$ for $Hx \geq 0$, where $I_0 > 0$ is the photon count. The conditional distribution is then $\mu_{Y|X=x}^{H^\delta} = \tfrac{1}{I_0}\mathrm{Pois}(I_0 Hx)$, and $K_s = \{ \delta_x \otimes \tfrac{1}{I_0}\mathrm{Pois}(I_0 Hx) : I_0 \leq M_2 \}$ with $M_2>0$.
\end{enumerate}
One can already note that the sets $K_s$ in the noise models defined before are neither convex nor weak* closed. Even if not relevant for the primal problem, this will be relevant for the next section, where strong duality is established under such assumptions. To resolve such an issue, it is enough to consider the weak*-closed convex envelope of $K_s$.
\end{example}
\section{Duality Theory for Structured DRO}\label{sec:duality}
The inner worst-case term in a DRO problem maximizes the expected loss over the entire ambiguity set, an infinite-dimensional family of probability measures, and cannot be computed directly. Under suitable assumptions, strong duality rewrites this supremum over measures as a finite-dimensional minimization over a dual variable, reducing the min-max problem to a minimization. The resulting dual representation also reveals the structure of the worst-case risk and makes the link between robustness and regularization explicit, which we develop in the following sections, and is essential for solving \eqref{eq:d-padro}. 
Defining
\begin{align*}
    I(\pi_{S|S}, G) &:= \int_S\int_S \ell(r,G)\,d\pi_{S|S=s}(r)\,d\mu^*(s),
\end{align*}
the \emph{primal problem} \eqref{eq:d-padro} can be thus rewritten as 
\begin{align*}
I_{\Okc}(G) := \sup_{\pi_{S|S} \in \Okc} I(\pi_{S|S}, G).
\end{align*}
For the dual formulation, define
\begin{align*}
    \D := \Big\{ (\lambda,h) : \lambda \in \R_+,\ & h\in C(S\times S),\ \int_S\int_S h(s,r)\,d\pi_{S|S=s}(r)\,d\mu^*(s) \\
    &\geq \int_S\int_S \big(\ell(r, G) - \lambda c(s, r)\big)\,d\pi_{S|S=s}(r)\,d\mu^*(s) \quad \forall\, \pi_{S|S} : S \rightarrow K_s \Big\}.
\end{align*}
We first prove the following weak duality result.


\begin{proposition}[Weak duality]\label{prop:weakduality}
    Assume $S$ is a compact, Polish space, $c: S \times S \to \R_+$ is non-negative, continuous, and such that $c(s,r) = 0$ if and only if $s=r$, and $\ell(\cdot,G)$ is continuous. Then
    \begin{align*}
        I_{\Okc} (G)\leq \inf_{(\lambda, h) \in \D} \lambda \varepsilon + \sup_{\pi_{S|S} : S \rightarrow K_s} \int_S\int_S h(s,r) d\pi_{S|S=s}(r)d\mu^*(s).
    \end{align*}
\end{proposition}
\begin{proof}
    For any $\pi_{S|S} \in \Okc$ and $(\lambda, h) \in \D$ we have
    \begin{align*}
        \lambda \varepsilon + \int_S\int_S h(s,r) & d\pi_{S|S=s}(r)d\mu^*(s) \geq \lambda \varepsilon + \int_S\int_S \ell(r, G) - \lambda c(s,r) d\pi_{S|S=s}(r)d\mu^*(s) \\
        &= \int_S\int_S \ell(r, G)d\pi_{S|S=s}(r)d\mu^*(s) + \lambda \left(\varepsilon - \int_S\int_S c(s,r)d\pi_{S|S=s}(r)d\mu^*(s) \right) \\
        &\geq \int_S\int_S\ell(r, G) d\pi_{S|S=s}(r)d\mu^*(s) = I(\pi_{S|S}, G).
    \end{align*}
    Taking first the supremum with respect to $\pi_{S|S} \in \Okc$ and then the infimum over $(\lambda, h) \in \D$, we conclude that
    \begin{align*}
        I_{\Okc} (G) & = \sup_{\pi_{S|S} \in \Okc} I(\pi_{S|S}, G) \leq 
        \inf_{(\lambda, h) \in \D}  \lambda \varepsilon + \sup_{\pi_{S|S}\in \Okc} \int_S\int_S h(s,r) d\pi_{S|S=s}(r)d\mu^*(s) \\
        & \leq \inf_{(\lambda, h) \in \D}  \lambda \varepsilon + \sup_{\pi_{S|S} : S\rightarrow K_s} \int_S\int_S h(s,r) d\pi_{S|S=s}(r)d\mu^*(s)
    \end{align*}
    which is the claimed bound.
\end{proof}

We are now ready to state the strong duality result. For the proof, we refer the reader to Appendix \ref{supp:general-duality-proof}.

\begin{theorem}[Strong duality]\label{thm:w-padro-strong-duality}
    Assume $S$ is a compact, Polish space, $c: S \times S \to \R_+$ is non-negative, continuous, and such that $c(s,r) = 0$ if and only if $s=r$, and $\ell(\cdot,G)$ is continuous. Assume additionally that $K_s$ is non-empty, weakly* closed and convex for $\mu^*-a.e.\ s \in S$. Then if $\Omega_{K,c,\varepsilon}$ is also non-empty it holds that 
    \begin{align*}
            I_{\Okc} (G) = \inf_{(\lambda, h) \in \D} \lambda \varepsilon + \sup_{\pi_{S|S} : S \to K_s} \int_{S \times S} h(s,r) d\pi_{S|S=s}(r)d\mu^*(s).
        \end{align*}
\end{theorem}

\begin{remark*}\label{rem:strong-duality}
    We already mentioned here that the strong duality result is true also for non-compact $S$, lower semicontinuous cost and upper semicontinuous losses $\ell(\cdot, G)$, by following the same proof strategy of \cite{blanchet_quantifying_2019}. However, in order to keep the paper concise, we prefer to cover in detail only the case of compact $S$ and continuous $c$ and $\ell$. 
\end{remark*}

\subsection{Dual Problems for Inverse Problems with Targeted Perturbations}
In this section, we specialize the results obtained in the previous section to perturbations that are relevant for inverse problems reconstructions. In particular, we derive explicit formulas in case the perturbations are in $X\times Y$, $Y|X$, and $X$. We start with perturbations in the joint space; the proof is straightforward and can be found in Appendix \ref{supp:p(s)-duality-proof}. 

\begin{theorem}[Strong duality for perturbations in $P(S)$]\label{th:k=p(s)}
    Let $K_s=P(S)$ for every $s$, and suppose all assumptions in Theorem~\ref{thm:w-padro-strong-duality} hold. Then
    \begin{align*}
        I_{\Okc} (G) &= \inf_{\lambda \geq 0} \lambda \varepsilon +  \int_{S} \phifix(s) d\mu^*(s), \\
        \text{where} \quad \phifix (x,y) &:= \sup_{(u,v) \in S} \ell(u,v,G) - \lambda c((x,y),(u,v)).
    \end{align*}
\end{theorem}

When $K_s=P(S)$, the structured DRO problem reduces to the standard DRO problem, and the dual representation above recovers the standard DRO strong duality reformulation established by Blanchet and Murthy~\cite[Theorem~1b]{blanchet_quantifying_2019}. The statement is a direct specialization of Theorem~\ref{thm:w-padro-strong-duality} to $K_s=P(S)$: since $P(S)$ contains all Dirac measures on $S$, the supremum over kernels $\pi_{S|S=s}:S\to K_s$ in the general dual collapses to the pointwise supremum over $S$ appearing in $\phifix$, and the argument is straightforward. In the inverse problems setting $S:=X\times Y$, this corresponds to perturbations in the joint probability space of $X$ and $Y$.


We now consider perturbations only on the conditional distribution $P(Y|X)$, and we derive a finite-dimensional dual in this case. We refer to this model as $P(Y|X)$-DRO.

\begin{theorem}[Strong duality for perturbations in $P(Y|X)$]\label{th:k=Y|X}
Let $K_{(x,y)}=\{\mu=\delta_x \otimes \mu_{Y|X=x},\ \mu_{Y|X} \in P(Y|X) \}$ where $P(Y|X)$ denotes the set of Markov kernels $X \to P(Y)$, and suppose all assumptions in Theorem~\ref{thm:w-padro-strong-duality} hold. Then
    \begin{align*}
        I_{\Okc} (G) &= \inf_{\lambda \geq 0} \lambda \varepsilon + \int_{S} \phifix(x,y) d\mu^*(x,y), \\
        \text{where} \quad \phifix (x,y) &:=\sup_{v \in Y}\ell (x, v, G) - \lambda c((x,y),(x,v)).
    \end{align*}
\end{theorem}
\begin{proof}
We first prove the inequality $\leq$. Writing $s=(x,y)$ and $r=(u,v)$, note that 
 \begin{align*}
        \left\{(\lambda, h): \lambda \in \R_+,\ h(x,y,u,v) = \phifix(x,y) \right\} \subseteq \D.
    \end{align*}
Therefore, it holds that 
\begin{align*}
    I_{\Okc} (G) & = \inf_{(\lambda, h) \in \D} \lambda \varepsilon + \sup_{\pi_{S|S} : S \rightarrow K_s} \int_{S \times S} h(r,s) d\pi_{S|S=s} (r) d\mu^*(s)\\
    & \leq \inf_{\lambda \geq 0} \lambda\varepsilon + \sup_{\pi_{S|S} : S \rightarrow K_s} \int_S \phifix(x,y) d\mu^*(x,y) \\
    & = \inf_{\lambda \geq 0} \lambda\varepsilon+ \int_S \phifix(x,y) d\mu^*(x,y)
\end{align*}
as we wanted to prove. Now, we prove the opposite inequality.
    Note that, for every $(\lambda, h) \in \D$ we have
    \begin{align*}
        \int h(s,r) \ d\pi_{S|S=s} (r) d\mu^*(s) \geq \int \ell(r,G) - \lambda c(s,r) \ d\pi_{S|S=s} (r) d\mu^*(s),\quad \forall \pi_{S|S}: S \rightarrow K_s.
    \end{align*}
     where in this case $K_{(x,y)}=\{\mu=\delta_x \otimes \mu_{Y|X=x},\ \mu_{Y|X} \in P(Y|X) \}$.
     Using this, we obtain that 
    \begin{align*}
        \inf_{(\lambda, h) \in \D} \lambda \varepsilon &+ \sup_{\pi_{S|S} : S \rightarrow K_s} \int_S \int_X \int_Y h(x,y,u,v)\ d\pi_{S|S=s} (r) d\mu^*(s) \\
        &\geq \inf_{(\lambda, h) \in \D} \lambda \varepsilon  + \sup_{\pi_{S|S} : S \rightarrow K_s}\int_S \int_X  \int_Y \ell(r,g) - \lambda c(s,r) d\pi_{S|S=s} (r) d\mu^*(s).
    \end{align*}
    Now, let us note that
    \begin{align*}
        \{(x,y) \mapsto \delta_x \otimes \delta_{\varphi(x,y,u)} : \varphi : S \times X \rightarrow Y \text{ measurable}\} \subset \{\pi_{S|S} : S \rightarrow K_s\}.
    \end{align*}
    Therefore, continuing the previous estimate, we get that 
     \begin{align*}
        \inf_{(\lambda, h) \in \D} \lambda \varepsilon &+ \sup_{\varphi:S \times X \rightarrow Y} \int_S \int_X \int_Y h(x,y,u,v)\ d\pi_{S|S=s} (r) d \mu^*(s) \\
        & \geq \inf_{(\lambda, h) \in \D} \lambda \varepsilon  + \sup_{\varphi :S \times X \rightarrow Y}\int_S  \ell(x,\varphi(s,x),G) - \lambda c(s,x,\varphi(s,x)) d \mu^*(x,y)\\
        & = \inf_{(\lambda, h) \in \D} \lambda \varepsilon  + \int_S   \sup_{\varphi :S\times X  \rightarrow Y} [\ell(x,\varphi(s,x),G) - \lambda c(s,x,\varphi(s,x))] d\mu^*(x,y)\\
        & = \inf_{(\lambda, h) \in \D} \lambda \varepsilon  + \int_S   \sup_{v \in Y} [\ell(x,v,G) - \lambda c(x,y,x,v)] d\mu^*(x,y)
        \end{align*}
    as we wanted to prove. Note that the first equality above is a consequence of the measurable selection principle (see, for example, \cite[Theorem 3A]{rockafellar_integral_1976}), since the loss is continuous.
\end{proof}

\begin{theorem}[Strong duality for perturbations in $P(X)$]\label{th:k=x}
Let $K_{(x,y)} = \{\mu_X \otimes \mu^*_{Y|X = x} : \mu_X \in P(X)\}$ and suppose all assumptions in Theorem~\ref{thm:w-padro-strong-duality} hold. Then

    \begin{align*}
        I_{\Okc} (G) &= \inf_{\lambda \geq 0} \lambda \varepsilon +  \int_{S} \phifix(x,y) d\mu^*(x,y), \\
        \text{where} \quad \phifix (x,y)&:= \sup_{u \in X} \left(\int_Y \ell (u, v, G) - \lambda c((x,y),(u,v))\ d\mu_{Y|X=u}^*(v) \right).
    \end{align*}
\end{theorem}

The proof is given in Appendix \ref{supp:p(x)-duality-proof}. The argument mirrors that of Theorem~\ref{th:k=Y|X}, the one difference being that the dependence of $\mu_{Y|X=u}^*$ on $u\in X$ prevents the supremum from being taken inside the integral.

One might hope to obtain explicit expressions for the dual potentials $h$ for any arbitrary admissible set $K_s$. However, the proofs in Theorems~\ref{th:k=p(s)}, \ref{th:k=Y|X}, and \ref{th:k=x}, rely on the assumption that $K$ is convex and contains the relevant Dirac measures, either on $X, Y$, or on $X\times Y$. Since the convex envelope of the set of Dirac measures $\{\delta_x:x \in X\}$ is already the entire probability space $P(X)$, any $K_s$ satisfying these requirements must be large enough to include $P(X)$ (or $P(Y)$ or $P(X\times Y)$, depending on the perturbation set). 

\section{Explicit Dual Representations and Connections to Regularization}\label{sec:explicit-forms}
The dual formulations in the previous section are not yet explicit enough to be interpretable or to reveal connections to classical regularization. We now derive explicit upper bounds on the worst-case risk for each perturbation space, assuming a quadratic loss and cost with a Lipschitz inverse operator. Throughout this section let $S:=X\times Y$ with $X=\R^n$ and $Y=\R^m$, let $\mathcal{G} = \text{Lip}(Y;X)$ with Lipschitz constant $L_G$, and define the loss $\ell(x,y,G):=\|G(y)-x\|^2_2$ and joint transportation cost $c_{X\times Y}((x,y),(u,v)):= \|(u,v)-(x,y)\|^2_2$ (see Remark~\ref{rem:strong-duality} for a note on the compactness of $S$). The bounds connect to Tikhonov-type regularization and, for structured perturbation sets, reveal how the perturbation class shapes the data-fidelity and regularization terms. 

\subsection{Perturbations in $P(X\times Y)$}\label{sec:joint-perturb}
The result below provides an explicit upper bound for the worst-case risk under Wasserstein perturbations when all joint distributions are admissible (i.e., $K_s=P(S)$). We will refer to this model as joint DRO, but note that the dual in Theorem \ref{th:k=p(s)} recovered the standard DRO formulation.

\begin{theorem}[Explicit form for perturbations in $P(S)$]\label{th:explicit-p(s)}
In addition to the assumptions in Theorem~\ref{thm:w-padro-strong-duality}, take $K_s=P(S)$ for every $s \in S$. Then
    \begin{align*}
        \inf_{G \in \mathcal{G}} I_{\Okc} (G) \leq \inf_{G \in \mathcal{G}} \left(\sqrt{\int_S \|G(y)-x\|^2_2\ d\mu^*(x,y)} + \sqrt{\varepsilon(1+L_G^2)} \right)^2.
    \end{align*}
\end{theorem}
\begin{proof}
First, apply Theorems~\ref{thm:w-padro-strong-duality} and \ref{th:k=p(s)} to obtain:
\begin{align}
    I_{\Okc} (G)&= \inf_{\lambda \geq 0} \lambda \varepsilon + \int_S \phifix (x,y)\ d\mu^*(x, y) \label{eq:dual-p(s)}, \\
    \text{where} \quad \phifix (x,y) &= \sup_{(u,v) \in S} \Big( \|G(v)-u\|^2 - \lambda(\|(u-x,v-y)\|^2)\Big).
\end{align}
Write $r:=G(y)-x$ and $(\Delta_u, \Delta_v)= (u-x,v-y)$. By the triangle inequality, Lipschitz condition, and Cauchy-Schwarz,
\begin{align*}
    \|G(v)-u\|_2 &= \|r+(G(v)-G(y))-\Delta_u\|_2  \\
    &\leq \|r\|_2 + L_G\|\Delta_v\|_2 + \|\Delta_u\|_2 \leq \|r\|_2 + \sqrt{1+L_G^2} \cdot \|(\Delta_u,\Delta_v)\|_2.
\end{align*}
Writing $t := \|(\Delta_u,\Delta_v)\|_2 \geq 0$ and $L := 1 + L_G^2$, the supremum is bounded by $\sup_{t \geq 0}((\|r\|_2 + \sqrt{L}\,t)^2 - \lambda t^2)$, which is finite for $\lambda > L$ and gives
\begin{align}\label{eq:phi-p(s)-bound}
    \phifix(x,y) \leq \|r\|_2^2 + \frac{L}{\lambda - L}\|r\|_2^2 = \frac{\lambda}{\lambda - L}\|r\|_2^2.
\end{align}
Substituting \eqref{eq:phi-p(s)-bound} into \eqref{eq:dual-p(s)} with $R := \int_S \|G(y)-x\|_2^2\, d\mu^*$ gives $I_{\Okc}(G) \leq \inf_{\lambda > L} \lambda\varepsilon + \frac{\lambda}{\lambda-L}R$. The first-order condition yields the minimizer $\lambda^* = L + \sqrt{LR/\varepsilon}$, and substituting it back gives $I_{\Okc}(G) \leq \left(\sqrt{R} + \sqrt{\varepsilon(1+L_G^2)}\right)^2$. Taking the infimum over $G$ on both sides yields the result.
\end{proof}

\subsection{Perturbations in $P(Y|X)$}\label{sec:reg-y|x}
We now find an explicit form for perturbations that affect only the conditional distribution of the measurement variable $Y$ given the input $X$.
   
\begin{theorem}[Explicit form for perturbations in $P(Y|X)$]\label{th:explicit-p(Y|X)}
Under the assumptions of Theorem \ref{th:k=Y|X}, 
\begin{align*}
    \inf_{G \in \mathcal{G}} &I_{\Okc}(G) \leq \inf_{G \in \mathcal{G}} \left(\sqrt{\int_S \|G(y)-x\|^2_2\ \,d\mu_{}^*(x,y)} 
    + L_G\sqrt{\varepsilon} \right)^2.
\end{align*}
\end{theorem}
\begin{proof}
For our choices of loss and cost, the dual representation of $I_{\Okc}(G)$ is
\begin{align}
        I_{\Okc} (G) &= \inf_{\lambda \geq 0} \lambda \varepsilon + \int_{S} \phifix(x,y) d\mu^*(x,y), \label{eq:I-Y|X} \\
        \text{where} \quad \phifix(x,y) &= \sup_{v \in Y} (\|G(v)-x\|^2 -\lambda\|y-v\|^2)
    \end{align}
by Theorems~\ref{thm:w-padro-strong-duality} and \ref{th:k=Y|X}. Introduce the variable change $\Delta := v - y$ and $r := G(y) - x$. Let $\delta_\Delta := G(v) - G(y)$, so that $\|\delta_\Delta\| \leq L_G\|\Delta\|$ by the Lipschitz condition. Then
\begin{align*}
    \|G(v) - x\|^2\  
    &= \|r + \delta_\Delta\|^2  
    =  \|r\|^2\ 
    + 2 r^T \delta_\Delta 
    + \|\delta_\Delta\|^2.
\end{align*}
Since $\|\delta_\Delta\|^2 \leq L_G^2\|\Delta\|^2$, we bound the supremum over $v$ by 
optimizing over $\delta_\Delta \in \R^n$ and $\Delta \in Y$ subject to $\|\delta_\Delta\| \leq L_G\|\Delta\|$, which gives
\begin{align*}
    \sup_{\Delta \in Y,\, \|\delta_\Delta\| \leq L_G\|\Delta\|}\left\{ \|r\|^2  + 2r^T\!\delta_\Delta  + \|\delta_\Delta\|^2 - \lambda\|\Delta\|^2 \right\}.
\end{align*}
Substituting $\delta_\Delta = L_G \Delta \cdot \hat{n}$ for the optimal unit direction $\hat{n}=r/\|r\|$ and maximizing the resulting concave quadratic in $\|\Delta\|$, the supremum is finite for $\lambda > L_G^2$ and yields
\begin{align}\label{eq:phi-Y|X-bound}
    \phifix(x,y) \leq \|r\|^2\  + \frac{L_G^2}{\lambda - L_G^2} \left\| r\right\|^2.
\end{align}
Substituting \eqref{eq:phi-Y|X-bound} into \eqref{eq:I-Y|X} and writing $R := \int_S \|G(y)-x\|^2\, d\mu^*$ gives
\begin{align}\label{eq:bound-y|x-lambda}
    I_{\Okc}(G) \leq \inf_{\lambda > L_G^2}\ \lambda\varepsilon + R + \frac{L_G^2}{\lambda - L_G^2} R.
\end{align}
The first-order condition $\varepsilon = L_G^2 R/(\lambda - L_G^2)^2$ gives the minimizer $\lambda^* = L_G^2 + L_G\sqrt{R}/\sqrt{\varepsilon}$, and substituting it into \eqref{eq:bound-y|x-lambda} yields $I_{\Okc}(G) \leq (\sqrt{R} + L_G\sqrt{\varepsilon})^2$. Taking the infimum over $G$ on both sides yields the stated result.
\end{proof}

In the bound of Theorem~\ref{th:explicit-p(Y|X)}, we recognize a data-fidelity term together with a Tikhonov-type penalty: it penalizes reconstructions with a large Lipschitz constant $L_G$, and scales with the Wasserstein radius $\varepsilon$. Thus, it suppresses the directions along which the inverse most amplifies noise. For a linear inverse $G$, this would correspond to penalizing its largest singular value.

Comparing with the joint bound of Theorem~\ref{th:explicit-p(s)}, restricting perturbations to $P(Y|X)$ replaces the factor $\sqrt{1+L_G^2}$ by $L_G$. This removes the unit contribution that, in the joint case, corresponds to the adversary's freedom to transport mass in the $X$-direction. Fixing the $X$-marginal suppresses this component of the penalty. When $L_G\gg 1$ and the bounds use the same cost and loss functions, the two bounds will converge. Otherwise, the $P(Y|X)$ bound will lie below the joint one. 

The same form arises in \cite{blanchet_robust_2019}, who enforce the conditional structure in a different manner. Instead of restricting the admissible set, they keep $K_s=P(S)$ and instead modify the transportation cost to
\begin{align}\label{eq:cond-cost}
    c_{Y|X}((x,y), (u,v)) = \begin{cases}
        \|y - v\|^2_2, &\text{ if } x=u,  \\
        \infty, &\text{ else.}
    \end{cases}
\end{align}
which assigns infinite cost to transport that perturbs the input, thus allowing perturbations only in the measurement space. The two formulations impose the same constraint in different ways. Our method imposes it at the level of the ambiguity set, removing the inadmissible measures from $K_s$, while \cite{blanchet_robust_2019} imposes it at the level of the cost, keeping these measures in $K_s$ but making them infeasible. Repeating the proof of Theorem~\ref{th:explicit-p(s)} with the cost~\eqref{eq:cond-cost} in place of $c_{X\times Y}$ yields the same bound as in Theorem~\ref{th:explicit-p(Y|X)}. Thus, the two formulations recover the same explicit upper bound on the worst-case risk. Reducing $X$ and $Y$ to $\R^n$ and $\R$ and assuming $G$ to be linear, turns the inequalities in the proof into equalities and recovers the exact estimator of \cite{blanchet_robust_2019}, an $L_2$-regularized form of the mean-squared-error estimator. Replacing the squared norm by an $L_p$-cost generalizes this to $L_q$ regularization with $q\in(1,\infty]$ and $1/p+1/q=1$. In this sense ,our structured framework contains the standard Wasserstein DRO formulation as the special case $K_s=P(S)$ with a conditional norm-based cost.

This observation is not surprising. For $P(Y|X)$ perturbations, it is easy to check that our ambiguity set simply constrains perturbations in the $\varepsilon$-ball of the conditional Wasserstein distance between $\mu_X^* \otimes \mu^*_{Y|X}$ and $\mu_X^* \otimes \mu_{Y|X}$ \cite{chemseddine_conditional_2025}. On the same line, the corollary below shows that the cost~\eqref{eq:cond-cost} induces a conditional Wasserstein distance, which clarifies in what sense it enforces the conditional constraint.

\begin{corollary}\label{cor:w-cond}
Let the cost function $c_{Y|X}$ be defined as above and $X \subset \R^n$ and $Y \subset \R^m$ to be compact sets. Then the corresponding Wasserstein distance on $X \times Y$ coincides with the \emph{conditional Wasserstein distance}
\begin{align*}
    W_{\rm cond}(\mu^*, \mu) := \int_X W(\mu_{Y|X=x}^*, \mu_{Y|X=x}) d\mu_X^*(x)
\end{align*}
with cost $c_Y(y,v):=\|y-v\|^2_2$.  
More precisely, defining  $W_{c_{Y|X}}(\mu^*, \mu)$ as the  Wasserstein distance with cost $c_{Y|X}$ we have:
\begin{enumerate}[label=\roman*)]
    \item $W_{c_{Y|X}}(\mu^*, \mu)< \infty$ if and only if $\mu_X^* = \mu_X$;
    \item in this case, $W_{c_{Y|X}}(\mu^*, \mu) = W_{\rm cond}(\mu^*, \mu)$.
\end{enumerate}
\end{corollary}
The proof can be found in Appendix \ref{sec:app-cond-wass}.

By part~(i), a finite transport cost forces $\mu_X=\mu_X^*$, so the infinite-cost in \cite{blanchet_robust_2019} fixes the $X$-marginal in the same way that we do directly through $K_s$. By part~(ii), the resulting distance is the conditional Wasserstein distance $W_\text{cond}$, which compares the conditionals $\mu_{Y|X}$ for each $x$ separately and averages over the fixed input distribution. The two formulations thus fix the same marginal and give the same explicit upper bound, but use different distances on the admissible conditional distributions. We remind the reader that for any $\mu^*,\mu\in P(X\times Y)$ with a common $X$-marginal it holds that $W(\mu_{X\times Y}^* ,\mu_{X\times Y}) \leq W_\text{cond} (\mu_{X\times Y}^* ,\mu_{X\times Y})$, where the inequality can be strict \cite{chemseddine_conditional_2025}. 

\subsection{Perturbations in $P(X)$ and $P(X|Y)$}
We consider perturbations in $P(X)$ by considering $K_{(x,y)}=\{\mu =\mu_X\otimes \mu_{Y|X}^*,\ \mu_X \in P(X)\}$. By Theorems~\ref{thm:w-padro-strong-duality} and~\ref{th:k=x}, the dual potential takes the form $\phifix (x,y) = \sup_{u \in X} \int_Y \ell(u,v,G) - \lambda c((x,y),(u,v))\ d\mu^*_{Y|X=u}(v)$. Unlike the $P(Y|X)$ setting, both the loss and the measure $\mu_{Y|X}^*$ depend on $u$, so the supremum does not reduce to a concave quadratic in a single variable. Consequently, no tractable closed form for $I_{\Okc}$ appears obtainable without additional assumptions on the family $\{\mu_{Y|X=u}^*\}_{u \in X}$.

Since no closed form is available for $P(X)$-perturbations, one could instead consider perturbations in $P(X|Y)$ by considering the ambiguity set $K_{(x,y)}=\{\mu =\delta_y \otimes \mu_{X|Y=y},\ \mu_{X|Y} \in P(X|Y)\}$. By Theorems~\ref{thm:w-padro-strong-duality} and~\ref{th:k=Y|X}, the dual potential is $\phifix(x,y)=\sup_{u \in X}\{\|G(y)-u\|^2-\lambda\|x-u\|^2\}$. In contrast to the $P(X)$ case, fixing the $Y$-marginal keeps $y$ and $G(y)$ fixed inside the supremum, leaving no integral over a conditional, so it reduces to a concave quadratic in $u$ for $\lambda>1$. Since the perturbed variable is now directly present in the loss rather than through $G$, the argument of Theorem~\ref{th:explicit-p(Y|X)} carries over with $L_G$ replaced by $1$, yielding
\begin{align*}
   \inf_{G \in \mathcal{G}} I_{\Okc}(G) \leq \inf_{G \in \mathcal{G}}
   \left( \sqrt{\int_S \|G(y)-x\|^2\, d\mu^*}
   + \sqrt{\varepsilon} \right)^2.
\end{align*}
The bound mirrors Theorem~\ref{th:explicit-p(Y|X)} but the regularization term carries no factor of $L_G$ yielding a weaker, structurally uniform penalty.

\section{Comparison of DRO Methods}\label{sec:comparison}
We validate the findings of Section~\ref{sec:reg-y|x} on two inverse problems of increasing ill-posedness. For each we fix an approximate inverse $G$ and study how the joint DRO bound (Theorem~\ref{th:explicit-p(s)}) and the $P(Y|X)$ bound (Theorem~\ref{th:explicit-p(Y|X)}) behave as the Lipschitz constant $L_G$ grows:
\begin{align*}
    \text{joint:}\quad \Big(F + \sqrt{\varepsilon(1+L_G^2)}\Big)^2, \qquad
    P(Y|X):\quad \Big(F + L_G\sqrt{\varepsilon}\Big)^2,
\end{align*}
where $F = \sqrt{\int\|G(y)-x\|^2\,d\mu^*}$ is the shared fidelity term, so that the two bounds differ only in the regularization.

\subsection{Problem Setup}
In both experiments we draw $N_x = 500$ samples for $x$ and, for each $x$, $N_u = 150$ samples of $u$, all drawn from a dataset of $5000$ points generated by $\mu_X^* = \mathcal{N}(\bar x, \sigma_X^2 I)$ with $\sigma_X = 0.1$ and $\bar x(t) = \sin(3\pi t)$ on the grid of each problem. Measurements are noiseless, $y = H(x)$ with $H : \mathbb{R}^{100} \to \mathbb{R}^{100}$, so that $F$ reflects only the misfit between $G$ and the true inverse, isolating the effect of $L_G$. We set $\varepsilon = 1.0$ and compare the bounds through $\text{gap}(\%) = 100 \cdot (\sqrt{\text{joint}}-\sqrt{P(Y|X)})/\sqrt{\text{joint}} = 100 \cdot \sqrt{\varepsilon}\left(\sqrt{1+L_G^2} - L_G\right)/(F + \sqrt{\varepsilon(1+L_G^2)}),$ which we study not by varying $L_G$ directly but through a parameter controlling each problem's ill-posedness.


The \emph{diagonal operator} $(Hx)_k = \sigma_k x_k$, $\sigma_k = (1+k)^{-p}$, $k = 0, \ldots, 99$, has polynomially decaying singular values and exact inverse $G(y) = y/\sigma_k$. The reconstruction is then perfect, so $F = 0$, and $L_G = \max_k \sigma_k^{-1} = 100^p$. We sweep $p \in [0.01, 1]$. The \emph{backward heat equation} on $t \in [-1,1]$ at diffusion time $T_\mathrm{diff} = 0.05$ has a forward operator
\begin{align*}
    (Hx)(t_k) = (x * \mathcal{K})(t_k)\,\Delta t, \quad \mathcal{K}(t) = \frac{1}{2\sqrt{\pi T_\mathrm{diff}}} \exp\left(-\frac{t^2}{4T_\mathrm{diff}}\right), 
\end{align*}
and the signal is sampled on a uniform grid of $100$ points. The approximate inverse is the Tikhonov-regularized Fourier deconvolution $G(y) = \mathcal{F}^{-1}\big(\hat{\mathcal{K}}\,\hat{y}/(\lvert\hat{\mathcal{K}}\rvert^2 + \alpha)\big)$, $\alpha = 5\times 10^{-4}$, as approximate inverse. Here the regularization makes $G$ inexact, so $F > 0$ even without noise, and $L_G = \tfrac{1}{2\sqrt{\alpha}}$. We sweep $\alpha \in [10^{-5}, 10^{-1}]$.

\subsection{Results}
\begin{figure}[tp]
    \centering
    \includegraphics[width=0.8\linewidth]{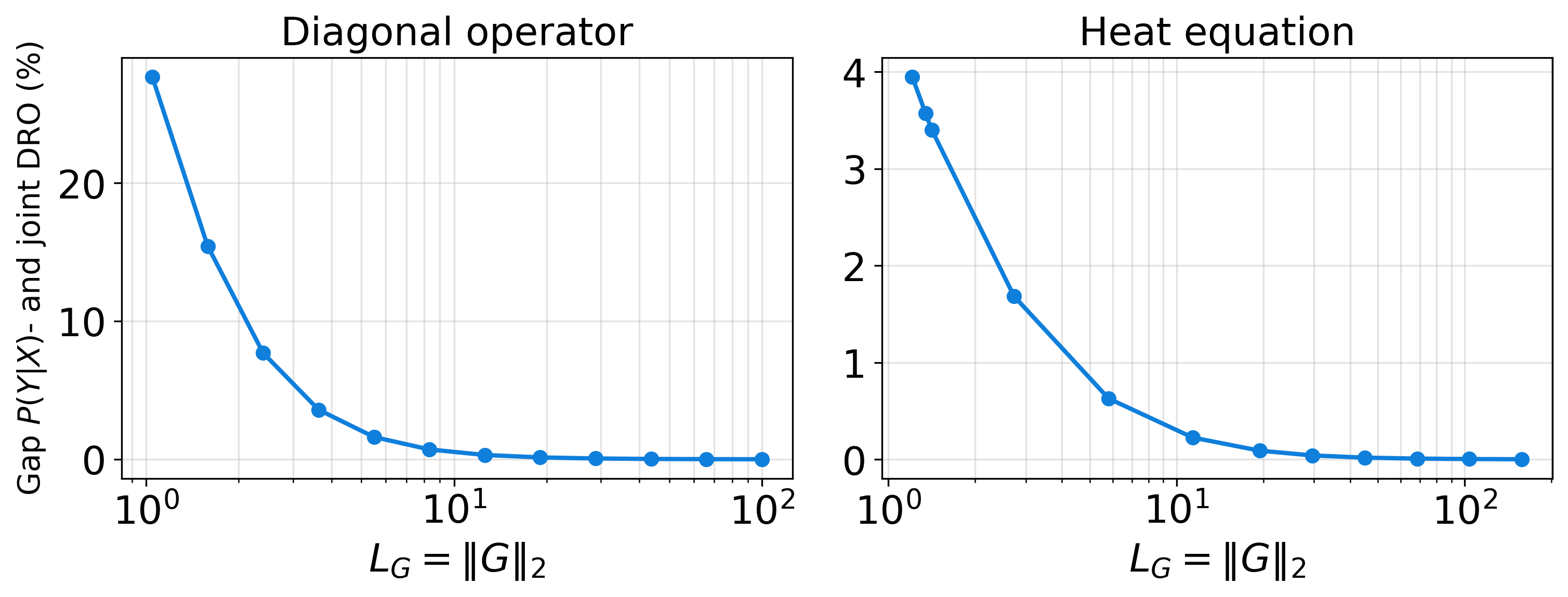}
    \caption{Percentage gap between the joint and $P(Y|X)$-DRO bounds as the Lipschitz constant $L_G$ of the fixed inverse grows, for the diagonal operator (left) and the backward heat equation (right). The gap is largest at small $L_G$ and closes monotonically, since the $P(Y|X)$ bound replaces $\sqrt{1+L_G^2}$ by $L_G$. The two problems start at nearly the same $L_G$ but different gaps ($27.6\%$ vs.\ $3.9\%$) because the heat-equation inverse has nonzero fidelity $F$, which enters only the denominator.}
    \label{fig:gap}
\end{figure}

Figure~\ref{fig:gap} shows the gap for both problems. It is largest at small $L_G$ and closes monotonically as $L_G$ grows, since the $P(Y|X)$ bound replaces $\sqrt{1+L_G^2}$ by $L_G$. The diagonal operator starts at $27.6\%$ and the heat equation at $3.9\%$, even though both begin at nearly the same $L_G$ ($1.05$ and $1.21$). The difference comes from the fidelity $F$, which is only present in the denominator: the exact diagonal inverse has $F = 0$ and attains the maximal gap at that $L_G$, whereas the heat-equation inverse has a large $F$ that dilutes the same absolute regularization difference. The extra conservatism of joint DRO is thus due entirely to the $+1$ in $\sqrt{1+L_G^2}$, significant only when $L_G$ is small. Once $L_G$ is large, both terms are dominated by $L_G\sqrt{\varepsilon}$ and the bounds nearly coincide. Since $L_G\sqrt{\varepsilon} \le \sqrt{\varepsilon(1+L_G^2)}$ for every $L_G$, $P(Y|X)$-DRO is never larger than joint DRO and strictly smaller at finite $L_G$, most noticeably for well-posed problems. For severely ill-posed problems or large reconstruction error, it adds little.
\section{Applications to Inverse Problems}
In this section, we show examples of implementations of the framework across inverse problems of increasing ill-posedness, as well as a robust approximation/simulator of an operator to demonstrate the flexibility of the framework beyond inverse problems. All examples consider perturbations in $P(Y|X)$, except the robust simulator, which perturbs $P(X|Y)$. In all experiments, out-of-distribution performance is evaluated by testing on distributional shifts unseen during training. An additional example validating robustness in $P(X\times Y)$ is provided in Appendix \ref{supp:additional-examples}. The code will be available upon publication.

\subsection{Computational Framework} In this subsection, we describe the framework used to implement structured DRO for inverse problems.

\subsubsection{Model}\label{sec:s-duality}
To address computational challenges and improve optimization properties, we consider the entropy-regularized Wasserstein distance $W^\delta$ \cite{peyre_computational_2019, cuturi_sinkhorn_2013} (also known as the Sinkhorn distance) defined as 
\begin{align*}
    W^\delta (\mu_1, \mu_2) := \inf_{\pi \in \Pi(\mu_1, \mu_2)} \int_{S \times S} c(s,r) + \delta \log \left( \frac{d\pi}{d\eta} \right) d\pi(s,r) 
\end{align*}
where $\delta > 0$ is a regularization parameter, $\mu_1,\mu_2 \in P(S)$ and $\eta \in P(S \times S)$ a chosen reference measure such that $\eta = \eta_1\otimes \eta_2$ and $\mu_1 \ll \eta_1,\ \mu_2\ll \eta_2$.  

We approximate the structured $P(Y|X)$-DRO objective by the conditional Sinkhorn DRO of \cite{wang_sinkhorn_2026}. Concretely, we use the Sinkhorn distance with cost $c=c_{Y|X}$ defined in \eqref{eq:cond-cost}. As outlined in Section \ref{sec:reg-y|x}, this is equivalent to a Sinkhorn regularization of the $P(Y|X)$-DRO.

We choose the reference measure $\eta(s,r)=\mu^*(s)\otimes\eta_{S|S=s}(r)$ so that the primal problem becomes
\begin{align*}
    \Ocd &:= \Big\{\pi_{S|S=s} \in P(S) \text{ for } \mu^*\text{-a.e } s\in S:\\
    &\qquad\int_S \int_S c_{Y|X}(s,r) + \delta \log\left(\frac{d\pi_{S|S=s}}{d\eta_{S|S=s}}\right) d\pi_{S|S=s}(r) d\mu^*(s) \leq \varepsilon\Big\}\text{ and }  \\
    I(\pi_{S|S}, G)&:=\int_S\int_S\ell(r,G)\ d\pi_{S|S=s}(r)d\mu^*(s), \quad I= \min_{G \in \mathcal{G}} \min_{\pi_{S|S} \in \Ocd} I(\pi_{S|S},G).
\end{align*}
The loss is given by the mean squared error between the target and the predicted reconstruction, i.e. $\ell(x,y;G)=\|x-G(y)\|^2_2$.  

We further specify the reference measure as $\eta_{S|S=(x,y)}(u,v)=\mu_X^*(u)\otimes\eta_{Y|X=x}(v)$ with $\eta_{Y|X=x}$ a Lebesgue measure. By \cite{wang_sinkhorn_2026}, this reduces to the dual formulation 
\begin{align}\label{eq:sinkhorn-dual}
    I & = \inf_{G\in \mathcal{G}} \inf_{\lambda \geq 0} \lambda \hat \varepsilon \nonumber \\
    & \qquad +  \lambda\delta\int_{S} \log \left[ \int_X\int_Y \exp\left(\frac{\|u-G(v)\|^2_2}{\lambda \delta}\right) d\mathcal{N}(Hx, \delta^2I)(v)\ d\mu_X^*(u)\right] d\mu^*(x,y)
\end{align}
with $\hat \varepsilon = \varepsilon + \delta \int\log\int\exp(-c(s,r)/\delta)\ d\eta_{S|S=s} (r) d\mu^*(s)$ an adjusted Wasserstein radius.

\subsubsection{Algorithm}
We follow \cite{wang_sinkhorn_2026} to compute the Sinkhorn DRO objective, with the optimal $\lambda$ determined via bisection search. At each candidate $\lambda$, the corresponding optimal $G$ is obtained using a Biased Stochastic Mirror Descent (BSMD) algorithm with gradient updates approximated by a stochastic gradient (SG) estimator. The bisection interval is updated by comparing objective values at the midpoint and boundaries, and terminates once the interval width falls below a pre-specified tolerance.

\subsubsection{Experimental Setup}
Throughout, training data consists of noise-free paired samples $(x_i, y_i)\in X\times Y$ generated via a known forward operator $H$ as $y_i = Hx_i$. The models are trained using Adam and an initial learning rate $\gamma$. During bisection optimization, the learning rate remains constant. The final model is trained using a scheduler with linear warmup from $0.1\gamma$ to $\gamma$ over the first few epochs, followed by cosine annealing to $10^{-6}$. 

We compare against two baselines. For experiments with linear parameterizations, the baseline is a Tikhonov-regularized pseudo-inverse $T_\alpha = (H^TH + \alpha L^TL)^{-1}H^T$, where $L$ is the first-difference matrix and $\alpha > 0$ is selected by minimizing the root mean square error (RMSE) on a validation set. For nonlinear parameterizations, the baseline uses the same architecture as our method (multilayer perceptron (MLP) or UNet, depending on the experiment) but is trained with an MSE loss on noisy paired data $y_i = Hx_i + \text{noise}$, with noise distributions varying per problem.

Performance is evaluated on a test set with increasing noise levels using the \textbf{RMSE}. We further quantify model \textbf{stability} by measuring sensitivity of the output to small input perturbations: for a given measurement $y$, stability is calculated as $\frac{1}{D}\sum_{i=1}^D \|f(y + \delta y_i) - f(y)\|_2 /\|\delta y_i\|_2,\ \delta y_i \sim \mathcal{N}(0, \delta^2 I)$, where $f$ is the reconstruction model, $D$ is the number of perturbation directions, and $\delta$ controls the perturbation magnitude. A lower value indicates a more stable model, i.e., one whose output changes little under small input perturbations.

To assess interpretability, we use \textbf{saliency maps} \cite{simonyan_deep_2013}, which measure the local sensitivity of a model's output to its input. For a trained model $G$ and a data pair $(x, y)$, the saliency map is $S(y) = |\nabla_y\, \mathrm{Loss}(G(y), x)|$, where each entry $S_i(y)$ quantifies the effect of a small change in $y_i$ on the prediction.  

All experiments are evaluated over $20$ independent trials in which a random subset of test samples is drawn, and the RMSE and stability metrics are computed. We report the mean and standard deviation across trials, with shaded regions indicating $\pm 2$ standard deviations. For the CT experiment, $50$ samples are drawn per trial from the $750$ test samples. For MNIST, $200$ samples are drawn from $5000$.

\subsection{Differentiation}
\begin{figure}[t]
    \centering
    \includegraphics[width=0.8\linewidth]{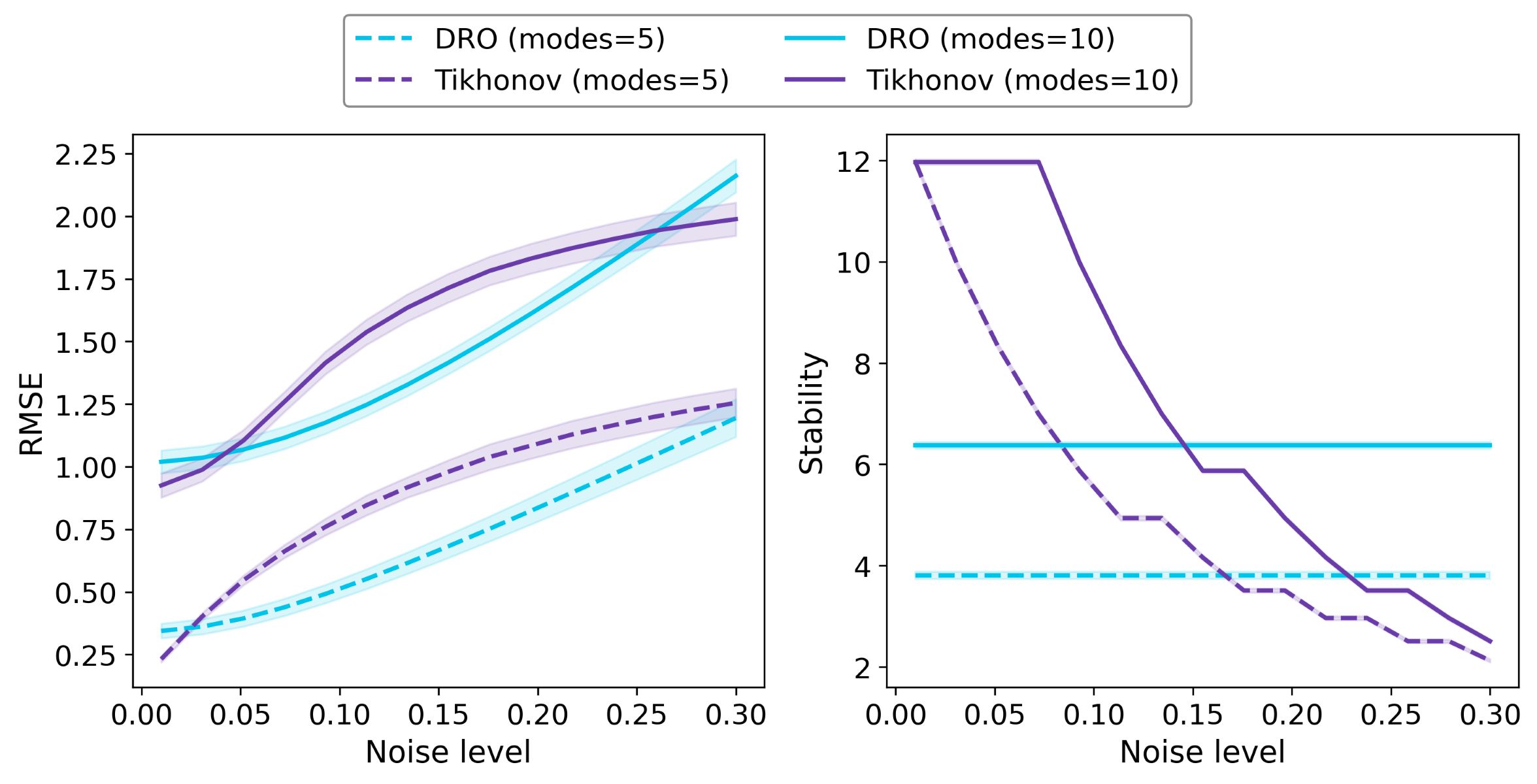}
    \caption{Differentiation: reconstruction metrics of the DRO inverse compared to the Tikhonov baseline at $M=5$ and $M=10$ modes. Left: RMSE versus noise level. Right: stability metric versus noise level, where a lower value means the output is less sensitive to small input perturbations. Solid lines show the mean over trials and shaded bands $\pm 2$ standard deviations. 
    }
    \label{fig:diff-metrics}
\end{figure}
\begin{figure}[t]
    \centering
    \includegraphics[width=0.8\linewidth]{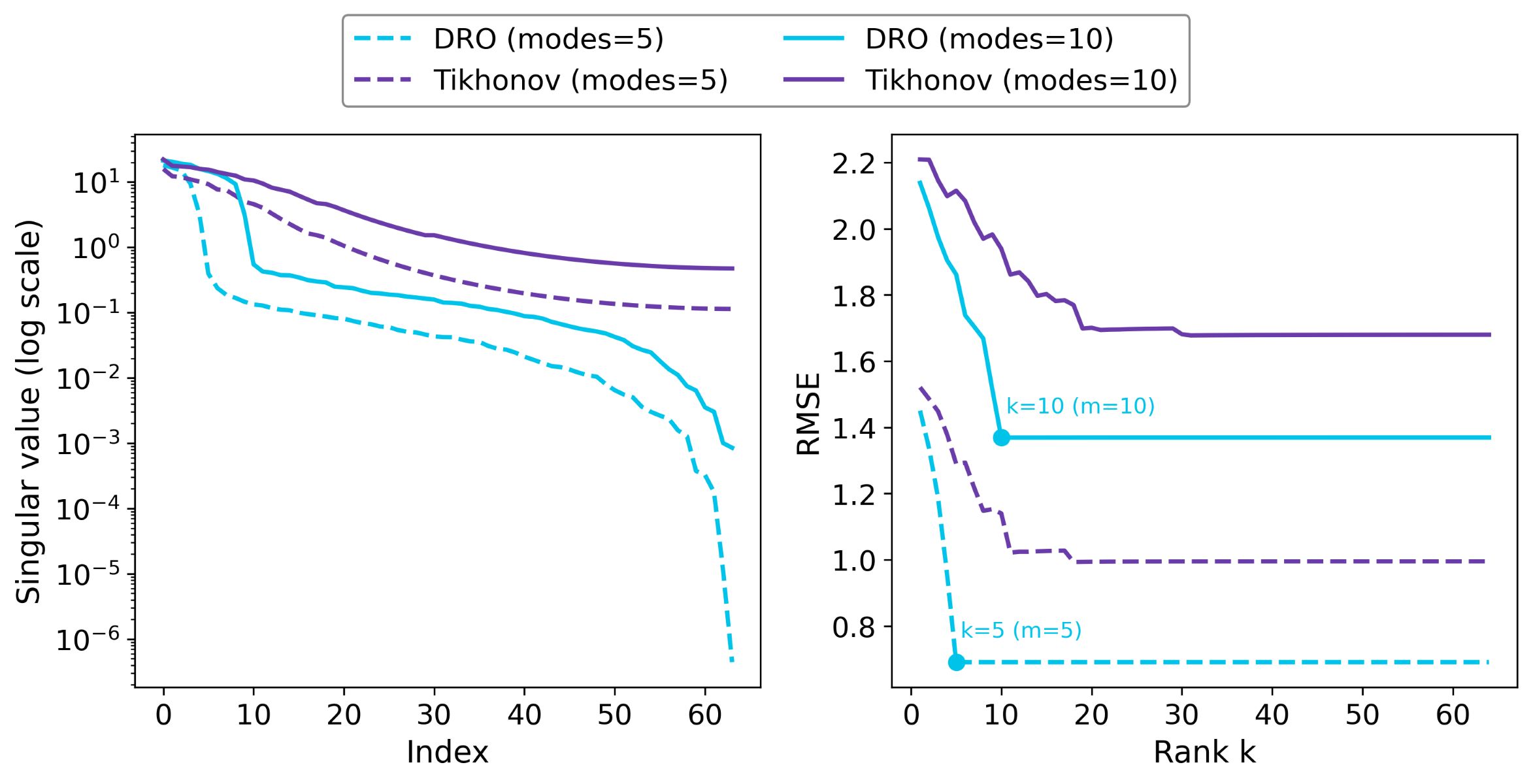}
    \caption{Spectral analysis of DRO and Tikhonov inverses. Left: singular values. Right: RMSE under rank-\textit{k} truncation.}
    \label{fig:diff-spectrum}
\end{figure}

As forward operator $H$ we adopt the cumulative integral $(Hx)(t_k) = \sum_{m=0}^k x(t_m)\,\Delta t$ with $\Delta t = \tfrac{1}{T-1}$, and generate the input as a truncated random Fourier sine series $x(t) = \sum_{k=1}^M c_k \sin(2\pi k t)$ with i.i.d.\ coefficients $c_k \sim \mathcal{N}(0,1)$, where $M$ is the number of sine modes (the higher, the more ill-posed). Sampling $x$ on a uniform grid of size $T$ yields $x \in \R^T$ and $H : \R^T \to \R^T$, a discrete cumulative sum. We use $M = 5$ and $10$ modes, feature length $T = 64$, and $N = 1000$ samples, and parameterize the inverse $g$ as a matrix. 

For the DRO formulation, we set the entropic regularization parameter to $\delta=0.01$ and the Wasserstein-radius to $\varepsilon=0.1$, with a constant learning rate of $\gamma=10^{-3}$ and bisection bounds $[0.1, 10^4]$ for $\lambda$. 

Figure \ref{fig:diff-metrics} shows the performance of the two DRO inverses against the Tikhonov baselines. All methods perform comparably in low-noise regimes, but the Tikhonov RMSE increases rapidly with noise, while the DRO inverse remains stable, as reflected in the stability plot. 

The left panel of Figure~\ref{fig:diff-spectrum} shows that the singular values of the DRO inverse drop sharply after the $k$-th largest singular value, which coincides with the number of modes. The right panel confirms that truncating beyond this rank does not further reduce the RMSE. This is the effect of the Lipschitz penalty in Theorem \ref{th:explicit-p(Y|X)}: the DRO objective penalizes the directions that most amplify worst-case perturbations, and since the data lives on a $M$-dimensional subspace, those are exactly the directions beyond rank $M$. The learned inverse therefore recovers the intrinsic dimension of the data, arriving at a data-driven analogue of truncated-SVD regularization. 

\subsection{Deblurring MNIST}\label{sec:mnist}
\begin{figure}[t]
    \centering
    \begin{subfigure}[t]{0.48\linewidth}
        \centering
        \includegraphics[height=7cm]{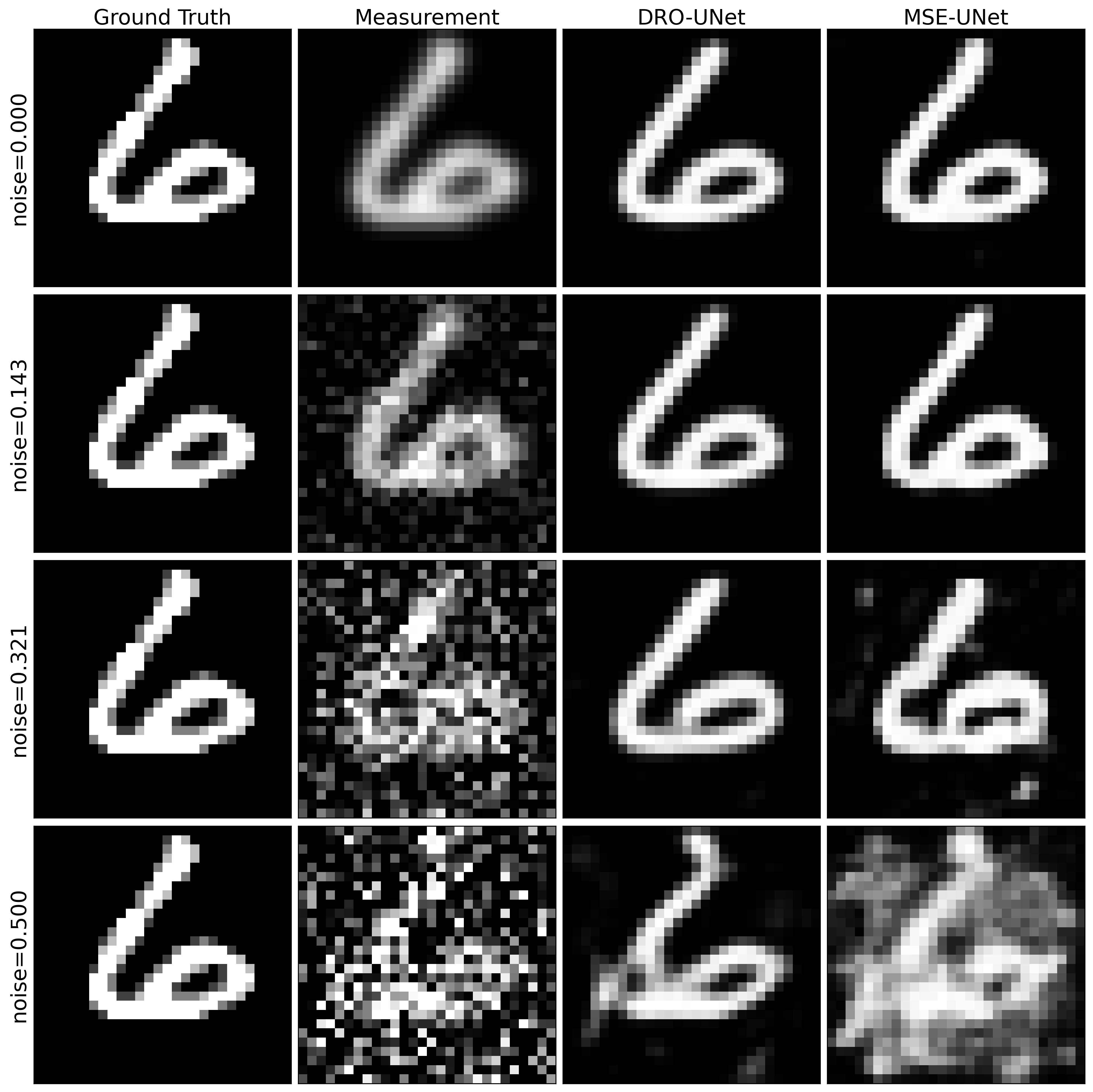}
    \end{subfigure} \hfill
    \begin{subfigure}[t]{0.48\linewidth}
        \centering
        \includegraphics[height=7cm]{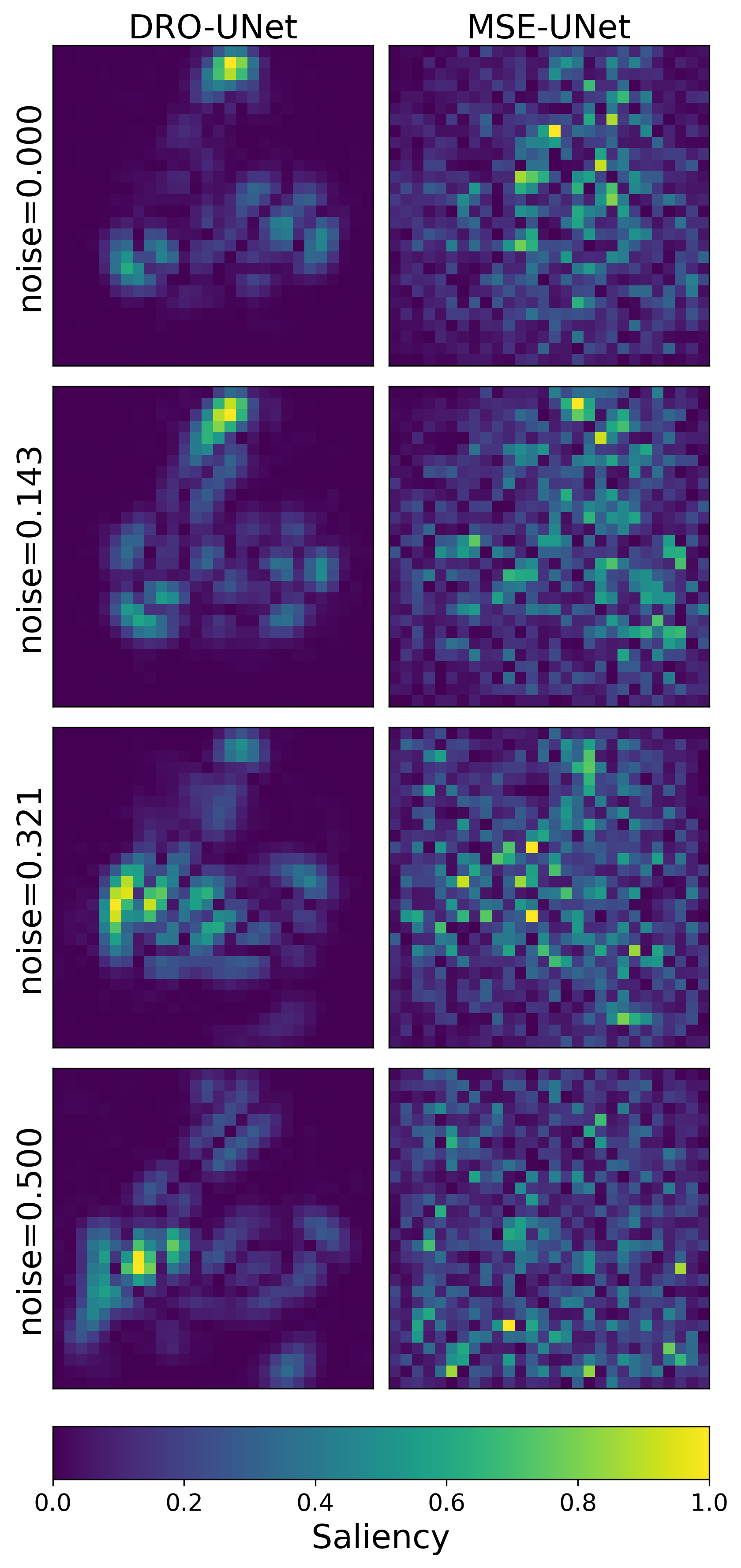}
    \end{subfigure}
    \caption{MNIST reconstructions (left) and saliency maps (right) under increasing Gaussian noise. Reconstruction columns (left to right): ground truth, noisy measurement, DRO-UNet, MSE-UNet. Rows correspond to increasing noise level (top to bottom). Higher saliency intensities indicate the input pixels the model is most sensitive to during reconstruction.}
    \label{fig:mnist_gaus_recon}
\end{figure}

\begin{figure}[t]
    \centering
    \includegraphics[width=0.8\linewidth]{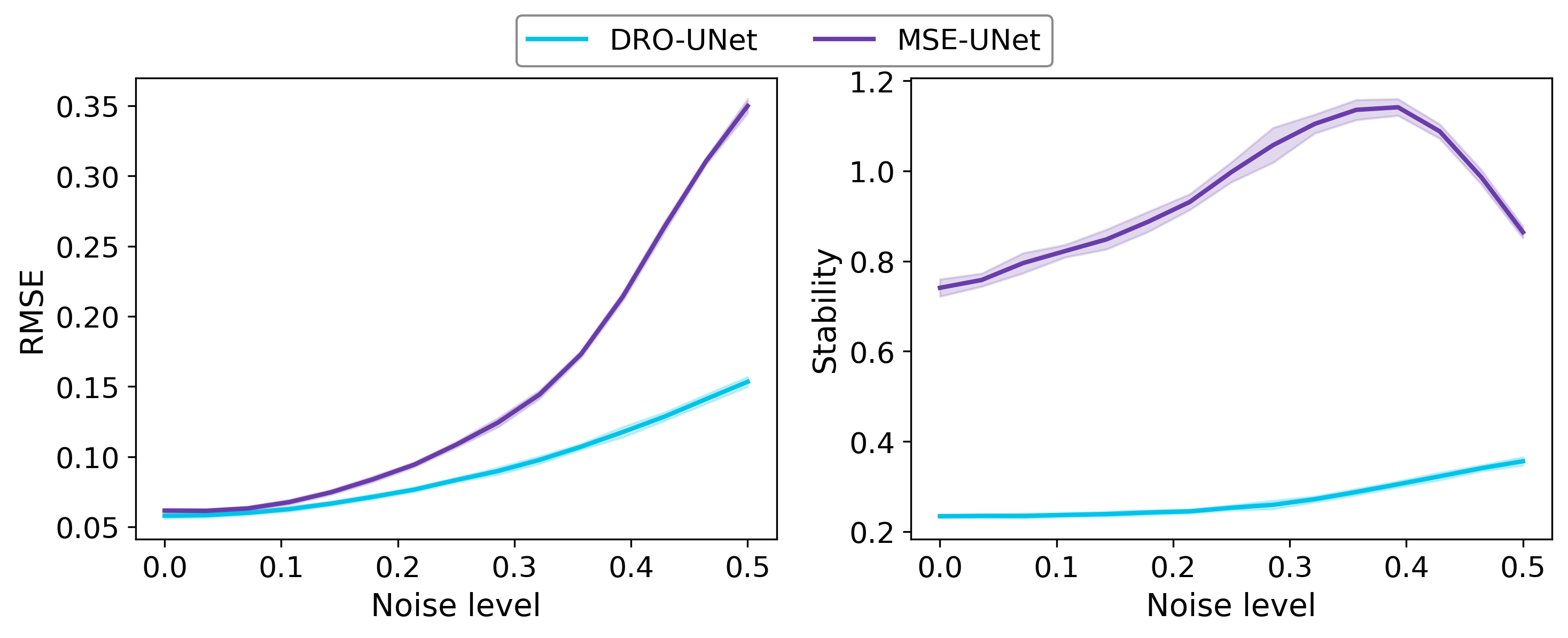}
    \caption{MNIST reconstruction metrics under increasing Gaussian noise with panels and shaded bands as in Figure~\ref{fig:diff-metrics}. }
    \label{fig:mnist_gaus_metrics}
\end{figure}
The input images $x \in \R^{28\times 28}$ are drawn from the MNIST dataset \cite{deng_mnist_2012}. As forward operator $H:L^2(\R^2) \to L^2(\R^2)$, we consider a convolution of an image $x(s,t)$ with a Gaussian kernel $\mathcal{K}$:
\begin{align*}
    Hx &= (x * \mathcal{K}_{k,\sigma})(s, t) = \int_{\R^2} x(s-\tau, t-\rho)\mathcal{K}_{k,\sigma}(\tau,\rho)\ d\tau\ d\rho, \\
   \text{where } \mathcal{K}_{k,\sigma}(\tau, \rho)&=\frac{1}{2\pi\sigma^2}\exp\left(-\frac{\tau^2+\rho^2}{2\sigma^2} \right).
\end{align*}
We discretize this using \verb|torch.nn.functional.conv2d| with \verb|same| padding and a $k\times k$ kernel $\mathcal{K}^\text{discrete} \in \R^{k\times k}$, giving $H^\text{discrete}: \R^{28\times 28} \to \R^{28\times 28}$. 

We set $\sigma^2=1$ and $k=5$, and parameterize the inverse as a UNet of depth $2$ with $64$ base channels. For the DRO formulation, we set $\delta=0.05$ and $\varepsilon=1.0$ with learning rate $\gamma=10^{-3}$, bisection bounds $[10^{-2}, 10]$ and tolerance $0.05$. The SG estimator uses $5$ samples per iteration during training and $6$ during validation. The DRO-UNet is trained on noise-free data while the MSE-UNet is trained on corrupted measurements $y_i = Hx_i + \eta$, $\eta\sim\mathcal{N}(0,0.15^2)$. We use $25000$ training pairs, $5000$ for validation, and $5000$ for testing, with $\lambda$ selected via bisection on the validation DRO objective. 

Figures \ref{fig:mnist_gaus_recon} and \ref{fig:mnist_gaus_metrics} show results for Gaussian noise. The DRO-UNets produce better reconstructions at higher noise levels than the MSE-UNet. The saliency maps show that the DRO model focuses on edges and strokes of the digit while ignoring the background, whereas the MSE-UNet shows diffuse, unstructured sensitivity without a consistent pattern, suggesting that the DRO objective encourages the model to focus on relevant features rather than noise. The DRO-UNets are more stable than the MSE-UNet in both noise settings. 

Without retraining, we apply the same model to measurements corrupted with Poisson noise. Given a noise level parameter $\sigma \in [0.01, 1]$, measurements are generated as $y=\text{Pois}(\alpha \cdot Hx)/\alpha$, with $\alpha=1/\sigma$, so that larger $\sigma$ corresponds to stronger noise while the mean is preserved. Figures~\ref{fig:mnist_pois_recon} and~\ref{fig:mnist_pois_metrics} show that the DRO model generalizes across both noise levels and noise families, while the MSE baseline degrades more sharply.
\begin{figure}[t]
    \centering
    \begin{subfigure}[t]{0.48\linewidth}
        \centering
        \includegraphics[height=7cm]{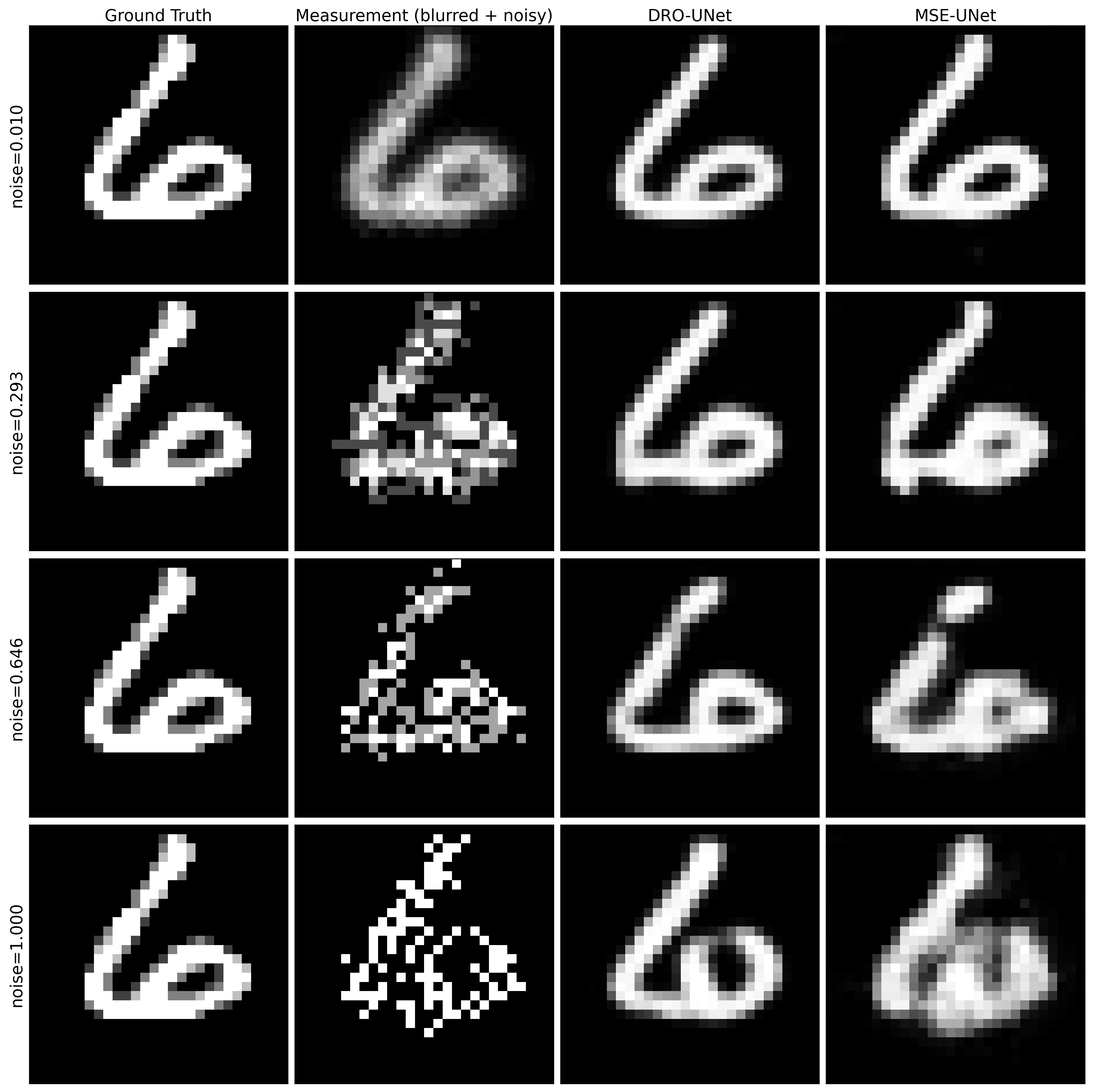}
    \end{subfigure}\hfill
    \begin{subfigure}[t]{0.48\linewidth}
        \centering
        \includegraphics[height=7cm]{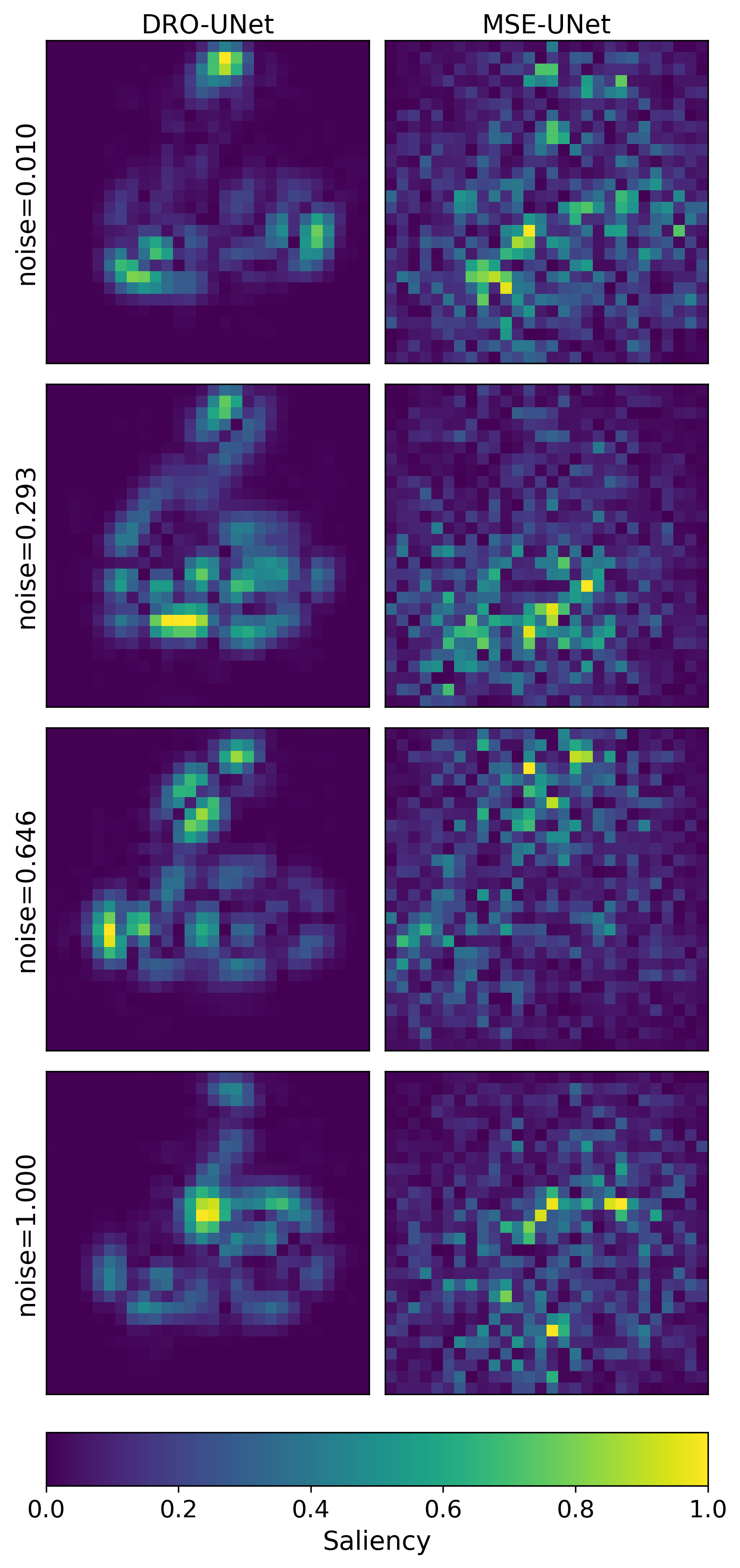}
    \end{subfigure}
    \caption{MNIST reconstructions and saliency under increasing Poisson noise. Reconstruction columns (left to right): ground truth, noisy measurement, DRO-UNet, MSE-UNet. Higher saliency intensities indicate the input pixels the model is most sensitive to during reconstruction}
    \label{fig:mnist_pois_recon}
\end{figure}

\begin{figure}[t]
    \centering
    \includegraphics[width=0.8\linewidth]{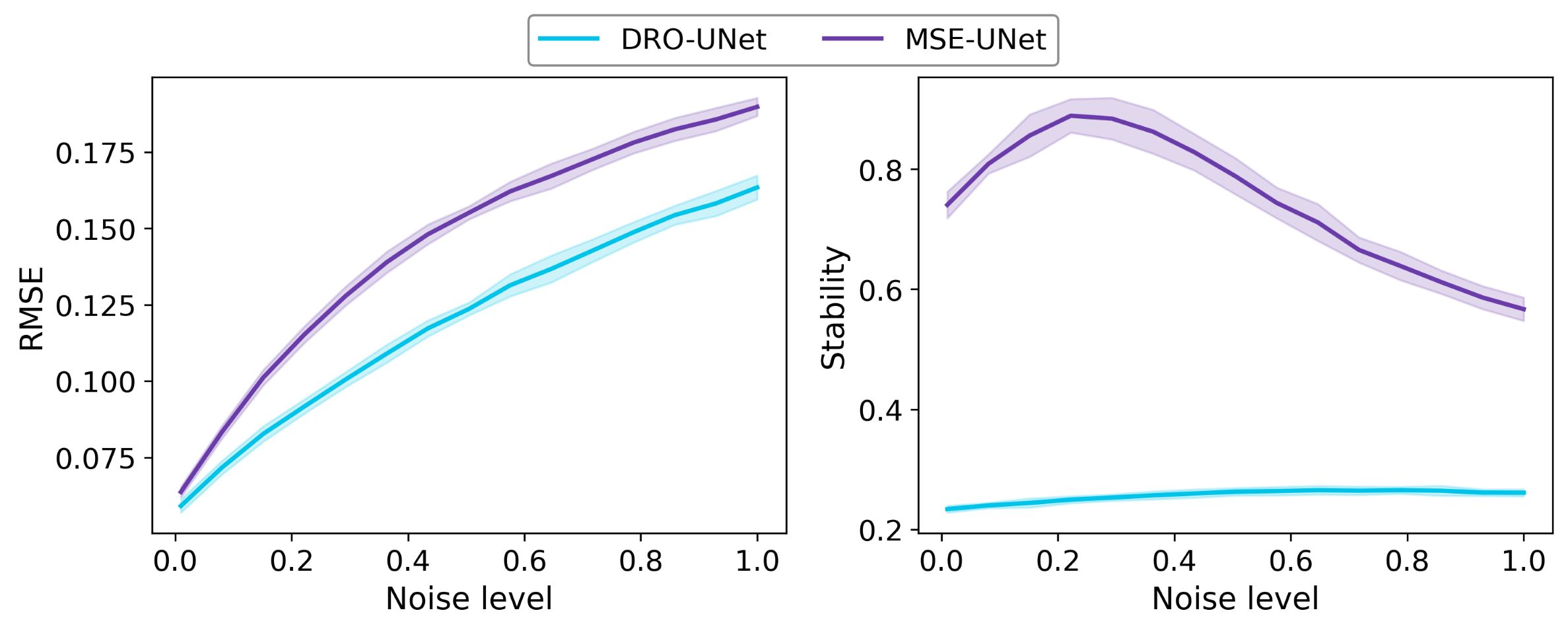}
    \caption{MNIST reconstruction metrics under increasing Poisson noise with panels and shaded bands as in Figure \ref{fig:diff-metrics}.}
    \label{fig:mnist_pois_metrics}
\end{figure}

\subsubsection{Robustness to Misspecified Forward Models}
\begin{figure}[t]
    \centering
    \begin{subfigure}[t]{0.48\linewidth}
        \centering
        \includegraphics[height=7cm]{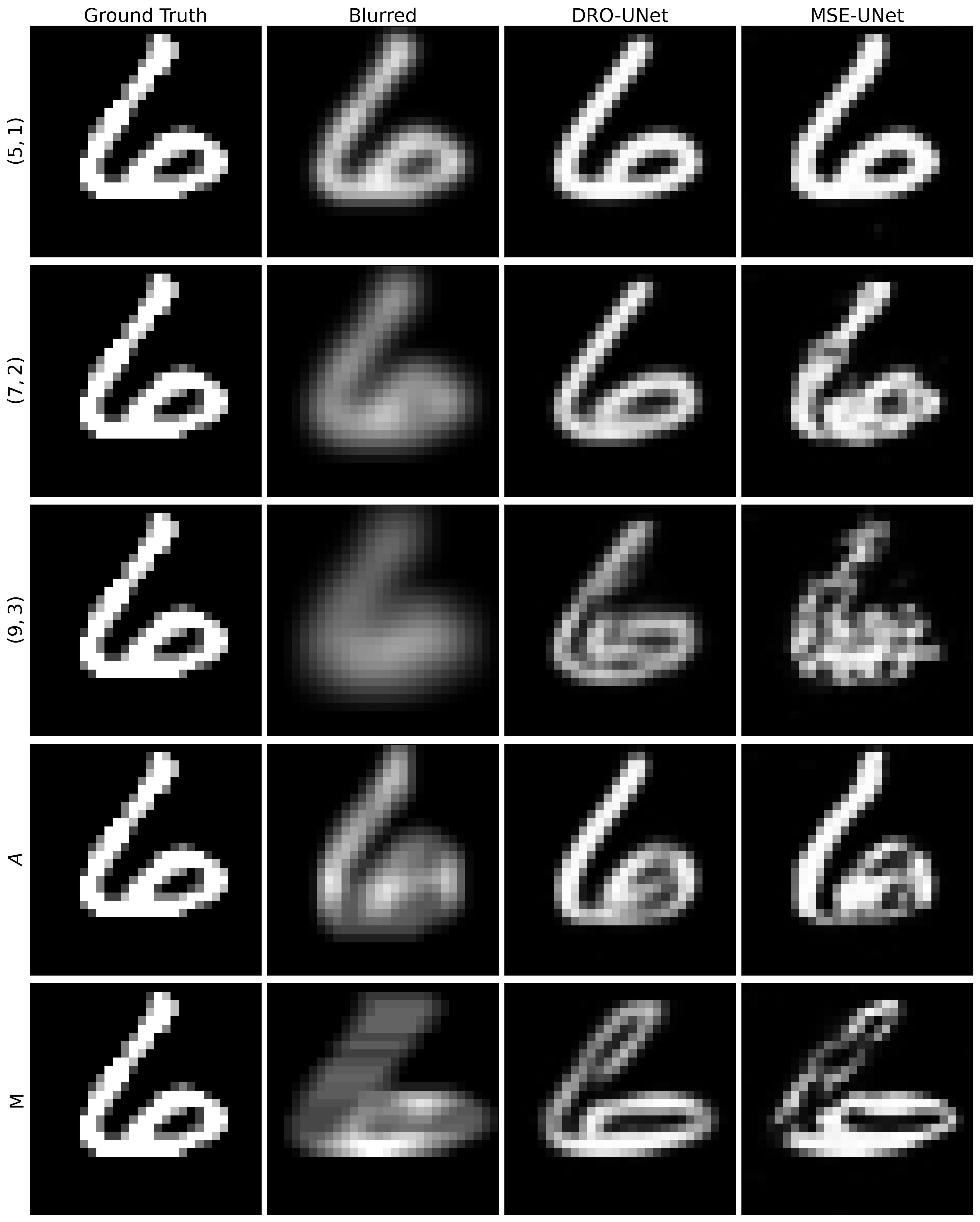}
    \end{subfigure}\hfill
    \begin{subfigure}[t]{0.48\linewidth}
        \centering
        \includegraphics[height=7cm]{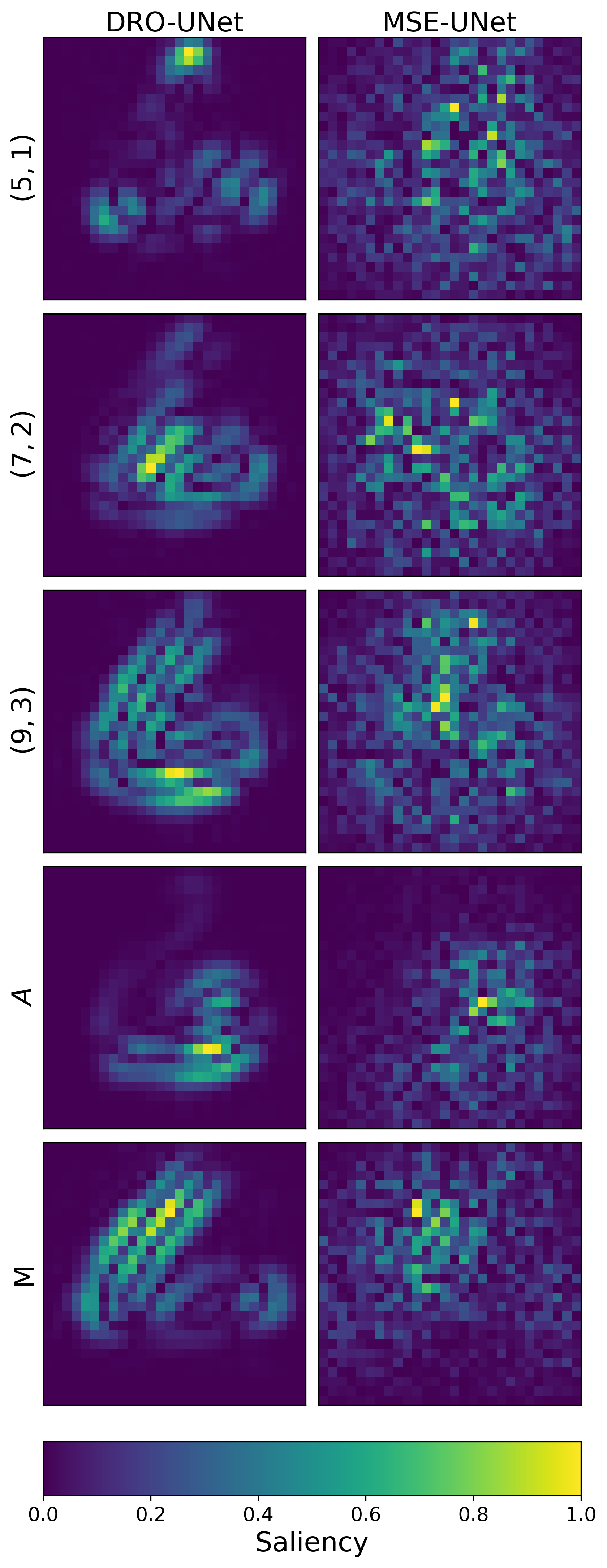}
    \end{subfigure}
    \caption{MNIST reconstructions and saliency maps for different blurring kernels. Reconstruction columns (left to right): ground truth, measurement, DRO-UNet, MSE-UNet. Rows are the five test kernels (training kernel, followed by two isotropic kernels, anisotropic, and motion blur). Higher saliency intensities indicate the input pixels the model is most sensitive to during reconstruction}
    \label{fig:mnist_kernel_recon}
\end{figure}

\begin{figure}[t]
    \centering
    \includegraphics[width=0.8\linewidth]{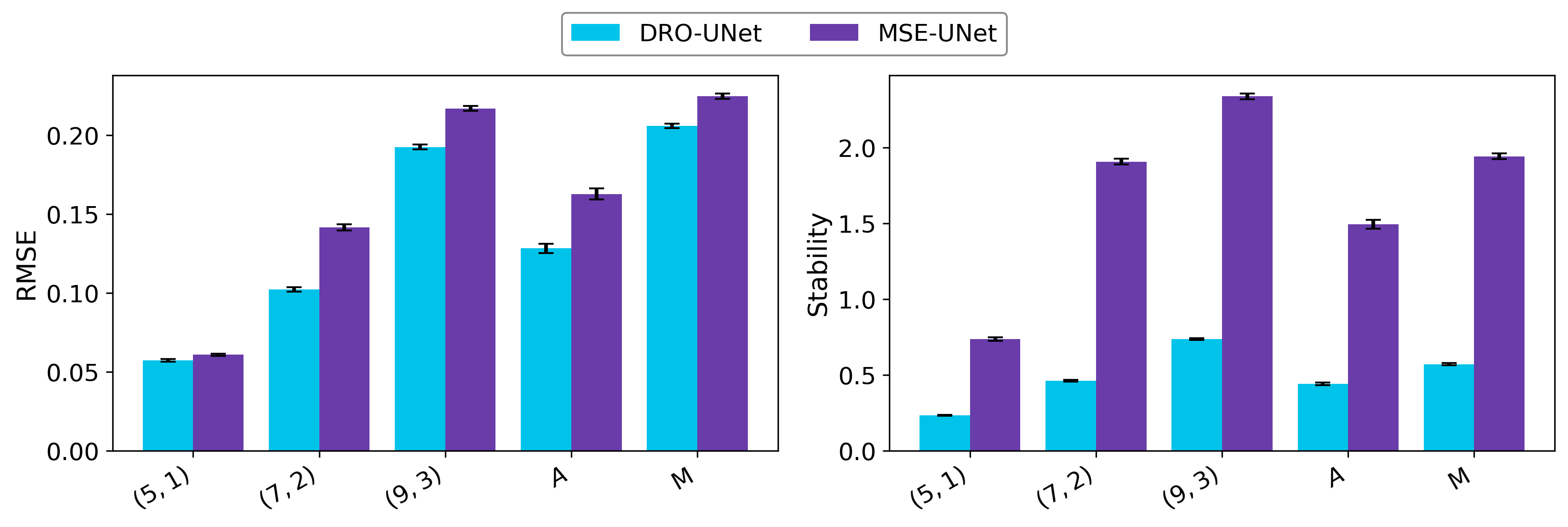}
    \caption{MNIST reconstruction metrics across the five test kernels of Figure~\ref{fig:mnist_kernel_recon}. Left: RMSE. Right: stability metric (lower is more stable).}
    \label{fig:mnist_kernel_metrics}
\end{figure}
To test robustness to misspecified forward operators (simulating equipment or settings mismatch at test time), we evaluate the trained model from Section~\ref{sec:mnist} on kernels not seen during training. In addition to various perturbed isotropic Gaussian kernels, which we denote as $\mathcal{K}_{k, \sigma}$, we test the $2$D anisotropic Gaussian kernel $\mathcal{K}_\text{A}$ and motion blur kernel $\mathcal{K}_\text{M}$:
\begin{align*}
    \mathcal{K}_\text{A} (s,r) = \frac{1}{2\pi\sigma_s\sigma_r} \exp\left( -\frac{s^2}{2\sigma^2_s}-\frac{r^2}{2\sigma^2_r} \right), \quad
    \mathcal{K}_\text{M} (s,r) = \frac{1}{L} \mathbf{1}_{[-L/2, L/2]} (s)\delta(r), 
\end{align*}
where $L>0$ is the blur length and $\delta$ is the Dirac delta. Motion blur introduces null frequencies in the kernel \cite{wang_optimization_2022}, making the problem more severely ill-posed. We evaluate on isotropic kernels $\mathcal{K}_{7,2}$ and $\mathcal{K}_{9,3}$, anisotropic kernel $\mathcal{K}_\text{A}$ with $k=7,\ \sigma_s=5,\ \sigma_r=0.5$, and motion blur kernel $\mathcal{K}_\text{M}$ with $L=9$.

Figures~\ref{fig:mnist_kernel_recon} and~\ref{fig:mnist_kernel_metrics} show reconstruction performance across the five kernels. The DRO-UNet consistently outperforms the MSE-UNet in both RMSE and stability. The saliency maps show the same edge-focused pattern.
\FloatBarrier
\subsubsection{Behavior under Extreme Ambiguity Radii}\label{sec:extreme-radii}
\begin{figure}[t]
    \centering
    \includegraphics[height=7cm]{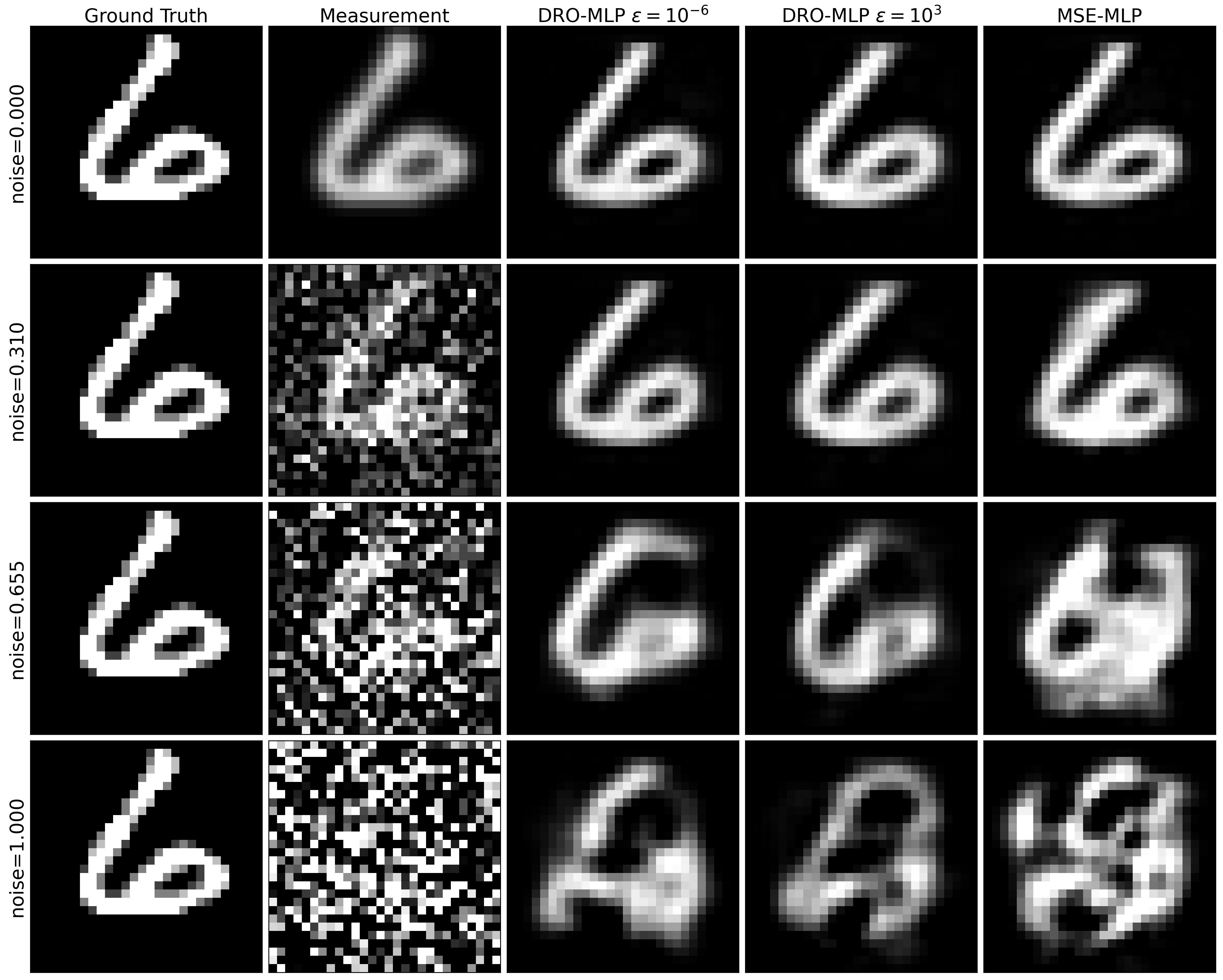}
    \caption{MNIST reconstructions for extreme radii. Columns (left to right): ground truth, noisy measurement, DRO-MLP small radius, DRO-MLP large radius, MSE-MLP.}
    \label{fig:mnist_extreme_recon}
\end{figure}

\begin{figure}[t]
    \centering
    \includegraphics[width=0.8\linewidth]{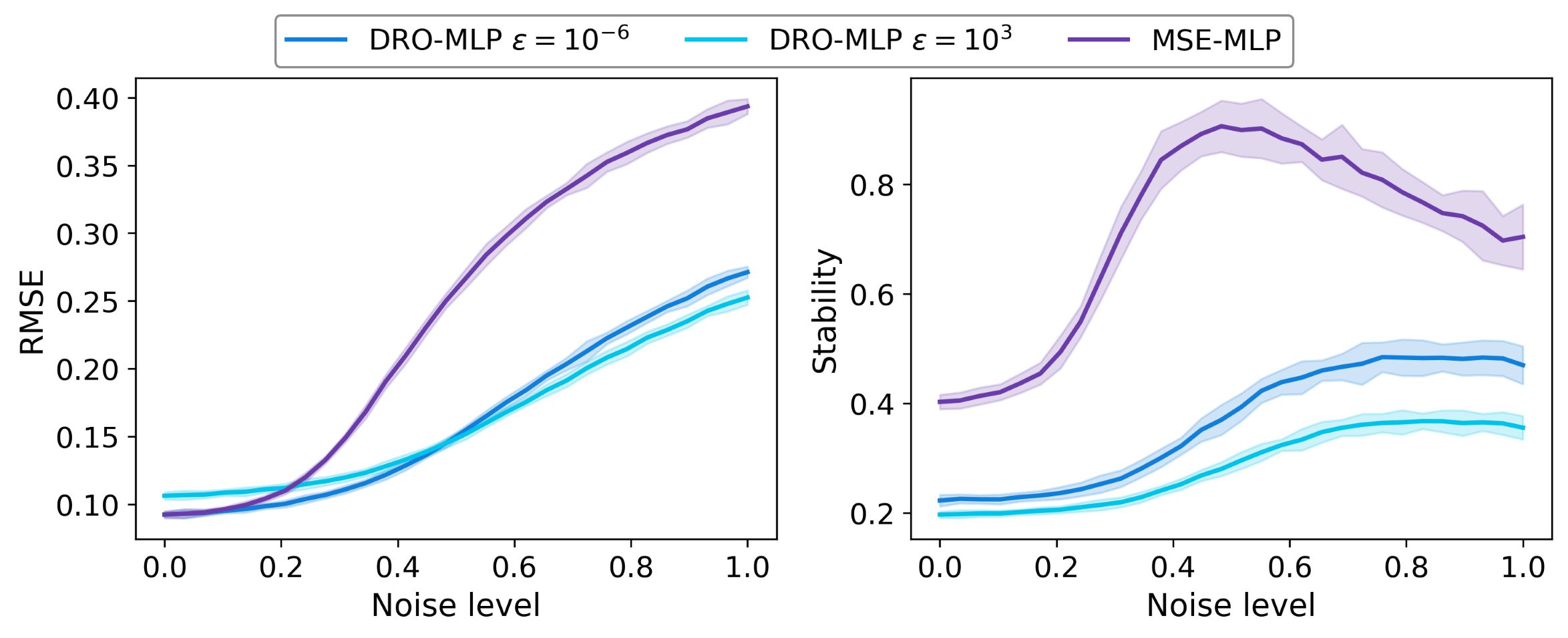}
    \caption{MNIST reconstruction metrics under increasing Gaussian noise for extreme radii with panels and shaded bands as in Figure \ref{fig:diff-metrics}.}
    \label{fig:mnist_extreme_metrics}
\end{figure}
To isolate the effect of the Wasserstein radius without the implicit robustness of a UNet, we replace the UNet with an MLP of $4$ layers, $512$ hidden units, and dropout rate $0.2$, using the same setup as the previous section. We test two extreme radii: $\varepsilon=10^{-6}$ with bisection bounds $[10^3, 10^6]$ and  tolerance $10$, and $\varepsilon=10^3$ with bounds $[10^{-5}, 10]$ and tolerance $10^{-6}$.

Figures~\ref{fig:mnist_extreme_recon} and~\ref{fig:mnist_extreme_metrics} illustrate the expected trade-off. The small-radius DRO-MLP performs better in low-noise regimes, where the empirical distribution closely approximates the truth and the tight ambiguity set makes the DRO objective behave similarly to empirical risk minimization, nearly coinciding with the MSE-MLP. The large-radius DRO-MLP outperforms at higher noise levels, as a larger radius allows the adversary to model stronger distributional shifts, trading fidelity for robustness. This is visible in the reconstructions, where the small-radius reconstruction is more faithful at low noise while the large-radius reconstruction is more smoothed but remains more faithful as noise grows. Small radii prioritize accuracy under the empirical distribution while large radii prioritize stability under worst-case perturbations, so that the ambiguity radius acts as an interpretable control over the accuracy-robustness trade-off. This is once more confirmed by the stability metric: both DRO-MLPs are more stable than the MSE-MLP across the entire noise range, with the large-radius model the more stable of the two.

\FloatBarrier
\subsection{Deblurring CIFAR-10}
\begin{figure}[t]
    \centering
    \begin{subfigure}[t]{0.48\linewidth}
        \centering
        \includegraphics[height=7cm]{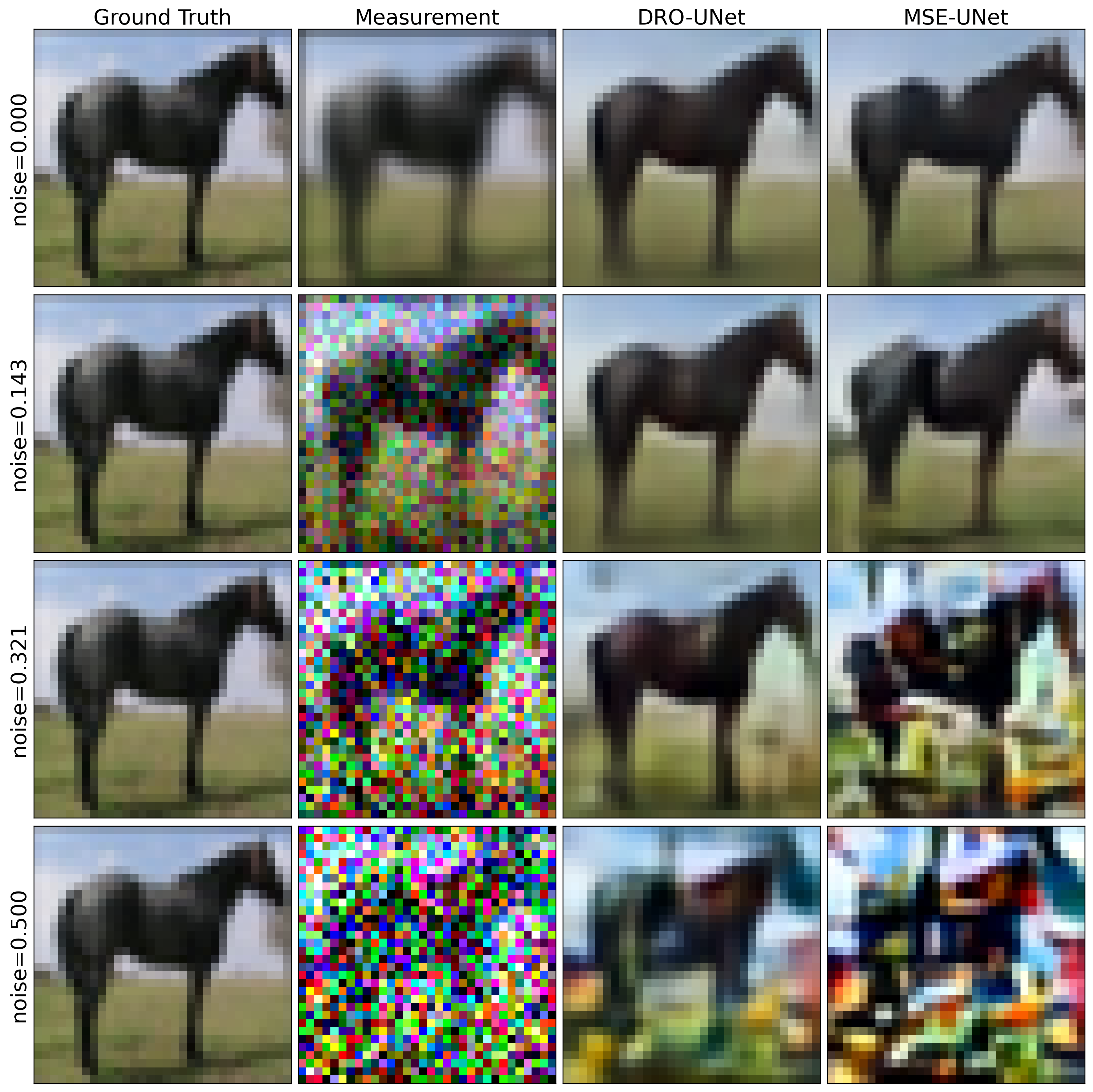}
    \end{subfigure}\hfill
    \begin{subfigure}[t]{0.48\linewidth}
        \centering
        \includegraphics[height=7cm]{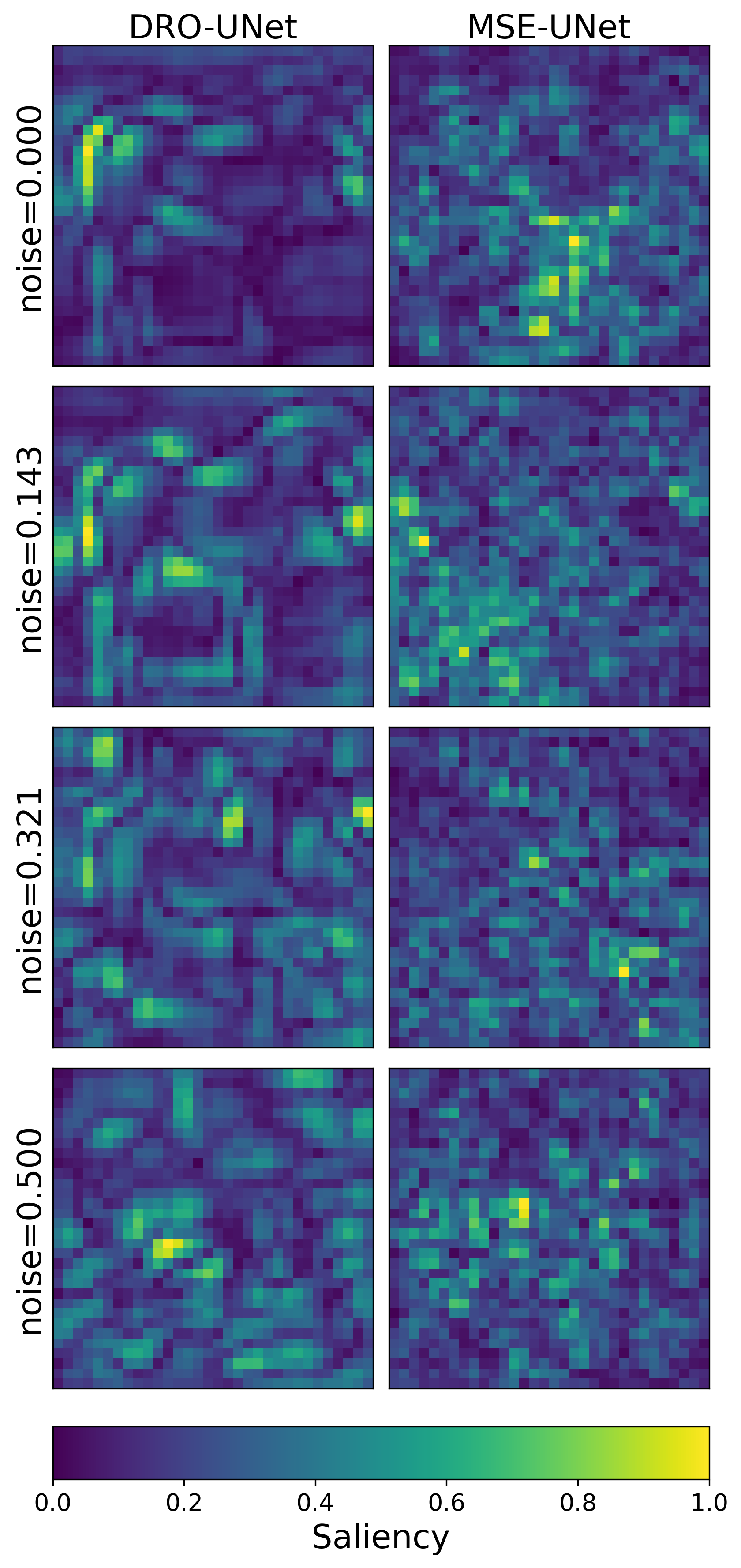}
    \end{subfigure}
    \caption{CIFAR-10 reconstructions and saliency under increasing Gaussian noise. Reconstruction columns (left to right): ground truth, noisy measurement, DRO-UNet, MSE-UNet. Higher saliency intensities indicate the input pixels the model is most sensitive to during reconstruction}
    \label{fig:cifar_recon}
\end{figure}

\begin{figure}[t]
    \centering
    \includegraphics[width=0.8\linewidth]{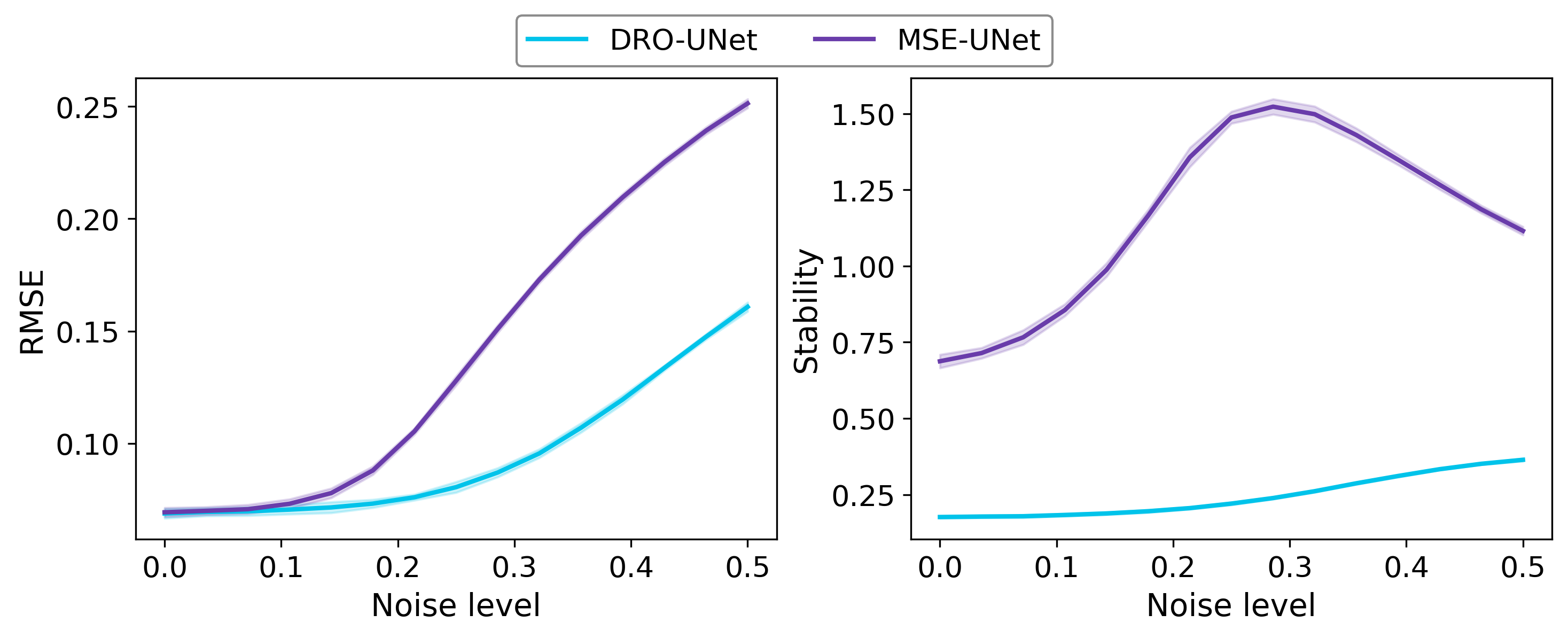}
    \caption{CIFAR-10 reconstruction metrics under increasing Gaussian noise with panels and shaded bands as in Figure \ref{fig:diff-metrics}.}
    \label{fig:cifar_metrics}
\end{figure}
We take the exact same setting as the MNIST deblurring experiment, with RGB images $x \in \R^{3\times32\times 32}$ drawn from CIFAR-10 \cite{krizhevsky_learning_2009} and a Unet with $3$ input channels and depth $3$. This experiment tests whether the framework scales to more complex natural image structure with higher inter-class variability than MNIST.

Figures~\ref{fig:cifar_recon} and~\ref{fig:cifar_metrics} largely mirror the MNIST findings. The DRO-UNet again achieves lower RMSE and better stability across all noise levels, with saliency maps showing the same pattern of structured, edge-focused sensitivity compared to the diffuse sensitivity of the MSE-UNet. The difference between the two models is most clearly visible at higher noise levels. While the reconstructions look similar in the low-noise regime, the MSE-UNet begins to introduce artifacts and hallucinations as the noise grows, whereas the DRO-UNet continues to return a coherent, denoised and deblurred image. This is reflected in the RMSE curve: both models start at a comparable value for low noise, but the RMSE of the MSE-UNet grows steeply while that of the DRO-UNet increases gradually, so that the gap between them widens with the noise level. The stability metric also shows a clear separation. The DRO-UNet has a stability ratio that stays low and grows slightly while the MSE-UNet is significantly more sensitive to input perturbations.

\subsection{Robust Simulator}\label{sec:robust-simulator}
\begin{figure}[t]
    \centering
    \begin{subfigure}[t]{0.48\linewidth}
        \centering
        \includegraphics[height=7cm]{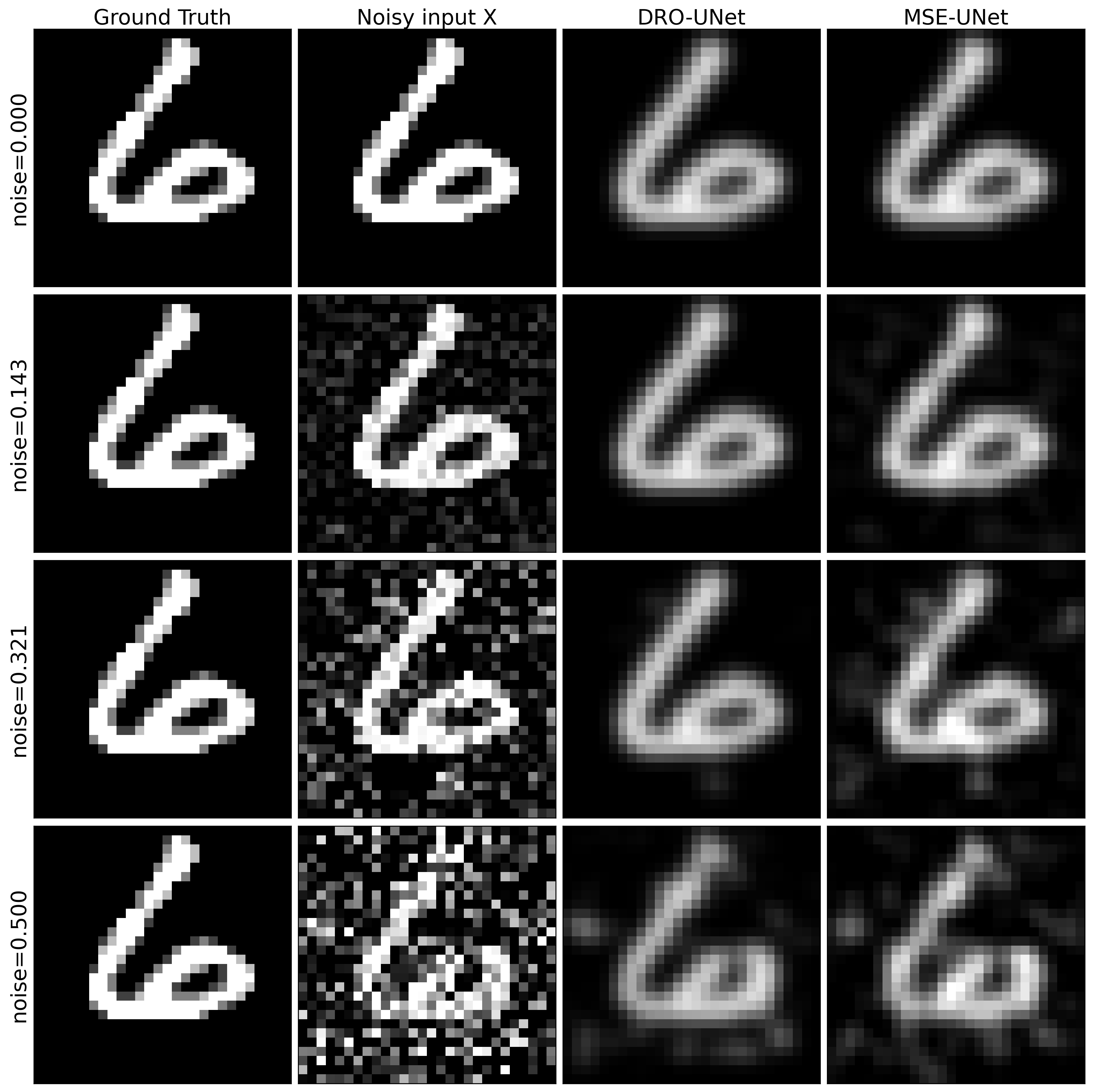}
    \end{subfigure}\hfill
    \begin{subfigure}[t]{0.48\linewidth}
        \centering
        \includegraphics[height=7cm]{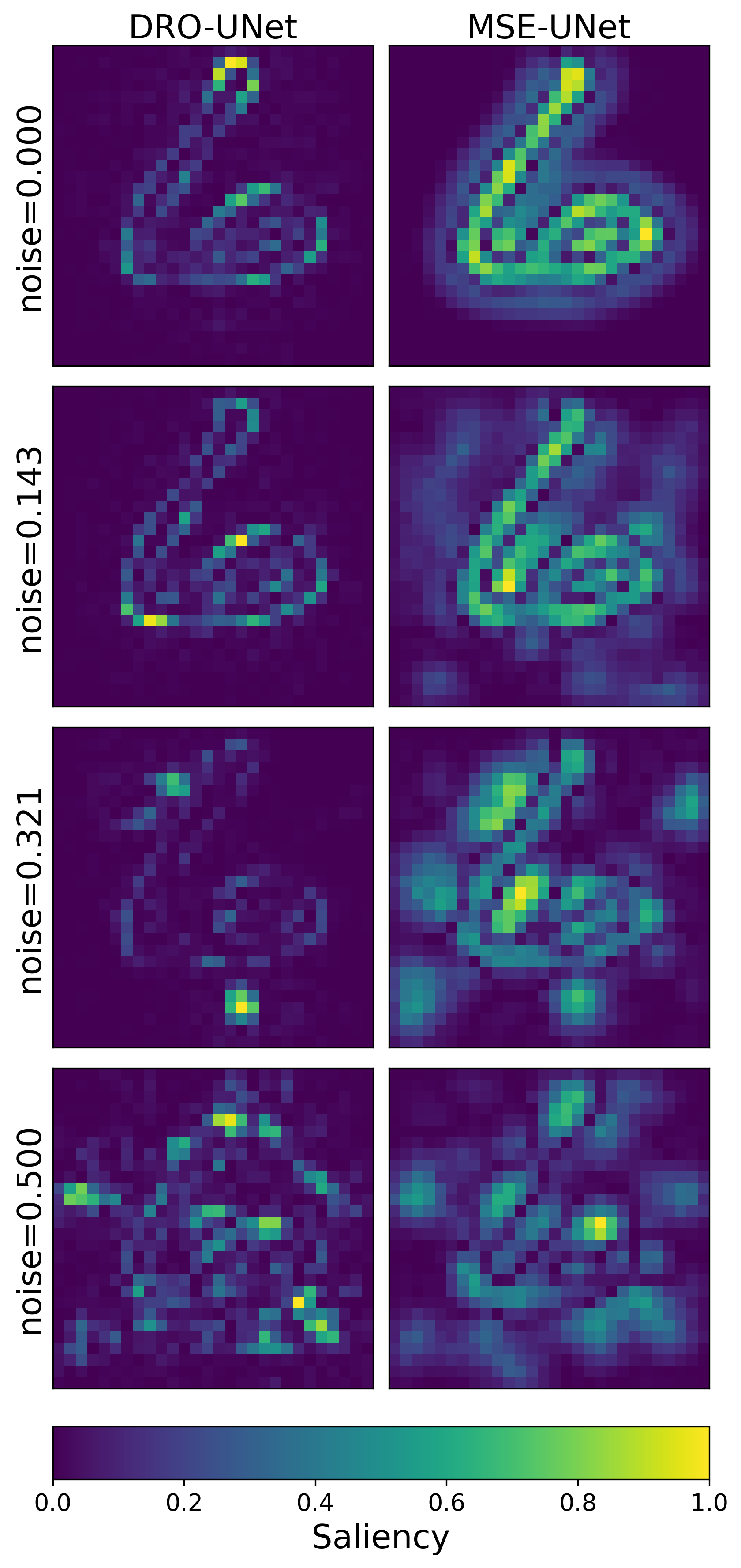}
    \end{subfigure}
    \caption{MNIST robust simulator: simulated measurements and saliency maps under increasing input noise. Simulation columns (left to right): ground truth, noisy ground truth, DRO-UNet applied to noisy ground truth, MSE-UNet. Rows correspond to increasing noise levels (top to bottom). Higher saliency intensities indicate the input pixels the model is most sensitive to during reconstruction. Here, the model learns the forward map.}
    \label{fig:mnist_simulator_recon}
\end{figure}

\begin{figure}[t]
    \centering
    \includegraphics[width=0.8\linewidth]{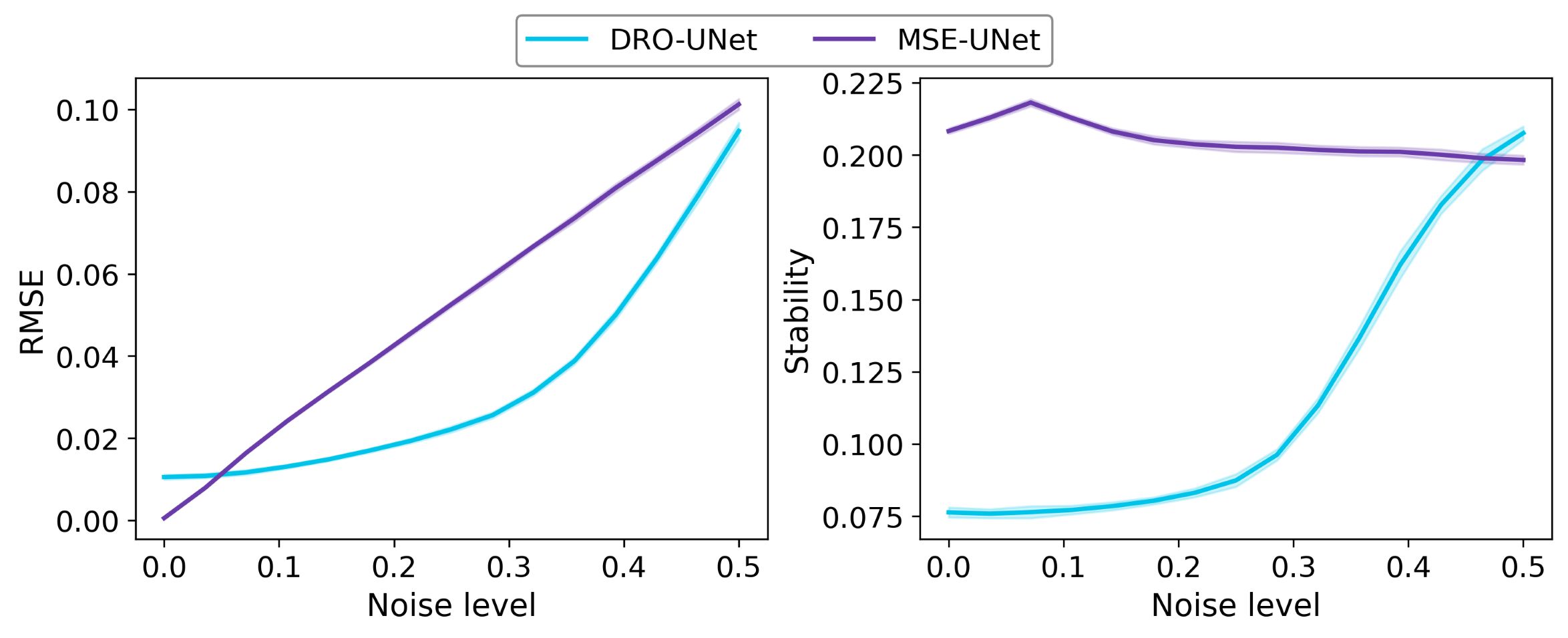}
    \caption{MNIST robust simulation metrics under increasing Gaussian noise with panels and shaded bands as in Figure~\ref{fig:diff-metrics}.}
    \label{fig:mnist_simulator_metrics}
\end{figure}

The previous sections studied robustness from the perspective of the inverse map, emphasizing stability under perturbations in $Y|X$. The robust simulator shifts focus to the forward operator $H$ itself, learning a forward mapping stable under perturbations in $X$. Learning a map for simulation has also been explored through invertible generative models that represent the joint distribution of inputs and measurements and can be run in either the forward or inverse direction \cite{leeuwen_invertible_2025}. In a blurring setting, a naive simulator propagates input noise through $H$, producing a blurred, noisy output. A robust simulator instead suppresses off-distribution artifacts before applying $H$, effectively learning a map of the form $H_\text{robust}(x)\approx(H\circ D)(x)$, where $D$ is an implicit projection onto the data manifold. Potential applications include synthetic data generation and learned forward modeling.

The setup follows previous sections, with the key difference being the prediction direction that results in a different loss. We consider $c=c_{X|Y}$ defined analogous to \eqref{eq:cond-cost}, the reference measure $\eta_{S|S=(x,y)}(u,v)=\eta_X(u)\otimes\mu_{Y|X=x}^*(v)$ with $\eta_X$ a Lebesgue measure, and the loss $\ell(x,y;G)=\|y-G(x)\|_2^2$, where $G$ is the learned robust approximation of $H$. This setting reflects perturbations in $X|Y$ (see Section \ref{sec:str}) with a forward loss.

We find the dual formulation 
\begin{align*}
    I = \inf_{G\in \mathcal{G}} \inf_{\lambda \geq 0} \lambda \hat \varepsilon +  \lambda\delta\int_{S} \log \left[ \int_X\int_Y \exp\left(\frac{\|v-G(x)\|^2_2}{\lambda \delta}\right)d\mu_{Y|X=x}^*(v)\ d\mathcal{N}(x, \delta^2I)(u)\right] d\mu^*(x,y).
\end{align*}
Compared with \eqref{eq:sinkhorn-dual}, the roles of $u$ and $v$ are exchanged, and the Gaussian is centered at $x$ rather than $Hx$. Figures~\ref{fig:mnist_simulator_recon} and~\ref{fig:mnist_simulator_metrics} show a clear distinction between the two models. The DRO-UNet produces outputs closely resembling a blurred but noise-reduced version of the ground truth, while the MSE-UNet struggles more to remove noise and propagates artifacts to its predictions.

\subsection{CT Reconstruction}
Computed tomography reconstructs a cross-sectional image of an object from X-ray projections taken at many angles, and is often used in medical and industrial imaging. We consider CT reconstruction from sinogram measurements using the 2DeteCT dataset \cite{kiss_2detect_2023}, preprocessed via the LION library \cite{biguri_lion_2026}. The forward operator $H:\R^{1024\times1024}\to\R^{3600\times956}$ is the discretized cone-beam
Radon transform, mapping a cross-sectional image to a sinogram over $3600$ projection angles with $956$ detector pixels. This is a severely ill-posed problem in which small sinogram perturbations can produce large reconstruction artifacts. Reducing the radiation dose lowers the photon count and amplifies this ill-posedness, making robust low-dose reconstruction a real-world problem.

Raw sinograms are preprocessed via flat-dark field correction and a negative log transform to convert intensity measurements into line integrals. The UNet input is a filtered backprojection (FBP) reconstruction of the preprocessed sinogram, computed using \verb|tomosipo| \cite{hendriksen_tomosipo_2021} and \verb|astra| \cite{van_aarle_astra_2015} with the scanner's cone-beam geometry, followed by normalization. The inverse is parameterized by a UNet with $16$ base channels. 

Perturbed inputs for the DRO objective are obtained by adding Gaussian noise of standard deviation $\delta$ to the flat-dark corrected sinogram before the negative log transform and FBP reconstruction. The DRO-UNet is trained on noise-free FBP reconstructions. The MSE-UNet is trained on FBP reconstructions from Poisson-corrupted sinograms, simulated by scaling by photon count $I_0=10^3$, drawing from the Poisson distribution, rescaling by $I_0^{-1}$, and applying the negative log transform and FBP. The same procedure is used at test time. The photon count $I_0$ reflects the number of photons reaching the detector per measurement, with a lower count corresponding to a lower radiation dose and a noisier signal due to the reduced number of detected counts. This Poisson corruption corresponds to the noise model in Example~\ref{ex:noise}(b), with the photon count $I_0$ playing the role of the intensity parameter: decreasing $I_0$ increases noise.

For the DRO formulation, we set $\delta=0.002$ and $\varepsilon=0.1$, with $\lambda$ determined via bisection over $[1,100]$ with tolerance $1.0$ evaluated on a $10\%$ training subsample for efficiency. We use a learning rate $\gamma=3\times10^{-4}$, batch size $3$, and the SG estimator with $4$ samples per iteration. We use $3500$ training samples, $750$ for validation, and $750$ for testing.

\begin{figure}[tp]
    \centering
    \includegraphics[width=0.8\linewidth]{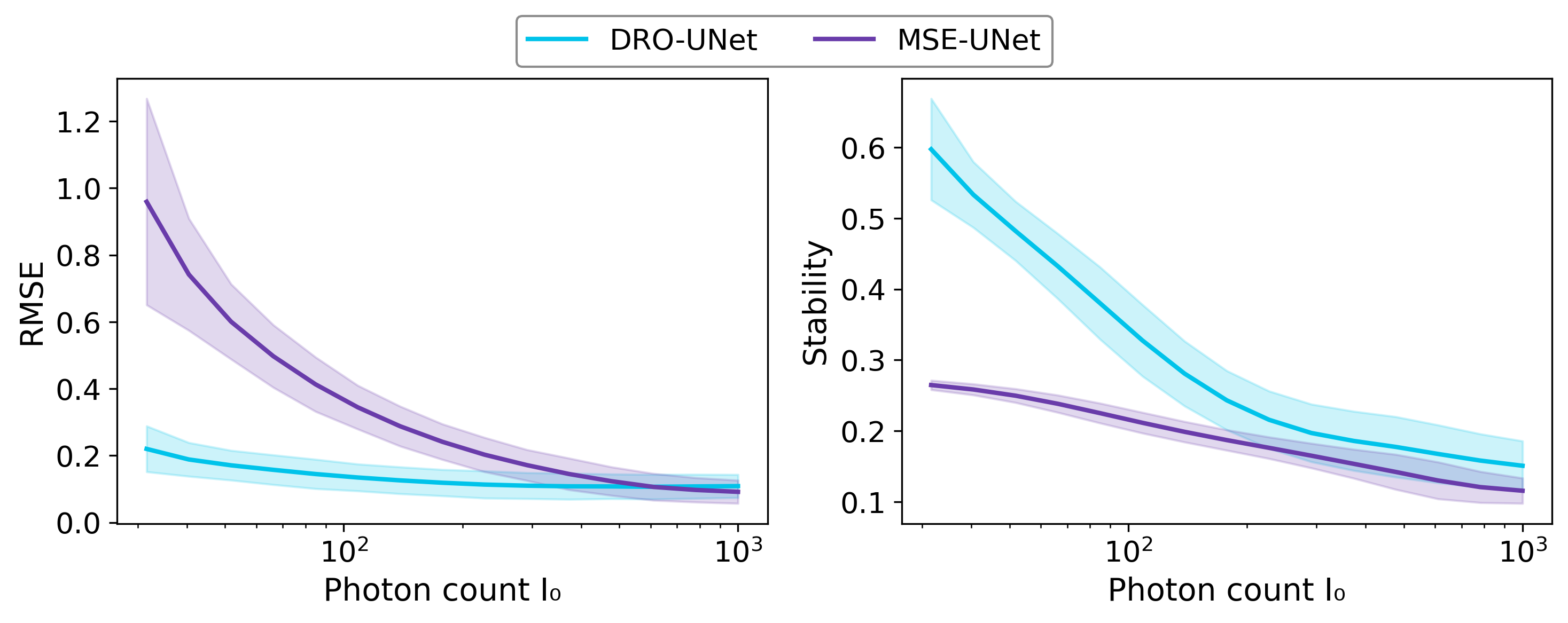}
    \caption{2DeteCT CT reconstruction metrics as the photon count $I_0$ increases, with panels and shaded bands as in Figure~\ref{fig:diff-metrics}. Lower $I_0$ means higher noise.}
    \label{fig:ct_rmse}
\end{figure}

\begin{figure}[tp]
    \centering
    \begin{subfigure}[t]{\linewidth}
        \centering
        \includegraphics[height=10cm]{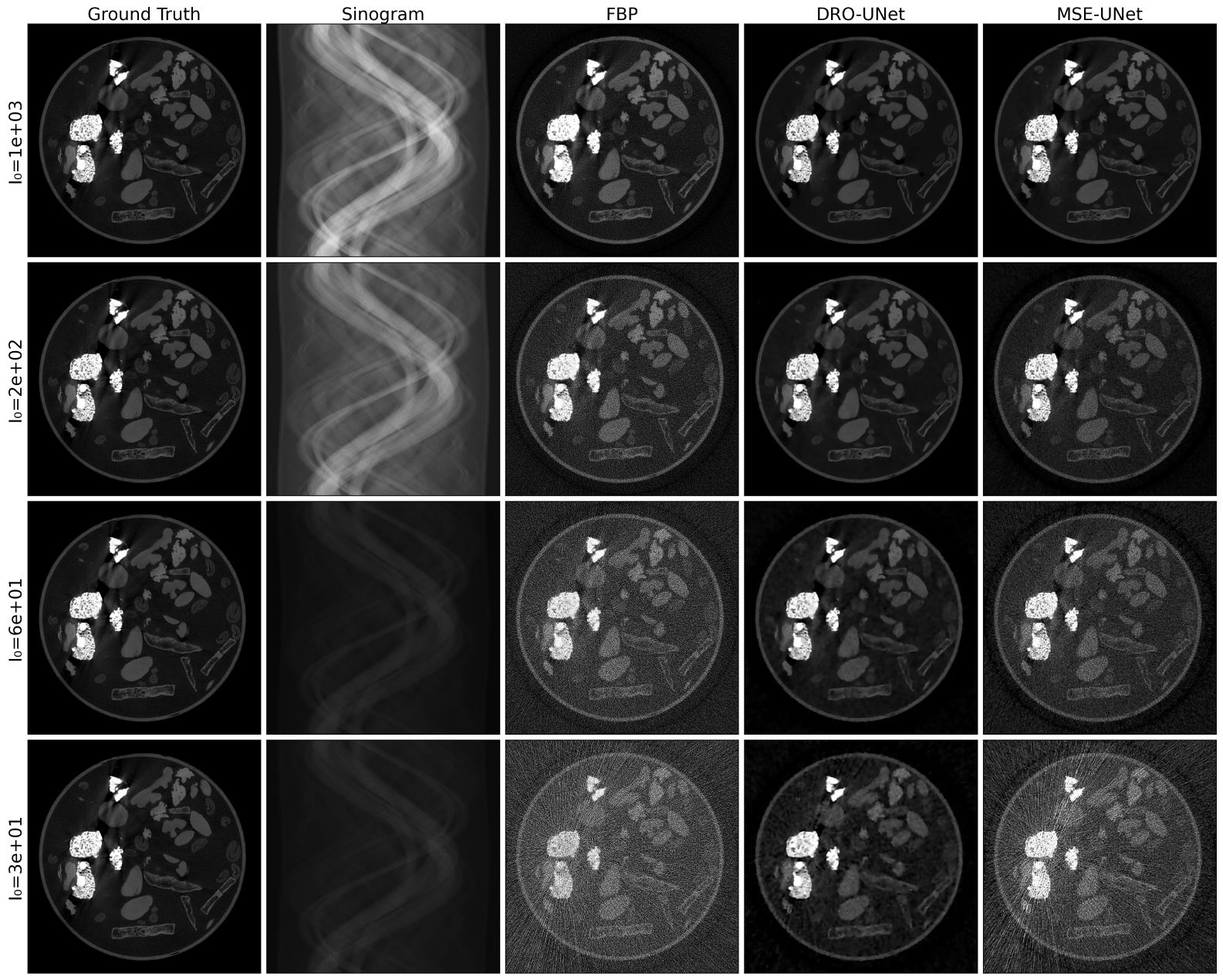}
    \caption{Reconstructions}
    \end{subfigure}\hfill
    \begin{subfigure}[t]{0.48\linewidth}
        \centering
        \includegraphics[height=10cm]{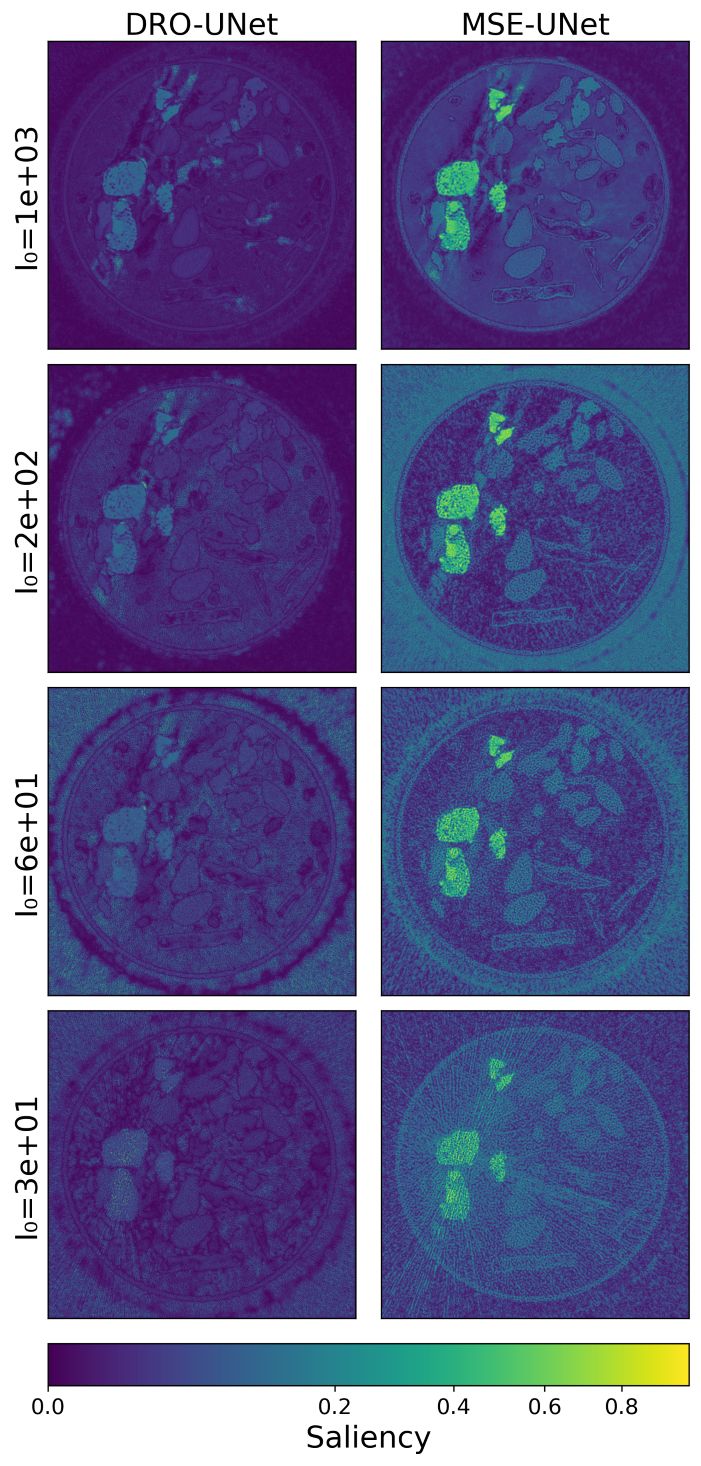}
    \caption{Saliency}
    \end{subfigure}\hfill
    \begin{subfigure}[t]{0.48\linewidth}
        \centering
        \includegraphics[height=10cm]{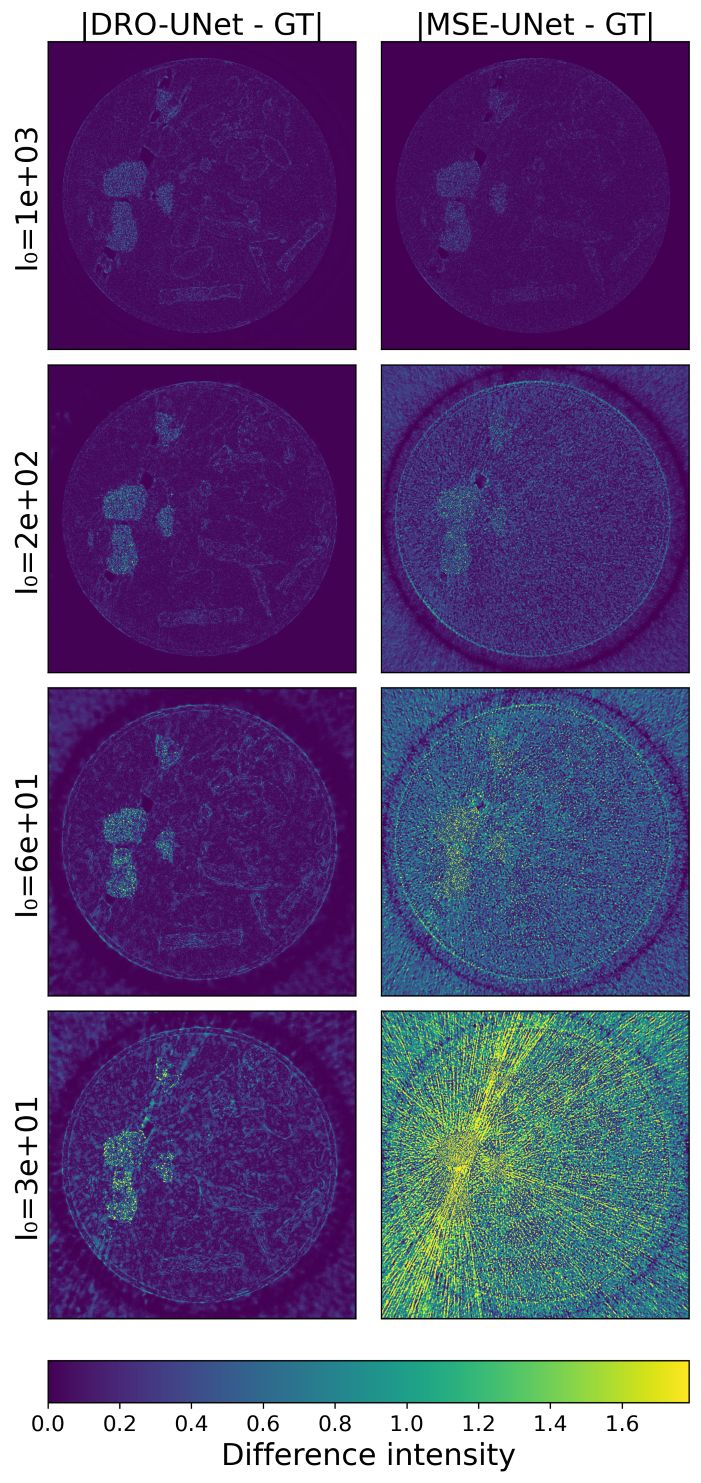}
    \caption{Absolute difference w.r.t. ground truth}
    \end{subfigure}
    \caption{2DeteCT CT reconstructions, saliency, and difference at photon counts from high-dose ($I_0=10^3$, top row) to low-dose ($I_0=30$, bottom row). Reconstruction columns (left to right): ground truth, sinogram, FBP input, DRO-UNet, MSE-UNet. Higher saliency intensities indicate the input pixels the model is most sensitive to during reconstruction.}
    \label{fig:ct_recon_sal}
\end{figure}

Figure~\ref{fig:ct_recon_sal} shows reconstruction results for decreasing photon counts from high-dose ($I_0=10^3$) to low-dose ($I_0=3\times10^1$). At high photon counts, both models perform comparably, closely matching the ground truth. As $I_0$ decreases, the MSE-UNet degrades significantly with noise artifacts and structural loss, while the DRO-UNet maintains considerably better image quality. The saliency maps in Figure~\ref{fig:ct_recon_sal} show that the DRO-UNet focuses on structurally meaningful regions such as shape boundaries, while the MSE-UNet shows increasingly diffuse and unstructured sensitivity as noise increases, and while focusing on diagonal noise artifacts, the DRO-UNet ignores. The maps visualizing the difference between the reconstruction and the ground truth (GT) corroborate this: DRO-UNet errors remain spatially concentrated and small, whereas MSE-UNet errors spread across the entire image at low $I_0$.

Figure~\ref{fig:ct_rmse} reflects this trend quantitatively, where the DRO-UNet consistently achieves lower RMSE with the gap widening at low $I_0$. The stability curves are more nuanced:  the MSE-UNet seems more stable than the DRO-UNet, likely because it maps all inputs to a similar blurred output, giving artificially low sensitivity to perturbations.

\FloatBarrier
\section{Conclusion}
We introduced structured DRO, a distributionally robust optimization framework for inverse problems in which adversarial perturbations are constrained to a structured subset $K_s \subseteq P(X \times Y)$ aligned with the data-acquisition process. We established strong duality for general $K_s$ and derived explicit finite-dimensional dual representations for perturbations in the joint, marginal, and conditional distributions. For perturbations in $P(Y|X)$, the explicit bound yields a Tikhonov-type regularization term on the Lipschitz constant of the reconstruction operator. Numerical comparison with joint DRO showed that the structured bound is less conservative for well-posed problems, with the decomposition into fidelity and regularization contributions making the source of conservatism in each method interpretable. In the linear case, this regularization is visible spectrally in the numerical experiments: the learned inverse truncates at the intrinsic dimension of the data and recovers a data-driven version of truncated-SVD regularization. Learned reconstruction experiments on differentiation, deblurring, and sinogram-to-CT reconstruction further demonstrated improved robustness, interpretability, and stability over MSE baselines, with out-of-distribution performance consistently improving over standard DRO and MSE baselines as noise levels increase. 

On the computational side, more faithful approximations of the structured Wasserstein constraint beyond the combination of the conditional Wasserstein distance and the Sinkhorn distance used here require further investigation. A natural theoretical extension concerns operator learning: rather than assuming a fixed forward operator $H$, one can parameterize $H$ by a neural operator and ask for a reconstruction operator $G$ that is robust to uncertainty in the learned forward model. The framework already accommodates stochastic forward operators via the robust simulator of Section~\ref{sec:robust-simulator}, making it suited to this setting. The framework also connects naturally to generative modeling: the conditional Wasserstein distance of Corollary~\ref{cor:w-cond} is used in the objective in Bayesian OT flow matching \cite{chemseddine_conditional_2025}, suggesting that a flow trained under a structured DRO objective would yield a learned distribution that is robust to structured perturbations. Finally, extending the framework to dynamic inverse problems where the data distribution $\mu_t$ evolves over time leads to a min-max problem over flows of measures rather than static distributions, connecting naturally to Wasserstein gradient flows.

\section*{Acknowledgments}
M.C.'s research is supported by the NWO-Vidi grant \emph{SPARGO: Exploring and Exploiting the Geometric Landscape of Infinite-Dimensional Sparse Optimization} (Grant Number VI.Vidi.243.200). C.B. acknowledges financial support from EU EFRO OPoost (No. OOST-00103).

\bibliographystyle{siamplain}
\bibliography{references_a}

\appendix
\section{Robustness in $P(X\times Y)$ (MNIST)}\label{supp:additional-examples}
\label{sec:num-joint}
To validate the framework under joint perturbations in $P(X\times Y)$, we repeat the MNIST deblurring experiment from Section~\ref{sec:mnist} with a reference measure that perturbs both the input and the measurement simultaneously. Specifically, we replace the reference measure with $\eta_{S|S=(x,y)}(u,v)=\eta_X(u)\otimes\eta_{Y|X=x}(v)$ with $\eta_{Y|X=x}$ and $\eta_X$ Lebesgue measures, so that the worst-case adversary is now free to transport mass in both the $X$- and $Y$-directions. All other hyperparameters are kept identical to Section 6.3. The resulting dual objective is
\begin{align*}
    I = \inf_{G\in \mathcal{G}} \inf_{\lambda \geq 0} \lambda \hat \varepsilon +  \lambda\delta\int_S \log \int_Y\int_X \exp\left(\frac{\|u-G(v)\|^2_2}{\lambda \delta}\right)d\mathcal{N}(x, \delta^2I)(u)\ d\mathcal{N}(Hx, \delta^2I)(v) \ d\mu^*(x,y).
\end{align*}

Figures \ref{fig:mnist_joint_recon} and \ref{fig:mnist_joint_metrics} compare the $P(Y|X)$ and $P(X\times Y)$ DRO-UNets. The joint DRO model achieves slightly better RMSE than its conditional counterpart at low noise, but degrades slightly faster as noise increases, reflecting the additional conservatism introduced by allowing the adversary to perturb the input distribution as well. These differences however are very small but might be larger when using for example an MLP with less implicit robustness. In any case, the results are in line with the expected ordering from Section~\ref{sec:comparison}.

\begin{figure}[t]
    \centering
    \begin{subfigure}[t]{0.48\linewidth}
        \centering
        \includegraphics[height=7cm]{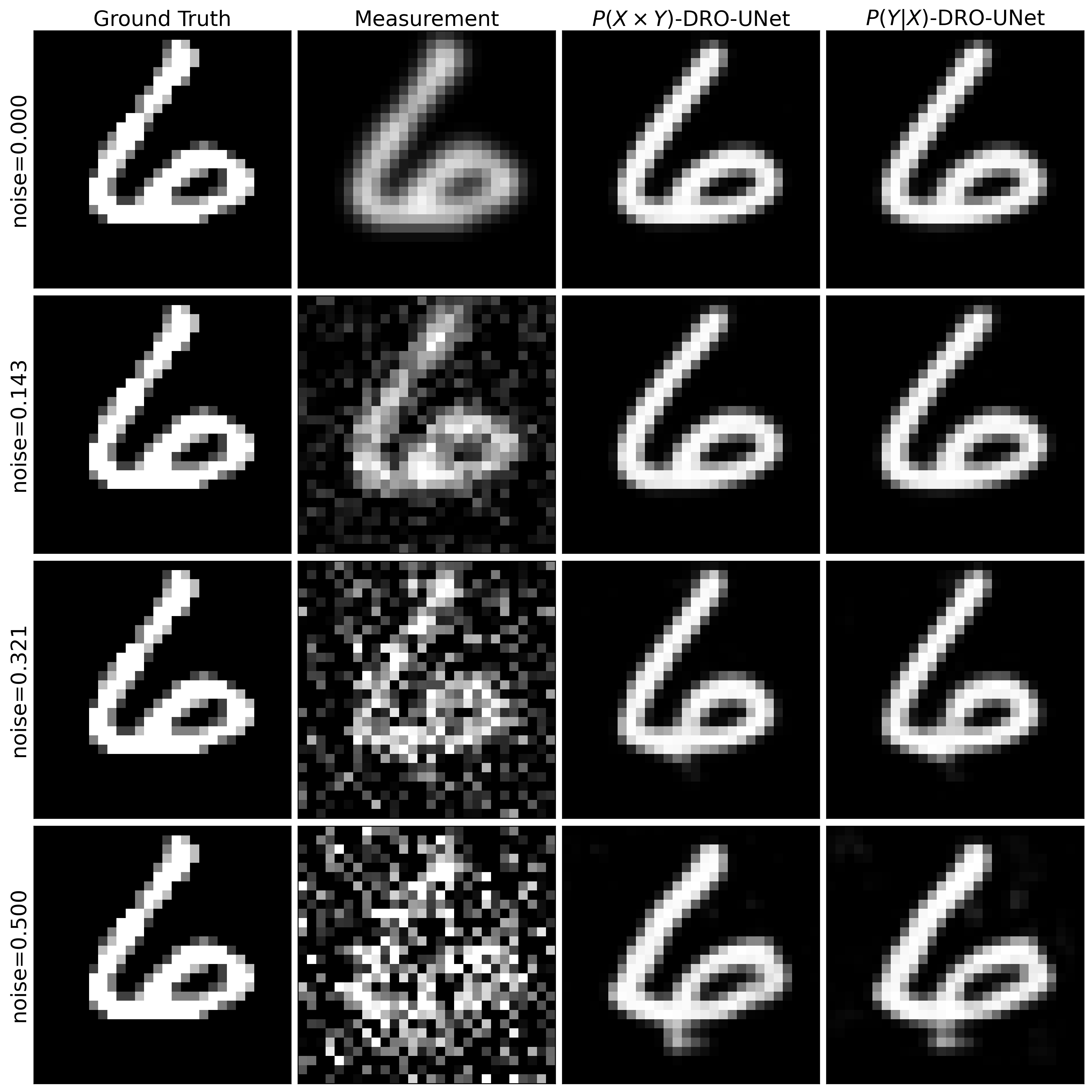}
    \end{subfigure}
    \hfill
    \begin{subfigure}[t]{0.48\linewidth}
        \centering
        \includegraphics[height=7cm]{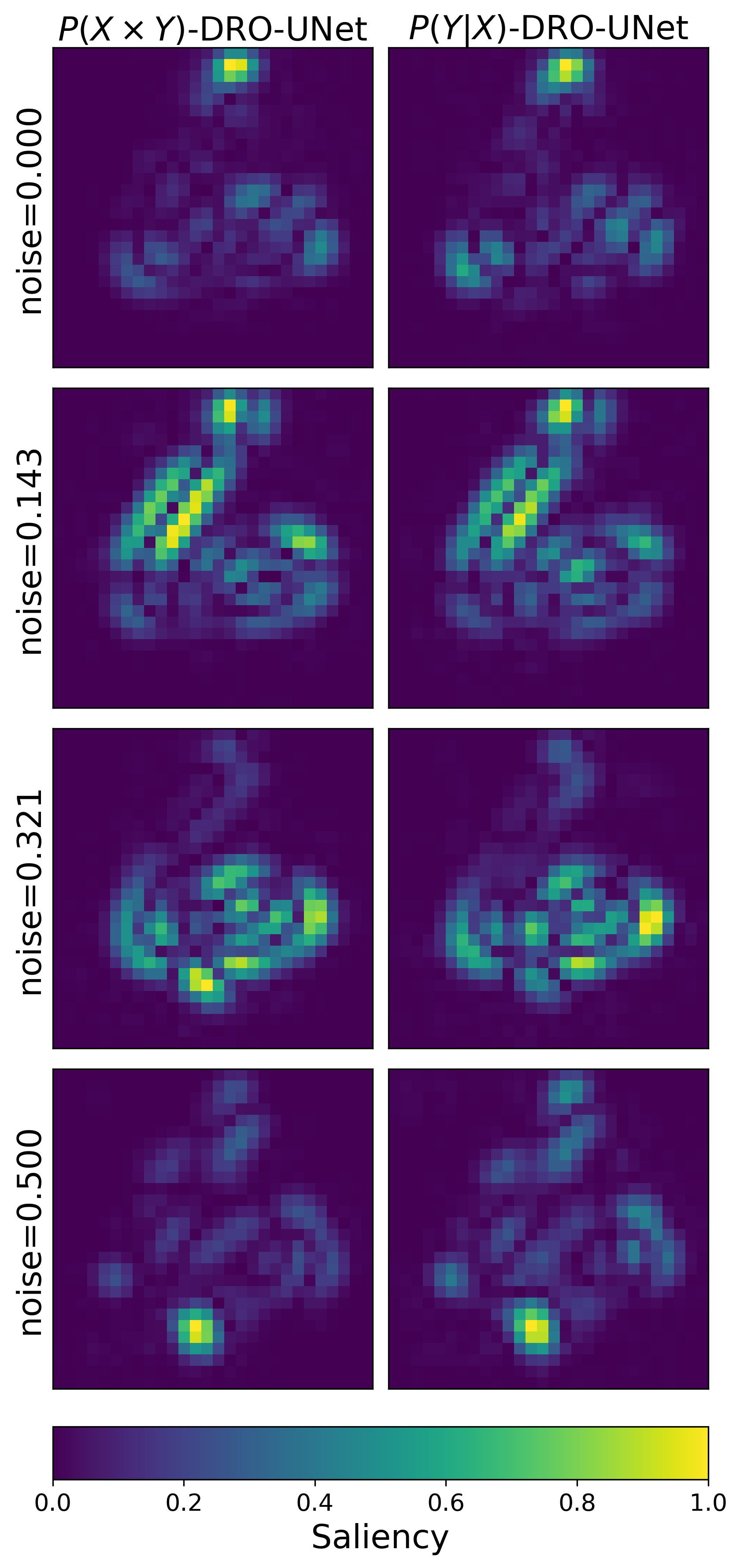}
    \end{subfigure}
    \caption{MNIST reconstructions and saliency under increasing Gaussian noise. Comparison of robustness in $X\times Y$ and $Y|X$. Reconstruction columns (left to right): ground truth, noisy measurement, $P(X\times Y)$-DRO-UNet, $P(Y|X)$-DRO-UNet. Higher saliency intensities indicate the input pixels the model is most sensitive to during reconstruction.}
    \label{fig:mnist_joint_recon}
\end{figure}

\begin{figure}[t]
        \centering
        \includegraphics[width=0.8\linewidth]{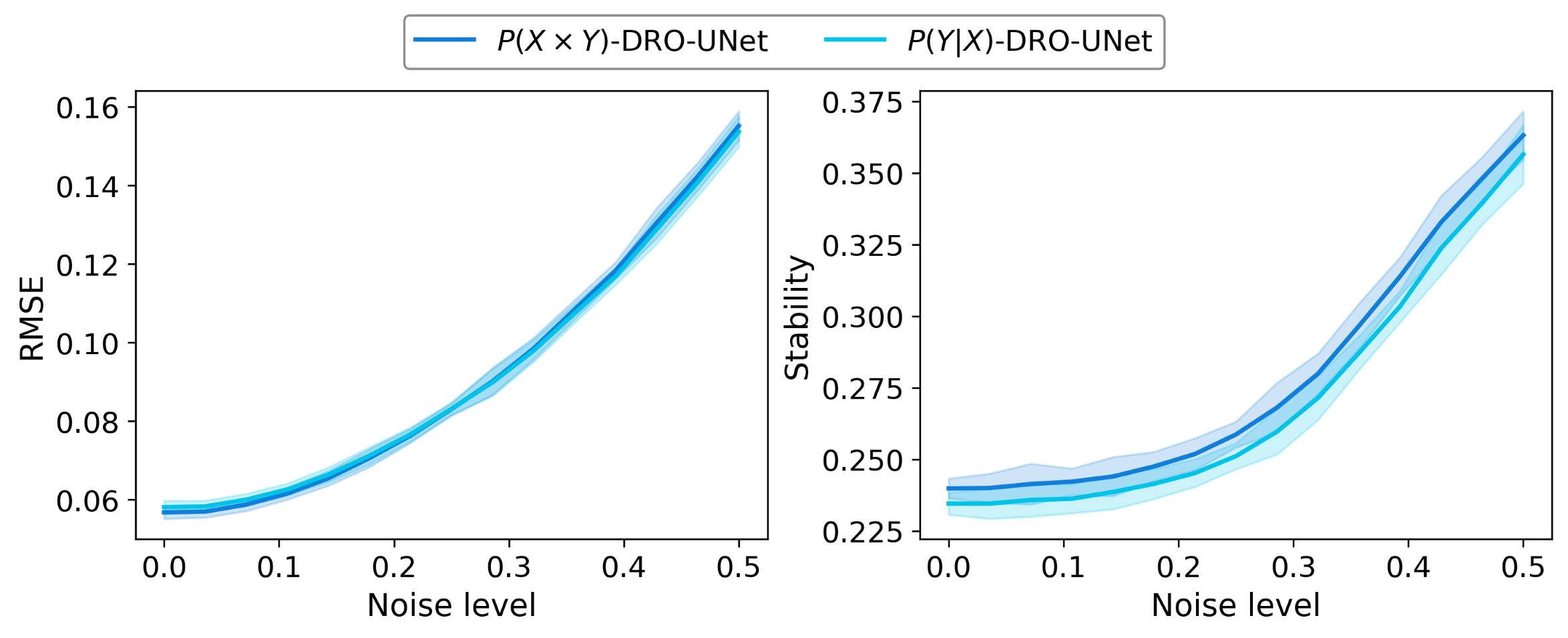}
    \caption{MNIST reconstruction metrics. Comparison of robustness in $X\times Y$ and $Y|X$.}
    \label{fig:mnist_joint_metrics}
\end{figure}
\section{Proofs for strong duality in Section~\ref{sec:duality}}
\subsection{General strong duality}\label{supp:general-duality-proof}
This section details the proof for Theorem \ref{thm:w-padro-strong-duality}. The strategy of the proof follows the one in \cite{blanchet_quantifying_2019}. First off, we need to define certain sets and functionals and introduce some lemmas, all in preparation for the application of the Fenchel Duality Theorem \cite[p.~201]{luenberger_optimization_1997}.
Let $\C:=C(S\times S)$ represent the vector space of continuous functions equipped with the supremum norm, and recognize its dual $\C^* = M(S\times S)$, the space of finite signed Borel measures on $S$ equipped with the total variation norm.

For simplicity, we recall that the primal problem is 
\begin{align*}
    I := \sup_{\pi_{S|S} \in \Okc} I(\pi_{S|S}, G)
\end{align*}
with the notatons given in Section \ref{sec:duality}.
Moreover, we define the dual form as $J$, such that
\begin{align*}
    J(\lambda, h) := \lambda \varepsilon + \sup_{\pi_{S|S} : S \rightarrow K_s} \int_{S \times S} h(s,r)\ d\pi_{S|S=s}(r)d\mu^*(s), 
    \qquad J:= \inf_{(\lambda, h) \in \Delta_{G,\ell,c}} J(\lambda,h),
\end{align*}
where we remind the reader that 
\begin{align*}
    \D := \bigg\{ (\lambda,h): & \, h\in \mathcal{C}, \lambda \geq 0,\ \int_S\int_S h(s,r) d\pi_{S|S=s}(r)d\mu^*(s) \\
    &\geq \  \int_S\int_S \ell(r, G) - \lambda c(s, r) d\pi_{S|S=s}(r)d\mu^*(s)   \quad \forall \pi_{S|S}: S\rightarrow K_s \bigg\}.
\end{align*}
Define the convex set of functions
\begin{align}\label{eq:B}
    B:= \bigg\{ g \in \C: \int_S\int_S &g(s,r)d\pi_{S|S=s}(r)d\mu^*(s) \\
    &\geq \int_S\int_S\ell(r,G) d\pi_{S|S=s}(r) d\mu^*(s), \ \forall \pi_{S|S}: S\rightarrow K_s \bigg\}\nonumber.
\end{align}
Define also the functional $\Theta: \mathcal{C} \to \R$ as
\begin{align*} 
    \Theta (g) &:= \inf_{(\lambda, h) \in D(g)} \left\{ \lambda \varepsilon + \sup_{\pi_{S|S}: S\rightarrow K_s} \int_S \int_S h(s,r) d\pi_{S|S=s}(r) d\mu^*(s) \right\},
\end{align*}
with 
\begin{align*} 
    D(g):= \bigg\{ (\lambda, h):  h\in \mathcal{C}, \, & \lambda \geq 0, \int_S \int_S g(s,r)d\pi_{S|S=s}(r)d\mu^*(s) \\
    &= \int_S \int_S h(s,r) + \lambda c(s,r) d\pi_{S|S=s}(r)d\mu^*(s), \ \forall \pi_{S|S}: S\rightarrow K_s \bigg\}.
\end{align*}
In the following lemma we verify that $\Theta$ is convex.

\begin{lemma}
    The function $\Theta : 
    \mathcal{C} \rightarrow \R$ is convex. 
\end{lemma}

\begin{proof}
    To show $\Theta$ is convex, first note that $J$ is convex. We can then rewrite 
\begin{align*}
    \Theta(g) = \inf_{(\lambda, h) \in D(g)} J(\lambda, h).
\end{align*}
Let $g_1,g_2 \in \mathcal{C}$ and $g = tg_1 + (1-t)g_2$ for $t \in (0,1)$. Let $\varepsilon >0$ be an arbitrary positive number and
$(\lambda_1, h_1) \in D(g_1)$ and $(\lambda_2, h_2) \in D(g_2)$ such that
\begin{align*}
    \Theta(g_1) \geq J(\lambda_1, h_1) - \varepsilon, \quad
    \Theta(g_2) \geq  J(\lambda_2, h_2) - \varepsilon.
\end{align*}
Now consider 
\begin{align*}
    \Theta(tg_1 + (1-t)g_2) = \inf_{(\lambda, h) \in D(tg_1 + (1-t)g_2)} J(\lambda, h).
\end{align*}
Note that
\begin{align*}
    (\lambda_3, h_3) := t(\lambda_1, h_1) + (1-t)(\lambda_2, h_2) \in D(tg_1 + (1-t)g_2).
\end{align*}
Therefore, by convexity of $J$,
\begin{align*}
    J(\lambda_3, h_3) \leq t J(\lambda_1, h_1) + (1-t)J(\lambda_2, h_2)
\end{align*}
and thus
\begin{align*}
    \Theta(tg_1 + (1-t)g_2) &= \inf_{(\lambda, h) \in D(tg_1 + (1-t)g_2)} J(\lambda, h)
    \leq J(\lambda_3, h_3) \\
     & \leq t J(\lambda_1, h_1) + (1-t)J(\lambda_2, h_2) 
     \leq t \Theta(g_1) + (1-t) \Theta(g_2) + 2\varepsilon.
\end{align*}
Since $\varepsilon$ is arbitrary the convexity of $\Theta$ follows.
\end{proof}

\begin{lemma}\label{lemma:fenchel-requirements}
    Let $S$ be a compact Polish space, and assume $c:S\times S\to\R_+$ is non-negative lower semi-continuous such that $c(s,r)=0$ if and only if $s=r$ and $\ell$ is continuous. Define  $B$ as in \eqref{eq:B}. Then, $B$ has non-empty interior.
\end{lemma}
\begin{proof}
Since $S$ is compact and
$\ell$ is continuous, there exists
\begin{align*}
M:=\sup_{r\in S}\ell(r,G)<+\infty .
\end{align*}
Define $k(s,r):=M+1$ for $ (s,r)\in S\times S$.
Then \(k\in C(S\times S)\). Moreover, for every $\pi_{S|S} :S \rightarrow K_s$ it holds that
\begin{align*}
\int_{S\times S} k(s,r)\,d\pi_{S|S=s}(r)d\mu^*(s)
=
M+1
\geq
\int_{S\times S}\ell(r,G)\,d\pi_{S|S=s}(r)d\mu^*(s),
\end{align*}
and therefore $k\in B$.
We claim that $k\in \operatorname{int}(B)$. Indeed, let
$\eta\in C(S\times S)$ be such that
\begin{align*}
\|\eta-k\|_\infty<\frac12 .
\end{align*}
Then, for every \((s,r)\in S\times S\),
\begin{align*}
\eta(s,r)
\geq
k(s,r)-\frac12
=
M+\frac12
\geq
\ell(r),
\end{align*}
so that $\eta \in B$ and thus $k\in\operatorname{int}(B)$.

\end{proof}

\begin{lemma}\label{lem:close}
    The set $\Omega_K$ defined as 
       \begin{align*}
  \Omega_K = \{\pi_{S|S} \otimes \mu^* : \pi_{S|S=s} \in K_s\}  
\end{align*}
    is weak* closed.
\end{lemma}
\begin{proof}
    Consider a sequence $\mu_n = \pi_{S|S}^n \otimes \mu^*$ such that $\mu_n \rightarrow \mu$ weakly*. Note that by weak* continuity of the projection, the first marginal of $\mu$ is $\mu^*$. Therefore, we can disintegrate $\mu$ as $\mu = \pi_{S|S} \otimes \mu^*$ for some $\pi_{S|S} : S \rightarrow P(S)$. It remains to prove that then $\pi_{S|S=s} \in K_s$ for $\mu^*$-a.e. $s\in S$.

    Define the support function of $K_s$ as $\sigma_{K_s} : C(S) \rightarrow \R$ with
    \begin{align*}
        \sigma_{K_s}(\varphi) = \sup_{\nu \in K_s} \int \varphi d\nu
    \end{align*}
    and note that since $K_s$ is convex and weak* closed it holds \cite[Theorem 7.51]{aliprantis_infinite_2006}
    \begin{align*}
        K_s=\left\{\mu\in P(S):\int_S \varphi\,d\mu\leq \sigma_{K_s}(\varphi)\text{ for every } \varphi \in C(S) \right\}.
    \end{align*}
    Therefore since $\pi^n_{S|S=s} \in K_s$ for $\mu^*$-a.e. $s\in S$ it also holds that for almost every $s \in S$ we have that $\int_S \varphi(r) d\pi^n_{S|S=s} (r) \leq \sigma_{K_s}(\varphi)$ for every $\varphi \in C(S)$. This implies, by passing to the limit for $n \rightarrow +\infty$ that for all non-negative $\eta \in C(S)$ and for all $\varphi \in C(S)$ it holds that 
\begin{align*}
\int_S \eta(s)
\left[
\int_S \varphi(r)\,d\pi_{S|S=s}(r)-\sigma_{K_s}(\varphi)
\right]d\mu^*(s)
\leq 0
\end{align*}
In particular, for all $\varphi \in C(S)$ 
\begin{align}\label{eq:ppl}
   \int_S \varphi(r)\,d\pi_{S|S=s}(r)
\leq \sigma_{K_s}(\varphi)
\end{align}
for $\mu^*-$a.e. $s\in S$.  Since $C(S)$ is separable ($S$ is supposed to be a compact Polish space) we can choose a countable dense subset of $C(S)$ denoted by $(\varphi_i)_{i \in I}$ and define by $A_i$ a set such that $\mu^*(A_i) = 1$ and \eqref{eq:ppl} holds for $\varphi = \varphi_i$. Considering $A = \bigcap_i A_i$ we thus have that $\mu^*(A) = 1$ and for every $s \in A$
\begin{align}\label{eq:ppl2}
   \int_S \varphi_i(r)\,d\pi_{S|S=s}(r)
\leq \sigma_{K_s}(\varphi_i)
\end{align}
for every $i \in I$. Since $(\varphi_i)$ is dense in $C(S)$, by a standard density argument we can conclude that for all $s\in A$ it holds that 
\begin{align}\label{eq:ppl3}
   \int_S \varphi(r)\,d\pi_{S|S=s}(r)
\leq \sigma_{K_s}(\varphi)
\end{align}
for every $\varphi \in C(S)$, implying that $\pi_{S|S=s}\in K_s$ for $\mu^*$-a.e. $s \in S$.
\end{proof}

\begin{lemma}\label{lemma:B-conjugate}
    Let $S$ be a compact Polish space, $\ell: S \to \R$ continuous and $B$ defined as in \eqref{eq:B}. If $\mu$ is not a positive measure, then
    \begin{align}
        \inf_{g \in B} \int_S\int_S g(s,r) \ d\mu = - \infty
    \end{align}
\end{lemma}

\begin{proof}
    Suppose that $\mu$ is not a positive measure. Then, there exists $g \in \mathcal{C}$ such that $g\geq 0$ and $\int g d\mu < 0$. Define then $g_n = ng + \sup_{r \in S} \ell(r,G)$ and note that $g_n \in B$ for every $n$ and 
\begin{align*}
\lim_{n \rightarrow + \infty} \int_S\int_S g_n(s,r) d\mu &  =  \lim_{n \rightarrow + \infty} \int_S\int_S ng + \sup_{r \in S} \ell(r,G) d\mu \\
& =  \lim_{n \rightarrow + \infty} n \int_S\int_S gd\mu + \sup_{r \in S} \ell(r,G) \mu(S) = -\infty
\end{align*}
as we wanted to prove.
\end{proof}
We are now ready to prove Theorem \ref{thm:w-padro-strong-duality}.

\paragraph{Proof of Theorem \ref{thm:w-padro-strong-duality}}
\begin{proof}

    Note first that due to the definition of $B$, $D(g)$ and $\Theta$ given above it holds that 
    \begin{align}\label{eq:initialest}
        J \leq \inf_{\substack{g \in \mathcal{C} \\ (\lambda, h) \in D(g)}} \bigg\{ & J(\lambda, h):  \int_{S\times S} h(s,r) + \lambda c(s,r)\ d\pi_{S|S=s}(r)d\mu^*(s) \nonumber \\
        &\geq \int_{S \times S} \ell(r,G)\ d\pi_{S|S=s}(r)d\mu^*(s),\ \forall \pi_{S|S}: S \rightarrow K_s \bigg\} =  \inf_{g \in  B} \Theta(g).
    \end{align}

    Next, we aim to identify the conjugate functional $\Theta^*$.
    To evaluate $\Theta^*$, note that for every $\mu \in M_+(S\times S)$
    \begin{align}\label{eq:chai}
        \Theta^*(\mu) &= \sup_{g \in \mathcal{C}} \bigg( \iint g d\mu - \Theta(g) \bigg) \nonumber\\
        &= \sup_{\substack{g \in \mathcal{C}}} \bigg(\iint g d\mu - \inf_{(\lambda, h) \in D(g)} \bigg(\lambda\varepsilon + \sup_{\pi_{S|S} : S \rightarrow K_s} \iint h\, 
        d\pi_{S|S=s} d\mu^* \bigg)\bigg)  \nonumber\\
        & = \sup_{\substack{g \in \mathcal{C}  \nonumber\\ (\lambda, h) \in D(g)}}
        \bigg(\iint g d\mu - \bigg( \lambda\varepsilon + \sup_{\pi_{S|S} : S \rightarrow K_s} \iint g - \lambda c, 
        d\pi_{S|S=s} d\mu^* \bigg) \bigg) \nonumber\\
        & \geq \sup_{\substack{g \in \mathcal{C}  \nonumber\\ (\lambda, h) \in D(g)}}
        \bigg(\iint g d\mu - \bigg( \lambda\varepsilon + \sup_{\pi_{S|S} : S \rightarrow K_s} \iint g d\pi_{S|S=s} d\mu^* \\
        &\qquad\qquad\qquad\qquad\qquad\qquad\qquad\qquad\qquad+ \sup_{\pi_{S|S} : S \rightarrow K_s} \iint - \lambda c d\pi_{S|S} d\mu^* \bigg)\bigg)\nonumber  \\
        & = \sup_{\substack{g \in \mathcal{C} \\ (\lambda, h) \in D(g)}} \bigg( \iint g d\mu - \sup_{\pi_{S|S} : S \rightarrow K_s} \iint g 
        d\pi_{S|S=s} d\mu^* -\lambda \varepsilon \\
        &\qquad\qquad\qquad\qquad\qquad\qquad\qquad\qquad\qquad+ \inf_{\pi_{S|S} : S \rightarrow K_s} \iint  \lambda c d\pi_{S|S=s} d\mu^* \bigg) \nonumber
        \end{align}
        Suppose that $\mu(S\times S) \neq 1$. Define $g_n = (\mu(S \times S) - 1) n \in \mathcal{C}$ for every $n$. Then 
        \begin{align}
            \lim_{n \rightarrow + \infty} \iint g_n d\mu - \sup_{\pi_{S|S} : S \rightarrow K_s} \iint g_n 
        d\pi_{S|S=s} d\mu^* = \lim_{n \rightarrow + \infty}n(\mu(S \times S) - 1)^2 = +\infty 
        \end{align}
implying that $\Theta^*(\mu) = \infty$ in case $\mu \notin P(S\times S)$. Suppose now then $\mu \in P(S\times S)$. Consider the set $\Omega_K = \{\mu^* \otimes \pi_{S|S} : \pi_{S|S}: S\rightarrow K_s\}$ and note that this is convex and weak* closed by Lemma \ref{lem:close}. Suppose that $\mu \notin \Omega_K$. Then, by Hahn-Banach separation theorem it holds that there exists $g \in \mathcal{C}$ such that
\begin{align}\label{eq:separation}
    \int g d\mu > 0 \quad \text{and} \quad \int g   d\nu \leq 0 \ \ \ \forall \nu \in \Omega_K.
\end{align}
Define $g_n(s,r) = ng(s,r)$. Then using \eqref{eq:separation} it holds that 
\begin{align*}
   \lim_n \iint g_n d\mu - \sup_{\pi_{S|S} : S \rightarrow K_s} \iint g_n 
        d\pi_{S|S=s} d\mu^*  \geq \lim_n n\int g(s,r) d\mu(s,r)  = + \infty.
\end{align*}
Therefore, using the previous limit in combination with \eqref{eq:chai} we obtain that $\Theta^*(\mu) = +\infty$ whenever $\mu \notin \Omega_K$.
Let now $\mu = \mu^* \otimes \pi_{S|S} \in \Omega_K$ Then 
\begin{align*}
    \Theta^*(\mu) & = \sup_{\substack{g \in \mathcal{C}}} \bigg(\iint g d\mu - \inf_{(\lambda, h) \in D(g)} \bigg(\lambda\varepsilon + \sup_{\pi_{S|S} : S \rightarrow K_s} \iint h\, 
        d\pi_{S|S=s} d\mu^* \bigg)\bigg) \\
        & = \sup_{\substack{g \in \mathcal{C} \\ (\lambda,h) \in D(g)} } \bigg(\iint g d\pi_{S|S=s} d\mu^* -\lambda \varepsilon - \sup_{\pi_{S|S} : S \rightarrow K_s} \iint h\, 
        d\pi_{S|S=s} d\mu^* \bigg)\\
        & =  \sup_{\substack{g\in\mathcal{C} \\ (\lambda,h) \in D(g)} } \bigg(\iint h + \lambda c d\pi_{S|S=s} d\mu^* -\lambda \varepsilon - \sup_{\pi_{S|S} : S \rightarrow K_s} \iint h\, 
        d\pi_{S|S=s} d\mu^* \bigg)\\
        & =  \sup_{\substack{g \in \mathcal{C} \\ (\lambda,h) \in D(g)} } \bigg(\lambda \bigg( \iint c\ d\mu - \varepsilon \bigg) + \iint h d\pi_{S|S=s} d\mu^* - \sup_{\pi_{S|S} : S \rightarrow K_s} \iint h\, 
        d\pi_{S|S=s} d\mu^* \bigg)\\
        &= \begin{cases}
            0 &\text{ if } \displaystyle \iint c\ d\pi_{S|S=s} d\mu^* \leq \varepsilon,  \\
            \infty &\text{ else.}
        \end{cases}
\end{align*}
So in particular for $\mu \in M_+(S\times S)$ then 
\begin{align*}
    \Theta^*(\mu) = \begin{cases}
            0 &\text{ if } \mu = \pi_{S|S} \otimes \mu^* \text{ for }  \pi_{S|S} : S \rightarrow K_s \text{ and }  \displaystyle \iint c\ d\pi_{S|S=s} d\mu^* \leq \varepsilon,  \\
            \infty &\text{ else.}
        \end{cases}
\end{align*}

By Fenchel duality (see e.g. \cite[Theorem 1, Page 201]{luenberger_optimization_1997}), thanks to Lemma \ref{lemma:fenchel-requirements}, it holds that 
\begin{align*}
    \inf_{g\in B} \Theta(g) = \sup_{\mu \in \mathcal{C}^*} - \Theta^*(\mu) + \inf_{g \in B} \int g \, d\mu.
\end{align*}
Now note that if $\mu = \mu^* \otimes \mu_{S|S}$ and $\ell$ is continuous it holds that
    \begin{align*}
        \inf_{g \in B} \iint g\ d\mu_{S|S=s} d\mu^* = \iint \ell\ d\mu_{S|S=s} d\mu^*.
    \end{align*}
Therefore by applying the previous computations and Lemma \ref{lemma:B-conjugate} we deduce that 

\begin{align*}
    \inf_{g\in B} \Theta(g) = \sup_{\mu_{S|S} \in \Okc} \iint \ell\ d\mu_{S|S=s}(r) d\mu^*(s) = I_{\Okc}(G)
\end{align*}
as we wanted to prove.
 Therefore from \eqref{eq:initialest} it follows that 
    \begin{align*}
    J \leq I.
    \end{align*}
Due to weak duality proven in Proposition \ref{prop:weakduality} we had $J \geq I$, therefore $J=I$ and we have strong duality.
\end{proof}


\subsection{Proof of Theorem \ref{th:k=p(s)}}\label{supp:p(s)-duality-proof}
\begin{proof}
Define
\begin{align}
\phifix (x,y) &:= \sup_{(u,v) \in S} (\ell(u,v,G) - \lambda c((x,y),(u,v))).
\end{align}
    Writing $s=(x,y)$ and $r=(u,v)$, observe that
    \begin{align*}
        \{(\lambda, h): \lambda \in \R_+, h(x,y,u,v) = \phifix(x,y)\} \subseteq \D
    \end{align*}
    and therefore 
    \begin{align*}
        \inf_{(\lambda, h) \in \D} \lambda \varepsilon + \sup_{\pi_{S|S} : S \rightarrow K_s} \int_{S \times S} h(s,r) d\pi_{S|S=s}(r)d\mu^*(s) \leq \inf_{\lambda \geq 0} \lambda \varepsilon +  \int_{S} \phifix(s) d\mu^*(s).
    \end{align*}
    For the reverse inequality, note that for every $(\lambda, h) \in \D$ we have
    \begin{align*}
        \int_{S\times S} h(s,r)\ d\pi_{S|S=s}(r)d\mu^*(s) \geq \int_{S\times S} \ell(r,G)-\lambda c(s,r)\ d\pi_{S|S=s}(r)d\mu^*(s).
    \end{align*}
    Hence, noticing that for the choice of $K_s = P(S)$ it holds that 
    \begin{align*}
    \{ \delta_{\varphi(s)} \otimes \mu^* : \varphi :S \rightarrow S \text{ measurable}\} \subset \{\pi_{S|S} : S \rightarrow K_s\}
    \end{align*}
    we have that
    \begin{align*}
        &\inf_{(\lambda, h) \in \D} \lambda \varepsilon + \sup_{\pi_{S|S} : S \rightarrow K_s} \int_{S \times S} h(s,r) d\pi_{S|S=s}(r) d\mu^*(s)  \\
        &\qquad\geq \inf_{\lambda \geq 0} \lambda \varepsilon +  \sup_{\varphi : S\rightarrow S} \int_S  \ell(\varphi(s),G) - \lambda c(s,\varphi(s))d\mu^*(s) \\
        &\qquad = \inf_{\lambda \geq 0} \lambda \varepsilon +  \int_S \sup_{\varphi : S\rightarrow S} \ell( \varphi(s),G) - \lambda c(s, \varphi(s)) d\mu^*(s)  \\
        &\qquad= \inf_{\lambda \geq 0} \lambda \varepsilon + \int_S \phifix (x,y) d\mu^*(x,y).
    \end{align*}
    as we wanted to prove,
    where in the first equality we used \cite[Theorem 3A]{rockafellar_integral_1976}.
\end{proof}

\subsection{Proof of Theorem \ref{th:k=x}}\label{supp:p(x)-duality-proof}
\begin{proof}
Define 
\begin{align*}
\phifix (x,y) := \sup_{u \in X} \left\{\int_Y \ell (u, v, G) - \lambda c((x,y),(u,v))\ d\mu_{Y|X=u}^*(v) \right\}.
\end{align*}
Since $K_{(x,y)} = K = \{\mu_X \otimes \mu^*_{Y|X}: \mu_X \in P(X)\}$ we have that for every $\pi_{S|S} : S \rightarrow K_s$ 
\begin{align*}
    \int_S \phifix (x,y) \, & d\pi_{S|S=s}(x,y) \,d\mu^*(x,y)  = \int_S \phifix (x,y)\,d\mu^*(x,y) \\
    & = \int_S \sup_{u \in X} \left\{\int_Y \ell (u, v, G) - \lambda c((x,y),(u,v))\ d\mu_{Y|X=u}^*(v) \right\} \,d\mu^*(x,y)\\
    & \geq \int_S \int_X\int_Y \ell (u, v, G) - \lambda c((x,y),(u,v))\ d\mu_{Y|X=u}^*(v)  d\mu_X(u)\,d\mu^*(x,y)\\
    & = \int_S  \ell (u, v, G) - \lambda c((x,y),(u,v)) d\pi_{S|S=s}(x,y) \,d\mu^*(x,y)
\end{align*}
so that 
\begin{align*}
        \left\{(\lambda, h): h(x,y,u,v) = \phifix(x,y) \right\} \subseteq \D.
    \end{align*}
    Therefore
    \begin{align}
        \inf_{(\lambda, h) \in \D} &\lambda \varepsilon + \sup_{\pi_{S|S} : S \rightarrow K_s} \int_{S \times S}  h(s,r) \ d\pi_{S|S=s}(r)d\mu^*(s) \nonumber \\
        &\leq  \inf_{(\lambda, h) \in \D} \lambda \varepsilon +  \int_S  \phifix (x,y) d\mu^*(x,y).  \label{eq:kx-lhs}
    \end{align}
    For the opposite inequality, for all $(\lambda, h) \in \D$ we have
    \begin{align*}
        \int h(s,r)\ d\pi_{S|S=s}(r) d\mu^*(s) \geq \int \ell(r,G) - \lambda c(s,r) \,  d\pi_{S|S=s}(r) d\mu^*(s), \quad \forall \pi_{S|S} : S \rightarrow K_s.
    \end{align*}
    Hence noticing that 
    \begin{align*}
        \{\delta_{\varphi(x,y)} \otimes \mu^*_{Y|X}   : \varphi: S \rightarrow X\} \subset \{\pi_{S|S} : S \rightarrow K_s\}
    \end{align*}
   it holds that 
   \begin{align*}
        &\inf_{(\lambda, h) \in \D} \lambda \varepsilon + \sup_{\pi_{S|S} : S \rightarrow K_s} \int_{S \times S} h(s,r) d\pi_{S|S=s} (r)d\mu^*(s)  \\
        &\qquad\geq \inf_{\lambda \geq 0} \lambda \varepsilon +  \sup_{\mu_X \in P(X)} \int_S  \int_S \ell(u,v,G) - \lambda c((x,y),(u,v)) d \mu_{Y|X=u}(v) d\mu_X(u) d\mu^*(x,y) \\
        &\qquad\geq \inf_{\lambda \geq 0} \lambda \varepsilon +  \sup_{\varphi : S \rightarrow X} \int_S  \int_Y \ell(u,v,G) - \lambda c((x,y),(\varphi(x,y),v)) d \mu_{Y|X=\varphi(x,y)}(v) d\mu^*(x,y)\\
        &\qquad = \lambda \varepsilon +   \sup_{\varphi : S \rightarrow X} \int_Y \ell(u,v,G) - \lambda c((x,y),(\varphi(x,y),v)) d \mu_{Y|X=\varphi(x,y)}(v) d\mu^*(x,y)  \\
        &\qquad= \lambda \varepsilon +  \int_S  \sup_{u \in X} \int_Y \ell(u,v,G) - \lambda c((x,y),(u,v)) d \mu_{Y|X=u}(v) d\mu^*(x,y) 
    \end{align*}
    as we wanted to prove, where in the first equality we used \cite[Theorem 3A]{rockafellar_integral_1976}.
\end{proof}
\section{Proof of Corollary \ref{cor:w-cond}}\label{supp:w-cond}
\label{sec:app-cond-wass}
\begin{proof}
    We first show (i). By definition of the cost $c_{Y|X}$, any transport plan $\pi$ with finite cost must be supported on the set $\{(x,y,u,v) : x = u\}$. Thus its marginals on $X$ must coincide, which implies $\mu_X^* = \mu_X$. Conversely, if $\mu_X^* = \mu_X$ and since $X,Y$ are compact admissible plans with finite cost exist.
    Assume in the rest of the proof that $\mu_X^* = \mu_X$.  

    We now prove the inequality 
    \begin{align*}
         W_\text{cond}(\mu^*, \mu) \geq W_{c_{Y|X}}(\mu^*, \mu).
    \end{align*}
    In the following, we will write  $\pi_x$ to mean $\pi_{(Y\times Y)|X=x}$ and $\delta_x$ to mean $\delta_{X|X=x}$  for readability. For each $x \in X$, let $\pi_x$ be an optimal transport plan between $\mu_{Y|X=x}^*$ and $\mu_{Y|X=x}$. Define a coupling $\pi$ on $(X\times Y)^2$ by $\pi(x,y,u,v):=\mu_X^*(x)\otimes [\pi_x(y, v)\delta_x(u)]$. By construction, $\pi$ has marginals $\mu^*$ and $ \mu$ and is thus admissible. Using the definition of $c_{Y|X}$, we obtain
    \begin{align*}
        \int_{(X\times Y)^2} c_{Y|X}((x,y),(u,v))\ d\pi((x,y),(u,v)) &= \int_X \int_{Y \times Y} c_Y(y,v)\ d\pi_x(y,v)\ d\mu_X^*(x) \\
        &= W_\text{cond}(\mu^*, \mu)
    \end{align*}
    where $c_Y(y,v):= \|y-v\|^2$ is the cost function only on $Y$. Taking the infimum over all admissible plans yields the desired inequality.  

    For the reverse inequality, let $\pi$ be an optimal transport plan for $W_{c_{Y|X}}$. Since finite cost requires $x=u$, the plan $\pi$ is concentrated on $\{x=u\}$. Disintegrating $\pi$ with respect to $x$, we may write $\pi(x,y,u,v) =  \mu_X^*(x)\otimes[\pi_x(y,v)\delta_x(u)]$, where for $\mu_X^*$-a.e. $x$, $\pi_x$ is a coupling of $\mu^*_{Y|X=x}$ and $\mu_{Y|X=x}$. Then,

    \begin{align*}
        \int_{(X\times Y)^2} c_{Y|X}((x,y),(u,v))\ d\pi((x,y),(u,v)) &= \int_X \int_{Y \times Y} c_Y(y,v)\ d\pi_x(y,v)\ d\mu_X^*(x) \\
        &\geq W_\text{cond}(\mu^*, \mu)
    \end{align*}
    thus proving the claim.

\end{proof}

\end{document}